\documentclass{amsart}
\usepackage{amsmath}
\usepackage{amssymb}
\usepackage{amsfonts}
\usepackage{mathrsfs}

\usepackage{amscd}

\newtheorem{thm}{Theorem}[section]
\newtheorem{cor}[thm]{Corollary}
\newtheorem{prop}[thm]{Proposition}
\newtheorem{cl}[thm]{Claim}

\newtheorem{fact}[thm]{Fact}

\begin{document}

 \title{  A negative answer to the problem: are stratifiable spaces $M_1$?}
 \author{Huaipeng Chen$^*$ and Bosen Wang$^{***}$  }

  \address {$^*$School of Mathematics and computer science, Shaanxi SCI-TECH University, Hanzhong,
    723000, Shaanxi P. R. China}
\email{chenhp@snut.edu.cn}

 \address {$^{***}$No.11 Jin Rong Avenue, Xi Cheng District, Beijing, P.R.China. 100033}
\email{wangbosen@pku.edu.cn}

 \thanks{$^*$ the author is supported by  The National Natural Science
            Foundation of China (No.11171162).
            \\
  keywords:{ stratifiable spaces, $M_1$-spaces, metric spaces,
      $g$-functions, unbounded sets}
 \\
  subjclass[2010]{ Primary: 54G20. Secondary: 04A20, 54E35, 54E20, 54E99}}

\begin{abstract}
 In accordance with  $M_3$-structures in paper [4],
  we construct  a stratifiable space
  which is not $M_1$-spaces.
\end{abstract}

\maketitle

\section{Introduction}
``Are stratifiable spaces $M_1$?'' is a well-known
   problem in set-theoretic topology. It comes from \cite{c}.
   Ceder \cite{c} defined $M_i$-spaces
   ($i= 1, 2, 3$) and proved $M_1 \Rightarrow M_2
   \Rightarrow M_3$.
   It is an interesting
   problem  whether
   these implications can be reversed.
   Borges  \cite{b} gave a characterization of
   $M_3$-spaces and renamed $M_3$-spaces as
   stratifiable spaces.
    Gruenhage  \cite{g} and Junnila \cite{j}
   proved that stratifiable spaces are
   $M_2$-spaces independently.
   Their results aroused people's great interest to
   the problem ``Are stratifiable spaces $M_1$?''.
   It\={o} and Tamano \cite{i}
   using closed mappings  got interesting results.
   T.Mizokami got some important progresses on the problem
   in \cite{tm},\cite{tmi} and \cite{tmiz}.
   Also there are many important results on
   stratifiable spaces commended by
   surveys of Tamano \cite{t}, Gruenhage \cite{gr} and
   \cite{gru}, Burke and Luter \cite{bl}.
   Wang made some comment about the problem
   and called it ``Problem on Generalized Metric Spaces"
   in \cite{W}.
   In 1990,  Rudin made some comment about the problem
   in her well-known paper ``Some Conjectures" in \cite{Ru}.
\par

   In 2000, Gartside and Reznichenko \cite{GR} gave out $C_k(P)$-spaces
   which was commended by Gruenhage \cite{grue},
   and was researched in wide range.
   In 2008, Chen \cite{ch}
   got a new characterization of stratifiable spaces.
   In this paper, we use an idea of Chen \cite{ch} wholly
   to prove the following main theorem:
\par

  Theorem 1. There exists a space $(X, \tau)$ which is a stratifiable space, and is not an $M_1$-space.
\par

   Then, by section 8 in \cite{t}, Theorem 1 gives a negative answer to the following questions:
\par

-Is every $M_3$-space an $M_1$-space (Ceder \cite{c} 1961)? .
\par

-Does any point in a stratifiable space
  have a $\sigma$-closure-preserving base (Tamano \cite{t} 1989)?
\par

-Is every (closed) subspace of an $M_1$-space
  an $M_1$-space (Ceder \cite{c} 1961)?
\par

-Is the closed image of an $M_1$-space
  an $M_1$-space (Ceder \cite{c} 1961)?
\par

-Is the perfect image of an $M_1$-space
  an $M_1$-space (Burke and Lutzer \cite{bl} 1976)?
\par

-Is every stratifiable spaces a $\mu$-space (Tamano \cite{t} 1985)?
\par

-Is each zero-dimension submetric stratifiable space $X$ with an
 $M_3$-structure an $M_1$-space (Chen \cite{ch} 2008)?
\par

Theorem 1 suggest some questions too in this paper.
\par

   Recall that a space $X$ is an
   $M_1$-\textit{space} if $X$ has a
   $\sigma$-closure preserving base $\mathscr{B}$.
   Recall that a family $\mathscr{B}$ is a
   \textit{quasi-base} for $X$ if for each open
   set $U$ of $X$ and a point $x\in U$, there is
   $B\in \mathscr{B}$ such that $x\in Int B \subset B
   \subset U$.
   A space $X$ is an
   $M_2$-\textit{space} if $X$ has a
   $\sigma$-closure preserving \textit{quasi-base}
   and an $M_3$-\textit{space} if $X$ has a
   $\sigma$-cushioned pair-base.
\par

In this paper,
   the letter $N$ denotes the set of positive integers
   and $\omega$ denotes the first infinite ordinal.
   $h,i, j, k, l,\ell, m,n,a,b,c,d$  and $e$ are used to denote
   members in $\omega$ and $N$.
If there are signs and definitions which have not been
   defined in this paper, we can see it in \cite{gr}
   or \cite{t} in topology and in \cite{ku} in set
   theory.

\section{To construct a set $X$}

Following only the idea of Theorem 5.1 in \cite{ch},
  in order to construct a zero-dimension metric
  space $(X,\rho)$, we construct the set $X$ at firstly.
  To do it let $Q=\{0,1/2,1/3,2/3,...\}=\{p_n: n\in N\}$
  be the set of all rational numbers in $[0,1)$
  and $[p_n, p_m)=\{q\in Q: p_n\leq q< p_m \}$
  for $p_n, p_m\in Q$.
  Then $Q=[0, 1)$.
  Let $\mathscr{Q}=\{[p_n, p_m): p_n, p_m\in Q\hbox{ and }q_n<q_m\}$.
  Then $\mathscr{Q}$ is a base of some topology
  such that each $[p_n, p_m)$ is a closed and open set.
  Denote the topology by $\rho'$.
  Then $(Q,\rho')$ is a zero-dimensional metric space since
  $(Q,\rho')$ is regular $T_1$ and $\mathscr{Q}
  =\{B_n:n\in N\}$ is a countable base.

\vspace{0.5cm}

\textbf{Construction 1.}

\vspace{0.5cm}

Let  $(Q,\rho')$ be the zero-dimensional metric space.
\par

 A).  Let $S_0=\{q_0\}=\{0\}\subset Q$ and
  $S_1(q_0)=\{q_{1i}:i\geq1\}\subset Q-\{q_0\}$
  be a convergence sequence which converges
  to $q_0$ with \[q_0<...<q_{1i+1}<q_{1i}<...<q_{12}<q_{11}<1
  \quad\hbox{ and }\quad [q_0, q_{1i})=\{q\in Q: q_0\leq q<q_{1i}\}.\]
  Here $q_{11}\in Q-\{q_0\}$ is the first number in $Q-\{q_0\}$ with
  $1-q_{11}=|1-q_{11}|\leq 1/2$.
  Denote $q_{1i}$ by $q^1_i$. Let $S_1=\{q^1_i:i\in N\}$.
  Then $S_1=S_1(q_0)$.
\par
B).
  Assume we have had a convergence sequences set
  $S_n=\{q^n_i:i\in N\}\subset Q$.
  Pick a $q^n_i$ from $S_n$. And then pick a $q_{i0}$
  such that:
\par
  (1). If $q^n_i\neq\hbox{max}S_n$, let
  $q_{i0}=\hbox{inf}\{q\in\cup_{i\leq n}S_i: q>q^n_i\}$.
  Then $q_{i0}\in\cup_{i\leq n}S_i$
  with $(q^n_i,q_{i0})\cap(\cup_{i\leq n}S_i)=\emptyset$.
\par
   (2). If $q^n_i=\hbox{max}S_n$,
  let $q_{i0}=1$. Then $(q^n_i,q_{i0})\cap(\cup_{i\leq n}S_i)=\emptyset$.
\par
  Denote $n+1$ by $k$.
  Take a convergence sequence
  \[S(q^n_i)=S_k(q^n_i)=\{q_{ij}:j\in N\}\subset (q^n_i,q_{i0})\cap[Q-(\cup_{i\leq n}S_i)]\]
  which converges to $q^n_i$ and
  $q^n_i<...<q_{ij}<...<q_{i2}<q_{i1}<q_{i0}$.
  Here $q_{i1}$ is the first number in
  $(q^n_i,q_{i0})\cap[Q-(\cup_{i\leq n}S_i)]$ with
  $q_{i0}-q_{i1}=|q_{i1}-q_{i0}|\leq 1/2^k$.
  Let $S_k=\cup\{S_k(q^n_i):q^n_i\in S_n\}
  =\{q_{ij}: i,j\in N\}$ with
  $S_k(q^n_i)=\{q_{i1},q_{i2},q_{i3},...,q_{in},...\}$
  for $i=1,2,3,...,n,...$.
  Then $S_k$ is countable. We numerate points of $S_k$
  in according with the following order:
   \begin{eqnarray}
  q_{11},\quad\rightarrow\quad  q_{12},\quad \quad\quad
  q_{13}, \quad\quad\quad q_{14},\quad...\quad,q_{1n},...\nonumber\\
 \swarrow \quad\quad  \ \quad\quad  \swarrow \ \ \quad  \quad\quad
  \swarrow \quad\quad  \quad ...  \quad\quad\quad\quad \nonumber\\
  q_{21},\quad\quad\quad  q_{22},\quad \quad\quad
  q_{23}, \quad\quad\quad q_{24},\quad...\quad,q_{2n},...\nonumber\\
  \quad \swarrow \quad\quad \quad\quad  \swarrow\  \quad  \quad\quad
  ...\quad\quad \quad\quad\quad \quad \quad\quad\quad \nonumber\\
   q_{31},\quad\quad\quad  q_{32},\quad \quad\quad
  q_{33}, \quad\quad\quad q_{34},\quad...\quad,q_{3n},...\nonumber\\
  \swarrow \quad\quad \quad ...  \quad  \quad\quad  \quad\quad
   \quad\quad\quad \quad \quad \quad\quad\quad \quad\nonumber \\
   q_{41},\quad\quad\quad  q_{42},\quad \quad\quad
  q_{43}, \quad\quad\quad q_{44},\quad...\quad,q_{4n},...\nonumber
 \end{eqnarray}
 \[............................................
 .......................................\]
 \[\quad \quad \hbox{ Figure 0 }\quad\]

Then $S_k=\{q_{11}, q_{12}, q_{21}, q_{13}, q_{22}, q_{31},
  ...,q_{1n}, q_{2n-1}, q_{3n-2}, ..., q_{n1},...\}$.
  Let $q^k_l=q_{ni}$. Then $l=1+2+...+(n+i-2)+n=(n+i-2)(n+i-1)/2 + n$.
  Then $S_k=\cup\{S_k(q^n_i): q^n_i\in S_n\}=\{q^k_l:l\in N\}$.
  We call it {\em $\Delta$-order} for convenience.
  Then, by induction, we may construct
  convergence sequences sets
  $S_0,S_1, ..., S_n,... $ such that
  $Q=\cup_{n\geq 0}S_n$ and $S_n\cap S_m=\emptyset$ if $n\neq m$.
  Enumerate points of $Q=\cup_{n\geq 0}S_n$ in accordance with
  $\Delta$-order.
  Then $Q=\{q_n: n\in \omega\}$.
  We will always use the following symbol throughout this paper.
  $Q$ and $S_0,S_1, ..., S_n,... $ means sets with the $\Delta$-order.
  $[q_i, q_{ij})$ means an interval with the end points $q_i$ and $q_{ij}$
  such that $q_{ij}\in S_n(q_i)\subset S_n$ and $S_n(q_i)$ converges to
  $q_i$. Here $S_n(q_i)$ means $q_i=q^{n-1}_i\in S_{n-1}$
  throughout this paper.

 \vspace{0.3cm}
Let $q=q_\ell\in Q=\cup_{n\geq 0}S_n$ and $q=q^m_{\ell'}\in S_m$.
  Then $\ell>\ell'$ by the $\Delta$-order.
\vspace{0.3cm}

\begin{prop} Let $Q=\{q_n: n\in \omega\}$ with $\Delta$-order. Then:
\par

1.  Let $q_h\in Q$ with $q_h=q^k_l\in S_k$ and $q^k_l=q_{ni}\in S_k(q_n)$.
  Then \[l=1+2+...+(n+i-2)+n=(n+i-2)(n+i-1)/2 + n \quad \hbox{ and}\]
  \[h=1+1+2+...+(k+l-2)+k>l.\]

2. If $q_{ki},q_{kj}$ in $S_n(q_k )$ with $i<j$,
  then $q_{ki}=q^n_{n_i},q_{kj}=q^n_{n_j}$ in $S_n$ with
  $\Delta$-order satisfy
  $n_i<n_j$ and $q_{ki}=q_{h_i},q_{kj}=q_{h_j}$ in
  $Q$ with $\Delta$-order satisfy $h_i<h_j$.
\par

3. For each $q_h\in Q$ with $q_h=q^k_{k_i}\in S_k$,
   there uniquely exists a $q_n\in Q$ with $q_n\in S_{k-1}$
   such that $q_h=q_{ni}\in S_k(q_n)$ and $h>n$ in $Q$.
\par

4. Let $q_{ki}\in S_n(q_k)$ and $q_{hj}\in S_m(q_h)$.
  Then $[q_k,q_{ki})\cap[q_h,q_{hj})=\emptyset$
  or $[q_k,q_{ki})\subset[q_h,q_{hj})$ or
  $[q_k,q_{ki})\supset[q_h,q_{hj})$.
  Here $q_k=q^{n-1}_k\in S_{n-1}$ and $q_h=q^{m-1}_h\in S_{m-1}$.
\par

5. $Q=\cup_{n\geq 0}S_n$ and $S_n\cap S_m=\emptyset$ if $n\neq m$.
\end{prop}

\begin{proof}
 To prove 1 let  $S^m_k(q_a)=\{q_{ai}\in S_k(q_a): i\leq m\}$ with
 $q_a=q^{k-1}_a\in S_{k-1}$
  and $a=1,2,...,n+i-2$, and  let $S^n_k=\{q^k_j\in S_k: j\leq n\}$.
  Take $q^k_l\in S_k$ and assume $q^k_l=q_{ni}\in S_k(q_n)$.
  Then we have $q^k_l=q_{ni}\in S^i_k(q_n)$,
  \[S^l_k=S^{n+i-1}_k(q_1)\cup...\cup S^{i+1}_k(q_{n-1})
  \cup S^i_k(q_n)\cup S^{i-2}_k(q_{n+1})\cup...
  \cup S^1_k(q_{n+i-2})\]
  and $l=1+2+...+(n+i-2)+n$ by the definition of $\Delta$-order in accordance with the above Figure 0. Take $q_h\in Q=\{q_n: n\in \omega\}$ with $\Delta$-order.
  Let $q_h=q^k_l\in S_k$ and $q^k_l=q_{ni}\in S_k(q_n)$, and
  let $A_h=\{q_i\in Q: 0\leq i\leq h\}$.
  Note that $S_0=\{q_0\}$ is the first line by the $\Delta$-order of $Q$.
  Then $q_h=q^k_l\in S^l_k$ and
  \[A_h=S_0\cup S^{k+l-1}_1\cup...\cup S^{l+1}_{k-1}
  \cup S^l_k\cup S^{l-2}_{k+1}\cup...\cup S^1_{k+l-2}.\]
  Then $h=1+1+2+...+(k+l-2)+k$ with $l=1+2+...+(n+i-2)+n$.

 \textit{2 of the proof.} Pick $q_{ni}$ and $q_{nj}$ from $S_k(q^{k-1}_n )$ with $i<j=i+1$.
  Then we have $q_{ni}=q^k_{k_i}\in S_k$ and $q_{ni}=q^k_{k_i}=q_{h_i}\in Q$.
  Then, by the above 1,
  \[h_i=1+1+2+...+(k+k_i-2)+k\quad\hbox{and}\quad k_i=1+2+...+(n+i-2)+n.\]
  Note  $q_{nj}=q^k_{k_j}\in S_k$ and $q_{nj}=q^k_{k_j}=q_{h_j}\in Q$ with $j=i+1$.
  Then, by the above 1,
  \[h_j=1+1+2+...+(k+k_j-2)+k\quad\hbox{and}\quad k_j=1+2+...+(n+j-2)+n.\]
  Note $j=i+1$. Then $k_j=1+2+...+(n+i-2)+(n+i-1)+n>k_i$ and $h_j>h_i$.
\par

 \textit{3 of the proof.} Pick a $q_h\in Q$ with $q_h=q^k_i\in S^i_k$.
  Then, by $\Delta$-order, there uniquely exists $q^{k-1}_n\in S_{k-1}$
  with $q_h\in S_k(q^{k-1}_n)$.
  Let $q^{k-1}_n=q_{m_n}\in Q$ for $q^{k-1}_n\in S^n_{k-1}$.
  Note that $S_0=\{q_0\}$ is the first line by the $\Delta$-order in $Q$.
  Then we have $A_{m_n}=S_0\cup S^{n+k-2}_1\cup...\cup S^n_{k-1}
   \cup S^{n-2}_k\cup...\cup S^1_{n+k-3}$ with $q^{k-1}_n\in S^n_{k-1}$ and
   \[m_n=1+1+2+...+(n+k-3)+(k-1).\]
   When $q_h\in Q$ with $q_h=q^k_i\in S^i_k$, then $h=1+1+2+...+(k+i-2)+k$.
  Note $q^k_i=q_{n\ell}\in S^\ell_k(q_n)$. Then $i=1+2+...+(n+\ell-2)+n>n$.
  Then $h=1+1+2+...+(k+i-2)+k>1+1+2+...+(k+n-2)+k>m_n$. This implies 3.
\par

To check 4, take $q_{ki}\in S_n(q_k)$ and $q_{hj}\in S_m(q_h)$
  with $[q_k,q_{ki})\cap[q_h,q_{hj})\neq\emptyset$.
\par

Case 1, $q_k=q_h$. Then $i>j$ implies
  $[q_h,q_{hj})=[q_k,q_{kj})\supset[q_k,q_{ki})$,
  $i<j$ implies
  $[q_h,q_{hj})=[q_k,q_{kj})\subset[q_k,q_{ki})$
  and $i=j$ implies
  $[q_h,q_{hj})=[q_k,q_{kj})=[q_k,q_{ki})$.
\par

Case 2, $q_k\in(q_h,q_{hj})$. Then $q_k\in [q_{hl+1},q_{hl})$
 for some $l\geq j$ with $q_{hl+1}$ and $q_{hl}$ in $S_m(q_h)$.
 Then $[q_k,q_{ki})\subset[q_k,q_{k1})\subset(q_h,q_{hj})$
 by the definition of $(q_h,q_{hj})$.
\par

Case 3, $q_h\in(q_k,q_{ki})$.
 Then $[q_h,q_{hj})\subset[q_h,q_{h1})\subset(q_k,q_{ki})$
 in the same way as    Case 2.
\par

It is easy to see 5 from Construction 1.
\end{proof}

In order to construct $g$-functions of stratifiable spaces, we
construct families of closed and open intervals in $(Q, \rho')$.

\vspace{0.5cm}
\textbf{Construction 2.}
\vspace{0.5cm}

Call a family $\mathscr{I}'$ \textbf{D in $(X,\tau)$}
  if $\mathscr{I}'$ is  discrete in $(X,\tau)$,
  and \textbf{P.D in $(X,\tau)$}
  if $\mathscr{I}'$ is pairwise disjoint $(X,\tau)$.

Call a set $B$ \textbf{c.o in $(X, \tau)$}
  if $B$ is a closed and open set in $(X, \tau)$.
  Call a family $\mathscr{I}'$ \textbf{c.o.D in $(X, \tau)$}
  if $\mathscr{I}'$ is D in $(X, \tau)$ and each $B\in \mathscr{I}'$
  is c.o in $(X, \tau)$.
\par

Take $Q$ with $\Delta$-order and pick an $a\in N$ with $a>1$
  since we always use $a=n+i\geq 2$ for $H(n,i)$ after Section 3.
\par
1. Take $q_0=0\in Q$. Let $Q_{a0}=\{q_0\}$ and
 $\mathscr{I}_{a0}=\{[q_0,q_{1a})\}$.
 Here $q_{1a}\in S_1(q_0)$.
\par

 Let $Q_{a1}=\{q^1_i\in S_1:q_0\in Q_{a0} \hbox{  with  }
  q^1_i=q_{1i}\in S_1(q_0) \hbox{ for } 1\leq i\leq a\}$.
  Then we have
  $Q_{a1}=\{q^1_j\in S_1: 1\leq j\leq a\}$  and
  $\mathscr{I}_{a1}=\{[q_i,q_{ia}): q_i\in Q_{a1}\}$.
  Then $\mathscr{I}_{a1}$ is c.o.D in $(Q, \rho')$
  and $(\cup\mathscr{I}_{a0})\cap(\cup\mathscr{I}_{a1})=\emptyset$.

\par
2. Assume we have had $Q_{ak}=\{q^k_{k_j}\in S_k: 1\leq j\leq a^k\}
  =\{q_i\in S_k: i\leq a^k\}$
  and $\mathscr{I}_{ak}=\{[q_i,q_{ia}): q_i\in Q_{ak}\}$ for $k\leq l$
  such that:
\par
  (1). $\mathscr{I}_{ak}$ is c.o.D in $(Q, \rho')$ for each $k\leq l$.
\par
  (2). $(\cup_{j<k}\cup\mathscr{I}_{aj})\cap(\cup\mathscr{I}_{ak})
   =\emptyset$ for $k\leq l$.
\par
  Pick a $q_i\in Q_{al}$. Then $q_i=q^l_{l_i}\in S_l$.
  Take $S_{l+1}(q_i)\subset S_{l+1}$. Denote $l+1$ by $\ell$.
  Let \[Q_{a\ell}=\{q_{ij}\in S_{\ell}: q_i\in Q_{al} \hbox{  with  }
  q_{ij}\in S_{\ell}(q_i) \hbox{ for } 1\leq j\leq a\}.\]
  Then we have
  $Q_{a\ell}=\{q^\ell_{\ell_j}\in S_{\ell}: 1\leq j\leq a^\ell\}
  =\{q_i\in S_{\ell}: 1\leq i\leq a^\ell\}$.
  Let \[\mathscr{I}_{a\ell}=\{[q_i,q_{ia}): q_i\in Q_{a\ell}\}.\]
  Then $\mathscr{I}_{a\ell}$ is c.o.D in $(Q, \rho')$.
  Note $q_i\in Q_{a\ell}$ implies
  $[q_i,q_{ia})\cap(\cup_{k\leq l}\cup\mathscr{I}_{ak})=\emptyset$.
\par
Then, by induction on $l$, we have $Q_{al}$ and $\mathscr{I}_{al}$
  for $l\in N$ such that (1) and (2) in 2 for each $l\in N$.
  Let $Q^*_a=\cup_lQ_{al}$,
  $\mathscr{I}_a=\cup_l\mathscr{I}_{al}$ and
  $\mathscr{I}=\cup_{a\in N}\mathscr{I}_a$ .
\par

\begin{prop}
 1. $\mathscr{I}_a=\cup_l\mathscr{I}_{al}$ is c.o.D in $(Q,
\rho')$
  and $\cup\mathscr{I}_a=[0, 1)$ for each $a\in N$.
\par

2. $Q^*_a\subset Q^*_b$ if $a<b$.
\par

3. If $b>a$, then $\cup\{[q_j,q_{jb})\in\mathscr{I}_b:
  [q_j,q_{jb})\subset[q_i,q_{ia})\}
  =[q_i,q_{ia})$ for each $[q_i,q_{ia})\in\mathscr{I}_a$,
  and $Q^*_b\cap (q_j,q_{jb})=\emptyset$ for $q_j\in Q^*_b$.
\par

4. $\cup_aQ^*_a=Q$.
\par

5. Let $[q_i,q_{ia}),[q_j,q_{jb})\in\mathscr{I}$ with $b>a$.
  Then $[q_i,q_{ia})\cap[q_j,q_{jb})=\emptyset$ or
  $[q_j,q_{jb})\subset[q_i,q_{ia})$.
\par

\end{prop}

\begin{proof}
To prove 1, pick a $q\in [0, 1)=Q=\cup_nS_n$.
  Then $q\in S_n$ for some $n\in N$.
  Then $q=q_{li}\in S_n(q_l)$ for some $q_l\in S_{n-1}$.
  If $i>a$, then $q=q_{li}\in[q_l,q_{la})\in\mathscr{I}_{an}$.
  If $a\geq i>1$,
  then $q_{li}\in Q_{an+1}\subset\cup\mathscr{I}_{an+1}$.
  This implies $\cup\mathscr{I}_a=[0, 1)$.
  Note $(\cup\mathscr{I}_{an})\cap
  (\cup_{i<n}\cup\mathscr{I}_{ai})=\emptyset$
  and every $\mathscr{I}_{al}$ is c.o.D.
  So $\mathscr{I}_a=\cup_l\mathscr{I}_{al}$
  is c.o.D in $(Q, \rho')$.
  This implies 1. It is easy to see 2
  since $q=q_{li}\in S_n(q_l)$ with $1\leq i\leq a$
  implies $1\leq i\leq a<b$ if $a<b$.
\par

{\it The proof of 3.} Let $[q_i,q_{ia})\in\mathscr{I}_a$ with $b>a$.
  Then  $q_i,q_{ia}\in Q^*_a\subset Q^*_b$.
  Pick a $q_j\in[q_i,q_{ia})\subset[0,1)$.
  Then there uniquely exists a $[q_h,q_{hb})\in\mathscr{I}_b$
  with $q_j\in[q_h,q_{hb})$ by the above 1.
  Then $(q_h,q_{hb})\cap Q^*_b=\emptyset$
  since $\mathscr{I}_b$ is pairwise disjoint.
  Then $q_i,q_{ia}\notin(q_h,q_{hb})$.
  Then $q_j\in[q_h,q_{hb})\subset[q_i,q_{ia})$.
  Then $\cup\{[q_j,q_{jb})\in\mathscr{I}_b:
  [q_j,q_{jb})\subset[q_i,q_{ia})\}=[q_i,q_{ia})$
  since $\mathscr{I}_b$ is a c.o.D family
  with $\cup\mathscr{I}_b=[0,1)$.
\par

{\it The proof of 4.}
Pick a $q_i\in Q$. Then $q_i\in S_n$ for some $n\in N$.
  Then $q_i=q_{n_1k_1}\in S_n(q_{n_1})$ for some $q_{n_1}\in S_{n-1}$.
  And then $q_{n_1}=q_{n_2k_2}\in S_{n-1}(q_{n_2})$
  for some $q_{n_2}\in S_{n-2}$.... And then
  $q_{n_n}=q_{n_nk_n}\in S_2(q_{n_n})$ for some $q_{n_n}\in S_1$.
  Let $a=max\{k_n,...,k_1,i \}$. Then $i\leq a$.
  Then  $q_i=q_{n_1k_1}\in Q_{an}\subset Q^*_a$.
\par
{\it The proof of 5.} It is a corollary of 3 in Proposition 2.2.
\end{proof}

Let $Q_n$'s be copies of $Q$ with the topology $\rho'$
 for $n\in N$ and $X=\Pi_nQ_n$ with product topology.
 Let $\mathscr{Q}_i$'s be copies of $\mathscr{Q}$,
 $X_n=\Pi_{i>n}Q_i$,
 \[\mathscr{B}'_n=\{
  B_{n_1}\times...\times B_{n_n}\times X_n:
  B_{n_i}\in \mathscr{Q}_i
  \hbox{ for } i=1,2,...,n \}\]
  and $\mathscr{B}'=\cup_n\mathscr{B}'_n$.
  Then $\mathscr{B}'$ is a countable base.
  Denote the topology with the countable base
  $\mathscr{B}'$ by $\rho$.
  Then $(X,\rho)$ is a zero-dimensional metric space.

\section{To construct $\mathscr{H}$ and $\mathscr{G}$ on set $X$}

 Following only the idea of Condition A of $M_3$-structures
  of Proposition 4.7 in \cite{ch},
  we construct
  $\mathscr{H}=\cup_{n\in N}\mathscr{H}_n$ in $X$.

\vspace{0.5cm}
\textbf{Construction 3.1}
\vspace{0.5cm}

(A). Take $Q=\{q_n: n\in \omega\}$ with $\Delta$-order.
 Let $H(1,i)=\{q_i\}\times X_1$ and $\mathscr{H}'_1=\{H(1,j):j\in N\}$
 with $j=i+1$.
 Then $\mathscr{H}'_1$ has the same $\Delta$-order as $Q$.
\par
(B). Assume we have had  $\mathscr{H}'_n=\{H(n,i):i\in N\}$.
\par
 Let $P(n,i)=\{q_{i_1}\}\times\{q_{i_2}\}\times ...\times
  \{q_{i_n}\}$ for $H(n,i)=P(n,i)\times X_n\in\mathscr{H}'_n$
  and $k=n+1$. We take $P(k,l)$ in accordance with the following order:
  \begin{eqnarray}
  P(n,1)\times\{q_1\},\rightarrow P(n,1)\times\{q_2\}, \quad P(n,1)\times\{q_3\},
   \quad P(n,1)\times\{q_4\},...\quad\quad\quad\quad\quad\quad
   \nonumber\\
            \quad\quad\quad\swarrow\quad\quad\quad\quad \quad
            \quad \ \swarrow\ \quad  \quad\quad \quad\quad  \quad
             \quad\swarrow\quad\quad\quad\quad ...\quad  \ \quad
             \quad\quad \quad\quad \quad\quad
             \nonumber\\
  P(n,2)\times\{q_1\},\quad P(n,2)\times\{q_2\},  \quad P(n,2)\times\{q_3\},
    \quad P(n,2)\times\{q_4\},... \quad\quad\quad \quad  \ \quad\quad\nonumber\\
          \quad\quad\quad\swarrow\quad\quad \quad\quad\quad\quad \quad
          \swarrow \quad  \quad\quad \quad\quad  \quad ...\quad
          \quad\quad\quad\quad\quad\quad\quad \quad  \ \quad\quad\quad \quad \quad \nonumber\\
   P(n,3)\times\{q_1\},\quad P(n,3)\times\{q_2\},
     \quad P(n,3)\times\{q_3\},
     \quad P(n,3)\times\{q_4\},...\quad\quad\quad\quad\quad\quad\
     \nonumber\\
          \quad\quad\quad\swarrow\quad\quad\quad\quad  \quad\quad ...
          \quad\quad \quad\quad\ \ \quad  \quad\quad \quad\quad  \quad\quad\quad\quad\quad\quad \quad  \ \quad\quad\quad \quad  \quad  \nonumber\\
   P(n,4)\times\{q_1\},\quad P(n,4)\times\{q_2\}, \quad P(n,4)\times\{q_3\}, \quad P(n,4)\times\{q_4\},...
   \quad\quad\quad\quad\quad\quad\ \nonumber\\
         .....................
          .......................\quad\hbox{Figure I.}
  \quad
          .........................................................\quad
   \quad\quad\quad\quad\quad\nonumber
 \end{eqnarray}
 Let  $P(k,l)=P(n,i)\times\{q_j\}$ for $k=n+1$.
  Then $l=1+2+...+(i+j-2)+i=(i+j-2)(i+j-1)/2 +i$.
  We call it a $\Delta$-order still. Let
       \[ H(k,l)=P(k,l)\times X_k.\]

Let $\mathscr{H}'_k=\{H(k,l): l\in N\}$.
  Then, by induction, we have $\mathscr{H}'_k$
  for each $k\in N$ and call that $\mathscr{H}'_k$ has a $\Delta$-order.
  Let $\mathscr{H}'=\cup_n\mathscr{H}'_n$.

\begin{prop}
  \ \ \ 1 \ $\mathscr{H}'_n$ is a partition of $X$
  for each $n\in N$.
\par

2.\   If $H(n,i), H(n,i')\in\mathscr{H}'_n$,
   then $\rho(H(n,i),H(n,i'))=r>0$ or $H(n,i')=H(n,i)$.
\par

3.\   If $H(n,i), H(n',i')\in \mathscr{H}'$ with $n'>n$,
 then  $\rho(H(n,i), H(n',i'))=r>0$  or $H(n',i')\subset H(n,i)$.
 \par

4.\   If $H(n,i), H(n',i')\in \mathscr{H}'$ with
 $H(n',i')\subset H(n,i)$,
 then $n'\geq n$ and $i'\geq i$.
\par

5.\ Let $H(n,i)=\{q_{i_1}\}\times...\times\{q_{i_n}\}\times X_n$.
   Then $i\geq i_j$ for each $j\leq n$.
\end{prop}

\begin{proof} To check 1, let $Q_n=\{q_i: i\in \omega\}$ be the copy
 of $Q$ with $\Delta$-order
 and $\mathscr{H}'_1=\{\{q_i\}\times X_1: q_i\in Q_1 \}$.
 Then $\cup\mathscr{H}'_1=X$.
\par

Assume $\cup\mathscr{H}'_n=X$ with $\cup_iP(n,i)=\Pi_{i\leq n}Q_i$.
  Let $k=n+1$. Then $\cup\{P(n,i)\times\{q_j\}: i,j\in N\}
  =(\Pi_{i\leq n}Q_i)\times Q_k$ by induction of assumption.
  Then $\cup\mathscr{H}'_k=X$.
\par

On the other hand,  Take $H(n,i)$ and $H(n,j)$ from $\mathscr{H}'_n$
  with  $H(n,i)\neq H(n,j)$.
  Then $P(n,i)\neq P(n,j)$.
  Then $H(n,i)\cap H(n,j)=\emptyset$.
  This implies 1.
\par

{\it The proof of 2.}
  Take $H(n,i)$ and $H(n,j)$ from $\mathscr{H}'_n$
  with  $H(n,i)\neq H(n,j)$.
  Then $P(n,i)\neq P(n,j)$.
  Then $\rho_n(P(n,i), P(n,j))=r>0$
  in metric space $(\Pi_{i\leq n}Q_i, \rho_n)$.
  Pike $x\in H(n,i)$ and $x'\in H(n,j)$.
  Then $\rho(x,x')=\rho_n(P(n,i), P(n,j))=r>0$.
  Then $\rho(H(n,i), H(n,j))= min\{\rho(x,x'):
  x\in H(n,i) \hbox{ and } x'\in H(n,j)\}=r>0$.
\par

{\it The proof of 3.}
  Note that for each $H(k,i)\in\mathscr{H}'_k$ and $k=n+1$,
  there uniquely exists an $H(n,j')\in\mathscr{H}'_n$
  with $H(k,i)\subset H(n,j')$. Then 3 is a corollary of 2.
\par

{\it The proof of 4.}
   It is easy to see that $H(n',i')\subset H(n,i)$ and $H(n',i')\neq H(n,i)$
   imply $n'>n$. Let $n'=n+1=k$ with $H(n,i)=P(n,i)\times X_n$ and
   $H(k,i')=P(n,i)\times \{q_j\}\times X_k\in\mathscr{H}'_k$.
   Then $i'=1+2+...+(i+j-2)+i=(i+j-2)(i+j-1)/2 +i\geq i$
   by the $\Delta$-order of $\mathscr{H}'_k$.
\par

{\it The proof of 5.}
We use induction on $k$ to prove 5.
\par
A. $k=1$. Take an  $H(1,i)=\{q_i\}\times X_1\in \mathscr{H}'_1$.
   Then $i\geq i$ for each $i\in N$.
\par

B. $k=2$. Take an  $H(2,i')\in \mathscr{H}'_2$.
  Then $H(2,i')=P(1,i)\times\{q_j\}\times X_2$.
  We can prove that $i'=(i+j-2)(i+j-1)/2 + i\geq i$
  and $i'=(i+j-2)(i+j-1)/2 + i\geq j$ for $i+j\geq 2$,
  and $i'>i$ and $i'>j$ for $i+j\geq 4$.
\par

In fact, we have $i+j\geq 2$ since $i\geq 1$ and $j\geq 1$.
\par

Case 1, $i+j=2$. Then $i=j=1$. Then $i'=(i+j-2)(i+j-1)/2 + i=i=j=1$.
\par

Case 2, $i+j=3$.  Then $i'=(1\cdot2)/2+i=1+i>i$
  and $i'=1+i\geq j$ since $j\leq 2$.
\par

Case 3, $i+j=4$.  Then $i'=(2\cdot3)/2+i=3+i>i$
  and $i'=3+i>j$ since $j\leq 3$.
\par

Case 4,  $i+j\geq 5$.  Then $(i+j)^2\geq 5(i+j)$
  and $(i+j)^2-3(i+j)\geq 2(i+j)$.
  Then $i'=(i+j-2)(i+j-1)/2 + i>2i+j$.
  So $i'>i$ and $i'>j$.
\par

C. Assume that $H(n,i)=P(n,i)\times X_n=\{q_{i_1}\}\times ...
  \times\{q_{i_n}\}\times X_n$ such that $i\geq i_j$ for
   each $j\leq n$ when $k=n$. Let $k=n+1$.
\par
   Take an $H(k,i')=P(n,i)\times\{q_{i_k}\}\times X_k$.
  Then $i'=(i+i_k-2)(i+i_k-1)/2 + i$.
  Note $i\geq i_j$ for each $j\leq n$ by inductive assumption.
  Then $i'\geq i$ by the above 4. So $i'\geq i\geq i_j$ for each $j\leq n$.
  We can prove $i'=(i+i_k-2)(i+i_k-1)/2 + i\geq i_k$  if $i+i_k\geq 2$
  in the same way as the above Case 1 - 4 in B. This implies 5.

\end{proof}

Then $\mathscr{H}'=\cup_n\mathscr{H}'_n$ satisfies condition A
  of $M_3$-structure in \cite{ch}.
  Following only the idea of Condition B of $M_3$-structures
  of Proposition 4.7 in \cite{ch},
  we construct a $g$-function in $X$ by $\mathscr{H}'=\cup_n\mathscr{H}'_n$
  in the following Construction 3.2.

\vspace{0.5cm}
\textbf{Construction 3.2}
\vspace{0.5cm}

A. Take $\mathscr{H}'_n=\{H(n,i):i\in N\}$ and
   an $H(n,i)\in \mathscr{H}'_n$. Note
   \[H(n,i)=\{q_{i_1}\}\times...\times\{q_{i_n}\}\times X_n.\]
   Let $a=n+i$. Take $\mathscr{I}_a$.
   Take $[q_{l_j},q_{l_ja})$'s  from $\mathscr{I}_a$ for $j\leq i$.
   Let \[J(n,i,l)=[q_{l_1},q_{l_1a})\times...
  \times[q_{l_i},q_{l_ia})\times X_{n+i}.\]
   Let $\mathscr{J}(n,i)$ be the family of all $J(n,i,l)$'s.
   Then $\mathscr{J}(n,i)$ is countable
   and c.o.D in $(X_n,\rho)$ with $\cup\mathscr{J}(n,i)=X_n$.  Let
   $\mathscr{J}(n,i)=\{J(n,i,l): l\in N\}$,
    \[ H(n,i,l)=P(n,i)\times J(n,i,l) \hbox{   and   }
   \mathscr{H}(n,i)=\{H(n,i,l):l\in N\}.\]
   Then we have the following definition.
  \par
\vspace{0.3cm}

 \[B.\ \textbf{Definition.}\  H(n,i,l)=\{q_{i_1}\}\times...\times\{q_{i_n}\}\times
   [q_{l_1},q_{l_1a})\times...
  \times[q_{l_i},q_{l_ia})\times X_{n+i}\hbox{ and}\]
   \[ \mathscr{H}(n,i)=\{H(n,i,l):J(n,i,l)\in\mathscr{J}(n,i)\}
   =\{H(n,i,l):l\in N\}.\]
   Then $\mathscr{H}(n,i)$ is a c.D family in $(X,\rho)$
   with $\cup\mathscr{H}(n,i)=H(n,i)$.
   Let \[\mathscr{H}_n=\cup_i\mathscr{H}(n,i)\quad\quad
    and \quad\quad \mathscr{H}=\cup_n\mathscr{H}_n.\]
\par

In the same way as the $\Delta$-order of $\mathscr{H}'_n$,
  we give $\mathscr{H}_n$ a $\Delta$-order by the $\Delta$-order of
  $\mathscr{H}'_n$ as the following:
\begin{eqnarray}
  H(n,1,1),\ \rightarrow\  H(n,1,2),\quad
  H(n,1,3), \quad H(n,1,4),\quad...,\ H(n,1,l),...\quad\nonumber\\
 \swarrow \quad\quad\quad  \ \quad\quad  \swarrow \ \ \quad  \quad\quad\quad
  \swarrow \quad\quad  \quad ...  \quad\quad\quad\quad\quad\quad\quad\quad
  \quad \nonumber\\
  H(n,2,1),\quad\quad  H(n,2,2),\quad
  H(n,2,3), \quad H(n,2,4),\quad...,\ H(n,2,l),...\quad\nonumber\\
  \swarrow \quad\quad \quad\quad \quad  \swarrow\  \quad  \quad\quad
  ...\quad\quad \quad\quad\quad \quad \quad\quad\quad\quad
  \quad\quad\quad\quad\quad \nonumber\\
   H(n,3,1),\quad\quad H(n,3,2),\quad
  H(n,3,3),\quad H(n,3,4),\quad...,\ H(n,3,l),...\quad\nonumber\\
  \swarrow \quad\quad \quad ...  \quad  \quad\quad  \quad\quad
  \quad\quad\quad \quad \quad \quad\quad\quad \quad
  \quad\quad\quad \quad \quad \quad\quad\nonumber \\
   H(n,4,1),\quad\quad H(n,4,2),\quad
  H(n,4,3), \quad H(n,4,4),\quad...,\ H(n,4,l),...\quad\nonumber
 \end{eqnarray}
 \[..........................................\quad\hbox{Figure II.}
  \quad ...................................\]

  Then we can calculate the $i'$th member $H(n,i,l)$ in
  $\mathscr{H}_n=\cup_i\mathscr{H}(n,i)$  by
  \[ i'=1+2+...+(i+l-2)+i=(i+l-2)(i+l-1)/2 + i.\]
  Call that \textit{$\mathscr{H}_n$ has a $\Delta$-order.}
  For $H(n,i,l)\in\mathscr{H}(n,i)$, note $a=n+i$. Let
  \[I(n,i)=[q_{i_1},q_{i_1a})\times ...\times
  [q_{i_n},q_{i_na})\quad \hbox{ and}\quad
  g(n,i,l)=I(n,i)\times J(n,i,l).\]
  Then we have the following definition.

  \[C.\ \textbf{Definition.}\ g(n,i,l)=[q_{i_1},q_{i_1a})\times ...\times
  [q_{i_n},q_{i_na})\times [q_{l_1},q_{l_1a})\times...
  \times[q_{l_i},q_{l_ia})\times X_{n+i}.\]
  Let $\mathscr{G}(n,i)=\{g(n,i,l):l\in N\}$.
  Then \[\mathscr{G}(n,i)=\{g(n,i,l)=I(n,i)\times J(n,i,l):J(n,i,l)\in\mathscr{J}(n,i)\}.\]
  Then $\mathscr{G}(n,i)$ is a c.o.D family
  in $(X,\rho)$. Then
  \[\cup\mathscr{G}(n,i)=[q_{i_1},q_{i_1a})\times ...\times
  [q_{i_n},q_{i_na})\times X_n=I(n,i)\times X_n\]
  by $\cup\mathscr{J}(n,i)=X_n$.
  Let $\mathscr{G}_n=\cup_i\mathscr{G}(n,i)$ and
  $\mathscr{G}=\cup_n\mathscr{G}_n$.
  Then $\mathscr{G}_n=\cup_i\mathscr{G}(n,i)$ has
  a $\Delta$-order by the $\Delta$-order of $\mathscr{H}_n$.
  And then $\mathscr{G}=\cup_n\mathscr{G}_n$ has
  a $\Delta$-order by the $\Delta$-order of $\mathscr{G}_n$,
  in the same way as the $\Delta$-order of $\mathscr{H}_n$ by the $\Delta$-order of
  $\mathscr{H}'_n$.
\par

 For $1\leq k\leq n$, take a projection
 $\pi_k:X\rightarrow Q_k=Q$ and let
 \[\mathscr{J}^n_k=\{[q_{i_k},q_{i_ka}):
 [q_{i_k},q_{i_ka})=\pi_k[g(n,i,l)]
 \hbox{ for each }g(n,i,l)\in\mathscr{G}_n\}.\]
 \par

\vspace{0.3cm}
\textbf{Note D1:}  Let
$g(n,i,l)=g(n,x)=g(n,i,x)$ for $x\in H(n,i,l)\subset H(n,i)$.
    Then \[\mathscr{G}(n,i)=\{g(n,i,l):l\in N\}
  =\{g(n,i,x):x\in H(n,i)\}=\{g(n,x):x\in H(n,i)\}.\]
  It is easy to see that
  $g:N\times X\rightarrow \cup_n\mathscr{G}_n$
  is a function.
   We prove that $g:N\times X\rightarrow \cup_n\mathscr{G}_n$
  is a $g$-function of some stratifiable space in section 7.
\par

\vspace{0.3cm}
\textbf{Note D2:}
   Let $x,y\in H(n,i,l)$. Then
   $g(n,i,l)=g(n,x)=g(n,y)$.
   Let
    \[H(n,i,l)=H(n,i,x)=H(n,x)=H(n,i,y)=H(n,y)\quad\hbox{ and}\]
   \[n_1<n_2<...<n_k\quad\hbox{with}\quad
   H(n_k,x_k)\subset...\subset H(n_2,x_2)\subset H(n_1,x_1).\]
    Then $H(n_i,x_k)=H(n_i,x_i)$ since $x_k\in H(n_i,x_i)$ for $1\leq i\leq k$.
\par

\vspace{0.3cm}
\textbf{Note D3:}
 $\mathscr{G}(*,*,*)=\{g(n_i,x_i)\in\mathscr{G}: g(n_i,x_i)\hbox{ satisfies } P\}$
 if and only if $\mathscr{H}(*,*,*)=\{H(n_i,x_i)\in\mathscr{H}: g(n_i,x_i)\hbox{ satisfies } P\}$.
 So we give only one of definitions $\mathscr{G}(*,*,*)$ or $\mathscr{H}(*,*,*)$ if
 we definite families.
\par

\vspace{0.3cm}
\textbf{Note D4:} Let $H(n,i)\in \mathscr{H}'_n$ with
   \[H(n,i)=\{q_{i_1}\}\times...\times\{q_{i_n}\}\times X_n.\]
   We have $a=n+i$,
   \[g(n,i,l)=[q_{i_1},q_{i_1a})\times ...\times
  [q_{i_n},q_{i_na})\times [q_{l_1},q_{l_1a})\times...
  \times[q_{l_i},q_{l_ia})\times X_{n+i}\ \hbox{ and}\]
  \[H(n,i,l)=\{q_{i_1}\}\times...\times\{q_{i_n}\}\times
   [q_{l_1},q_{l_1a})\times...
  \times[q_{l_i},q_{l_ia})\times X_{n+i}.\]
   Let $g(n,x)=g(n,i,l)$ and $H(n,x)=H(n,i,l)$ for arbitrary $x\in H(n,i,l)$. Then
   \[g(n,x)\cap H(n,i)=H(n,x).\]
\par

\begin{prop}
Let $q_{i_k},q_{j_k}\in Q=\{q_i: i\in \omega\}$ with $\Delta$-order and $i_k,j_k\in N$. Then:
1. If $q_{j_k}\in(q_{i_k},q_{i_ka})$, then $j_k>i_k$, $j_k>a$ and
  $[q_{j_k},q_{j_kb})\subset(q_{i_k},q_{i_ka})$.

2. If $q_{j_k}\notin[q_{i_k},q_{i_ka})$ with $j_k>i_k$, then
  $[q_{i_k},q_{i_ka})\cap[q_{j_k},q_{j_kb})=\emptyset$.
\par
3. If $[q_{j_k},q_{j_kb}),[q_{i_k},q_{i_ka})\in\mathscr{J}^n_k$
  with $b>a$, then $[q_{i_k},q_{i_ka})\cap[q_{j_k},q_{j_kb})=\emptyset$
  or $[q_{j_k},q_{j_kb})\subset[q_{i_k},q_{i_ka})$.
\par

\end{prop}

\begin{proof}
To prove 1 let $q_{j_k}\in(q_{i_k},q_{i_k0})$,
  $q_{j_k}\in S_{j^*_k}$
 and $q_{i_k}\in S_{i^*_k}$.
 Note, by B) in Construction 1, $(q_{i_k},q_{i_k0})\cap(\cup_{j\leq i^*_k}S_j)=\emptyset$.
 Then $j^*_k>i^*_k$. Let $j^*_k=i^*_k+h$.
 Because of $q_{j_k}\in (q_{i_k},q_{i_k0})$,
 take $S_{h_1}(q_{i_k})=\{q_{i_kl}\in S_{h_1}:l\in N\}$
 with $h_1=i^*_k+1$. Then there exists a $q_{i_kl_1}\in S_{h_1}(q_{i_k})$
 such that  $q_{j_k}\in [q_{i_kl_1},q_{i_kl_1-1})$, $q_{i_kl_1}=q_{1_k}$
 and $1_k>i_k$ in $Q$ by 3 of Proposition 2.1.
\par

If $q_{j_k}=q_{i_kl_1}$, then $j_k=1_k>i_k$.
  If $q_{j_k}\neq q_{i_kl_1}$, then $q_{j_k}\in (q_{1_k},q_{1_k0})$.
  Take $S_{h_2}(q_{1_k})=\{q_{1_kl}\in S_{h_2}:l\in N\}$ with $h_2=h_1+1$.
  Then there exists a $q_{1_kl_2}\in S_{h_2}(q_{1_k})$
 such that  $q_{j_k}\in [q_{1_kl_2},q_{1_kl_2-1})$, $q_{1_kl_2}=q_{2_k}$
 and $2_k>1_k$  in $Q$ by 3 of Proposition 2.1.
\par

Note $i^*_k<i^*_k+1=h_1<...<i^*_k+h=j^*_k$. Then, by finite induction,
  there is an $m\leq h$,
  a $q_{\ell_kl_m}\in S_{h_m}(q_{\ell_k})=\{q_{\ell_kl}
  \in S_{h_m}:l\in N\}$ ($\ell=m-1$) such that
  $q_{j_k}\in [q_{\ell_kl_m},q_{\ell_kl_m-1})$,
  $q_{j_k}=q_{\ell_kl_m}=q_{m_k}$
  and $m_k>\ell_k$ in $Q$ by 3 of  Proposition 2.1.
  Then $j_k=m_k>\ell_k>...>2_k>1_k>i_k$.
\par

 On the other hand, $q_{j_k}\in(q_{i_k},q_{i_ka})$
  implies that
  there exists an $l>a$ with $q_{j_k}\in [q_{i_kl},q_{i_kl-1})$.
  Let $q_{i_kl}=q_{l_k}$ in $Q$. Then, by the above proof,
  we have $j_k\geq l_k$ in $Q$. Note that $q_{i_kl}$ is the $l$th
  in $S_{h_1}(q_{i_k})\subset Q$, and $q_{i_kl}=q_{l_k}$ is the
  $l_k$th in $Q$. So $l_k> l$ by 1 of Proposition 2.1.
  So $j_k\geq l_k>l>a$.
  And then, it is easy to see
  $[q_{j_k},q_{j_kb})\subset[q_{i_k},q_{i_ka})$
  since $q_{j_k}\in(q_{i_k},q_{i_ka})$
  implies $[q_{j_k},q_{j_kb})\subset
  [q_{j_k},q_{j_k1})\subset
  (q_{i_k},q_{i_ka})$. This implies 1.
\par

{\it The proof of 2.}
$q_{j_k}\notin [q_{i_k},q_{i_ka})$ implies $q_{j_k}<q_{i_k}$
  or $q_{i_ka}<q_{j_k}$. Case 1, $q_{i_ka}<q_{j_k}$ implies
  $[q_{i_k},q_{i_ka})\cap [q_{j_k},q_{j_kb})=\emptyset$.
  Case 2, $q_{j_k}<q_{i_k}$. Suppose $q_{i_k}\in (q_{j_k},q_{j_k0})$.
  Then $i_k>j_k$ by 1 of Proposition 3.2,
  a contradiction to $j_k>i_k$.
  So $q_{i_k}\notin (q_{j_k},q_{j_k0})$. Then $q_{j_k}<q_{i_k}$
  implies $[q_{j_k},q_{j_k0})\cap[q_{i_k},q_{i_ka})=\emptyset$.
  This implies 2.
\par

{\it The proof of 3.} Note $b>a$.
 Let $[q_{j_k},q_{j_kb})\cap[q_{i_k},q_{i_ka})\neq\emptyset$.
 If $q_{i_k}=q_{j_k}$,
 then $[q_{j_k},q_{j_kb})\subset[q_{i_k},q_{i_ka})$.
 If $q_{i_k}\neq q_{j_k}$, then $q_{i_k}\in(q_{j_k},q_{j_kb})$
 or $q_{j_k}\in(q_{i_k},q_{i_ka})$.
 $q_{i_k}\in(q_{j_k},q_{j_kb})$
 implies $i_k>j_k$ and $i_k>b$ by 1 of Proposition 3.2.
 $i_k>b$ implies $i_k>b=n+j>a=n+i>i$,
 a contradiction to $i\geq i_k$ by 5 of Proposition 3.1.
 So $q_{j_k}\in(q_{i_k},q_{i_ka})$ implies
 $[q_{j_k},q_{j_kb})\subset[q_{i_k},q_{i_ka})$.
\end{proof}

\begin{prop}
 Function $\mathscr{G}$ satisfies the following
 conditions:
\par
1 \  $\cap_n g(n,i,y)=\{ y\}$.
\par

4 \ $y\in g(n,i,x)$ implies $g(n,j,y)\subset g(n,i,x)$
  for some $j$.
\par
5 \ $g(n+1,j,x)\subset g(n,i,x)$.
\par
6 \ $j>k$ implies  $H(n,k)\cap(\cup\mathscr{G}(n,j))=\emptyset$.

\par
6$'$ \  $H(n,k,h)\cap g(n,j,l)\neq\emptyset$
  and $H(n,k,h)\neq H(n,j,l)$ imply $k>j$.

\par
7  \  Every $\mathscr{G}(n,i)$ is a c.o.D family.
\par
8  \  If $g(n,i,l),g(n,j,\ell)\in\mathscr{G}_n$ with $j>i$,
    then $g(n,i,l)\cap g(n,j,\ell)=\emptyset$ or
    $g(n,j,\ell)\subset g(n,i,l)$.
\end{prop}

\begin{proof}
We prove 8 at firstly.
Let $g(n,i,l),g(n,j,\ell)\in\mathscr{G}_n$ with $j>i$
and $g(n,i,l)\cap g(n,j,\ell)\neq\emptyset$.
 Then $[q_{j_k},q_{j_kb})\subset[q_{i_k},q_{i_ka})$
 by 3 of Proposition 3.2 since
 $j>i$ implies $b=n+j>n+i=a$.
 Then $[q_{l_k},q_{l_ka})\cap[q_{\ell_k},q_{\ell_kb})\neq\emptyset$
 implies $[q_{\ell_k},q_{\ell_kb})\subset[q_{l_k},q_{l_ka})$
 by 5 of Proposition 2.2.
\par

{\it The proof of 1:} Let $x\in H(n,i_n,l_n)\subset g(n,i_n,x)$
  for $n\in N$.
  Then $\cap_nH(n,i_n,l_n)=\{x\}$.
  Then $\cap_n g(n,i_n,x)=\{x\}$.
\par

{\it The proof of 7:} $\mathscr{G}(n,i)$
  is a c.o.D family in $(X,\rho)$
  by C of Construction 3.2.
\par

{\it The proof of 6:} Take an $H(n,i,l)$ and a $g(n,j,\ell)$ with
   $j>i$.
   Then there exists $k\leq n$ with $q_{i_k}\neq q_{j_k}$.
   If $q_{i_k}<q_{j_k}$, then $q_{i_k}\notin [q_{j_k},q_{j_kb})$.
   So $H(n,i,l)\cap g(n,j,\ell)=\emptyset$.
   If  $q_{i_k}\in (q_{j_k},q_{j_kb})$,
   then $i_k>j_k$ and $i_k>b$
   by 1 of Proposition 3.2.
   Note $i_k>b$ implies $i_k>b=n+j>n+i>i$,
   a contradiction to $i\geq i_k$
   by 5 of Proposition 3.1.
   So $q_{j_kb}<q_{i_k}$. Then $q_{i_k}\notin [q_{j_k},q_{j_kb})$.
   Then $H(n,i,l)\cap g(n,j,\ell)=\emptyset$.
   So $H(n,i,l)\cap[\cup\mathscr{G}(n,j)]=\emptyset$ if $j>i$.
\par

{\it The proof of 6$'$:} Suppose $j\geq k$.
  Case 1, $j=k$. Note that $h=l$ implies
  $H(n,k,h)=H(n,j,l)$,
  a contradiction to $H(n,k,h)\neq H(n,j,l)$.
  So $H(n,k,h)\cap H(n,j,l)=\emptyset$
  implies $h\neq l$. Then $g(n,k,h)\cap g(n,j,l)=\emptyset$
  by the above 7, a contradiction
  to $H(n,k,h)\cap g(n,j,l)\neq\emptyset$.
\par

  Case 2, $j>k$. Then, by the above 6,
  $H(n,k,h)\cap g(n,j,l)=\emptyset$,
  a contradiction to
  $H(n,k,h)\cap g(n,j,l)\neq\emptyset$.

{\it The proof of 4:} In fact, pick a $y\in g(n,i,x)=g(n,i,l)$.
   If $y\in H(n,i,l)$, then $g(n,i,y)=g(n,i,x)$.
   If $y\notin H(n,i,l)$, then $y\in H(n,j,\ell)$
   with $H(n,j,\ell)\cap g(n,i,x)\neq\emptyset$
   and  $H(n,j,\ell)\neq H(n,i,l)$.
   Then $j>i$ by 6$'$ of Proposition 3.3. Then $b=n+j>n+i=a$.
   Then $[q_{j_k},q_{j_kb})\subset[q_{i_k},q_{i_ka})$
   by 3 of Proposition 3.2.
   Then $[q_{\ell_k},q_{\ell_kb})\subset[q_{l_k},q_{l_ka})$
   by 5 of Proposition 2.2.
   This implies  $g(n,j,y)\subset g(n,i,x)$.
\par

{\it The proof of 5:} Let $\pi_k$ be a projection from $X$ to $Q_k$
 and let $H(n,i,l)\subset H(n-1,j,\ell)$.
 Then $i\geq j$ by 4 of Proposition 3.1.
 Then $b=n+i>(n-1)+j=a$.
 So $\pi_k(x)=q_{i_k}\in[q_{j_k},q_{j_ka})$ implies
 $[q_{i_k},q_{i_kb})\subset[q_{j_k},q_{j_ka})$
 by 3 of Proposition 3.2 for $1\leq k\leq n-1$.
 On the other hand, $H(n,i,l)\subset H(n-1,j,\ell)$ implies
 $[q_{l_k},q_{l_kb})\subset[q_{\ell_k},q_{\ell_ka})$
 by 5 of Proposition 2.2 since $[q_{l_k},q_{l_kb})$ and
 $[q_{\ell_k},q_{\ell_ka})$ in $\mathscr{I}$
 for $n< k\leq n+i=b$.
 And then $\pi_n(x)=q_{i_n}\in[q_{\ell_1},q_{\ell_1a})$
 implies $[q_{i_n},q_{i_nb})\subset[q_{\ell_1},q_{\ell_1a})$
 since $b>a$.
 So $g(n,x)=g(n,i,l)\subset g(n-1,j,\ell)=g(n-1,x)$.
\end{proof}

\begin{prop}
$\mathscr{G}=\cup_n\mathscr{G}_n$ is a base
  of metric space $(X,\rho)$.
\end{prop}

\begin{proof}
Let $B=B_{n_1}\times...\times
  B_{n_n}\times X_n\in\mathscr{B}'$.
  Then $B=[q_{n_1},q_{m_1})\times...\times
  [q_{n_n},q_{m_n})\times X_n$.
  Let $x=(q_{i_1},...,q_{i_n},...)\in B$.
  Then $q_{i_j}\in[q_{n_j},q_{m_j})$ for $j\leq n$.
\par

Case 1, $q_{i_j}=q_{n_j}\in[q_{n_j},q_{m_j})$.
  Then $[q_{n_j},q_{n_j\ell_j})\subset
 [q_{n_j},q_{m_j})$ for some $\ell_j$.
 Take $\ell_j$.
\par

Case 2, $q_{i_j}\in(q_{n_j},q_{m_j})$.
  Let $q_{i_j}\in S_i$ and $q_{m_j}\in S_m$.
  If $i\geq m$, take $\ell_j=0$.
  If $i<m$, then there exists $[q_{i_jl_j+1},q_{i_jl_j})$
  with $q_{m_j}\in[q_{i_jl_j+1},q_{i_jl_j})$
  and $q_{i_jl_j+1}\in(q_{n_j},q_{m_j})$.
  Take $\ell_j=l_j+1$.
\par

Let $k=\hbox{max}\{\ell_j: j\leq n\}$ and
 $l=\hbox{max}\{k, n\}$.
 Note, for each $\mathscr{H}_n=\cup_i\mathscr{H}(n,i)$,
 there uniquely exists an $H(n,i_n,l_n)\in\mathscr{H}_n$
 with $x\in H(n,i_n,l_n)$.
 Let $\mathscr{H}_x=\{H(n,i_n,l_n)\in\mathscr{H}_n:
 x\in H(n,i_n,l_n)\}$.
\par

 Take $H(l,i,h)=H(l,i_l,l_l)$ from $\mathscr{H}_x$.
 Let $a=l+i_l$. Then
 \[g(l,x)=g(l,i,h)=[q_{i_1},q_{i_1a})\times...\times
  [q_{i_l},q_{i_la})\times J(l,i,h)\subset B.\]
\end{proof}

In topological space $(X,\rho)$, denote the closure
 of a set $A$ by $\hbox{Cl}_\rho A$, and let
 \[\hbox{Int}_\rho A=X-\hbox{Cl}_\rho(X-A).\]

\begin{prop}
1. \quad $\hbox{Int}_\rho H(n, x')=\emptyset$  for arbitrary $H(n, x')\in \mathscr{H}$.
\par

2. $Cl_\rho[H(n,y)-H(n+1,y)]=H(n,y)$ for arbitrary $H(n, y)\in \mathscr{H}$.
\par

3. Let $H(n_{i+1},y_{i+1})\subset H(n_i,y_i)$ and $n_i<n_{i+1}$.
  Then $\cap_i H(n_i,y_i)=\{y\}$.
\end{prop}

\begin{proof}
Note $g(\ell,x)-H(l,y)\neq\emptyset$ for arbitrary $g(\ell,x)$ and arbitrary
  $H(l,y)$.
  Then $\hbox{Int}_\rho H(n, x')=\emptyset$
   for arbitrary $H(n, x')\in \mathscr{H}$.
\par

To see 3 note $H(n_i,y_i)=\{q_{i_1}\}\times...\times\{q_{i_n}\}\times
   [q_{l_1},q_{l_1a})\times...
  \times[q_{l_i},q_{l_ia})\times X_{n_i}$ for $n_i=n$ by the definition
  of $H(n,i,l)$. Note
  $H(n_{i+1},y_{i+1})\subset H(n_i,y_i)$ for $i\in N$. Let $n_{i+1}=m>n=n_i$.
  Then
  \[H(n_{i+1},y_{i+1})=\{q_{i_1}\}\times...\times\{q_{i_n}\}\times...\times\{q_{i_m}\}
   \times[q_{h_1},q_{h_1b})\times...
  \times[q_{h_k},q_{h_kb})\times X_{n_{i+1}}.\]
  Then $\cap_iH(n_i,y_i)=(q_{i_1},q_{i_2},...,q_{i_n},...)$.
\end{proof}

\section{To construct  Ln covers sequences on $G_m$}

  We'll use Ln covers and Ln covers sequences throughout this paper.
  To construct Ln-covers sequences on $g(m,a^*,y_0)$,
  let $m^2=m+a^*$, $n\geq m$,
  \[G_m=g(m,a^*,y_0)=I(m,a^*,y_0)\times J(m,a^*,y_0),\quad
  H_0=H(m,a^*)\cap g(m,a^*,y_0),\]
  \[\mathscr{H}(m,1)=\{H(n,1'_i,i_l)\in\mathscr{H}_n: H(n,1'_i,i_l)
  \subset G_m\} \hbox{ and }\]
  \[\mathscr{G}(m,1)=\{g(n,1'_i,i_l)\in\mathscr{G}_n: H(n,1'_i,i_l)
  \in\mathscr{H}(m,1)\}.\]
  Then $H_0=H(m,a^*,y_0)$ by Note D4. Then we have the following fact.
\vspace{0.3cm}

\textbf{Fact 4.0.} $H(\ell,h,k)\cap G_m\neq\emptyset$
 implies $g(\ell,h,k)\subset G_m$ if $\ell\geq m$.

\vspace{0.3cm}

\textbf{Construction 4.}

\vspace{0.3cm}
 \textsf{Condition.}
1. $\mathscr{H}(m,1)=\{H(n,1'_i,i_l)\in\mathscr{H}_n:
  H(n,1'_i,i_l)\subset G_m\}$ with $\Delta$-order.
\par

2.  $\cup\mathscr{H}(m,1)=\cup\mathscr{G}(m,1)=G_m$.
\par
\vspace{0.2cm}

\AE. We use Conditions of Construction 4 to construct
  a c.o.D family
  $\mathscr{G}(n,m,0)$ in $(X,\rho)$ by induction.
\par

\vspace{0.3cm}
\textbf{Operation O1.}
 \textsf{Conditions.}
1.  $\mathscr{H}(m,1)$ with $\Delta$-order.
\par

2. $1'_1<1'_2<...<1'_i<...$.
\par

We use Conditions of Operation O1 to construct the first cover of $G_m$
  by induction.
\par
A. Pick the least number $1'_1$. Denote $1'_1$ by $1_1$.
   Let \[\mathscr{H}^*(1,1)=
   \{H(n,1_1,1_l)\in\mathscr{H}(n,1_1):g(n,1_1,1_l)\subset G_m\}.\]
   Then $\mathscr{O}^*(1,1)=\{g(n,1_1,1_l):l\in N\}$
   is a c.o.D family.
   Let $O^1_m=\cup\mathscr{O}^*(1,1)$.

\par

B. Assume we have had a c.o.D family $\mathscr{O}^*(1,k)$, a
   c.D family $\mathscr{H}^*(1,k)$ and an $O^k_m$ for each $k<h$.
   Let $G^h_m=G_m-\cup_{k<h}O^k_m$
   and \[\mathscr{H}'(1,h)=\{H(n,1'_i,1_l):
   H(n,1'_i,1_l)\cap G^h_m\neq\emptyset\}.\]
   If $\mathscr{H}'(1,h)=\emptyset$,
   let $\mathscr{H}(n,m,0)=\cup_{k<h}\mathscr{H}^*(1,k)$ and
   $\mathscr{G}(n,m,0)=\cup_{k<h}\mathscr{O}^*(1,k)$.
   If $\mathscr{H}'(1,h)\neq\emptyset$,
   let $1_e=min\{1'_i:H(n,1'_i,1_l)\in\mathscr{H}'(1,h)\}$.
   Denote $1_e$ by $1_h$.
   Suppose $H(n,1_h,h_l)\in\mathscr{H}'(1,h)$
   with $H(n,1_h,h_l)\cap O^k_m\neq\emptyset$
   for some $k<h$.
   Then $H(n,1_h,h_l)\cap g(n,1_k,k_l)\neq\emptyset$
   for some $g(n,1_k,k_l)\in\mathscr{O}^*(1,k)$.
   Then $g(n,x)=g(n,1_h,h_l)\subset g(n,1_k,k_l)
   \subset O^k_m$ for $x\in H(n,1_h,h_l)\cap g(n,1_k,k_l)$
   by 4 of Proposition 3.3.
   Then $g(n,1_h,h_l)\cap G^h_m=\emptyset$,
   a contradiction to $H(n,1_h,h_l)\cap G^h_m\neq\emptyset$.
   Then $H(n,1_h,h_l)\cap G^h_m\neq\emptyset$ implies
   $g(n,1_h,h_l)\subset G^h_m$.
   Let \[\mathscr{H}^*(1,h)=
   \{H(n,1_h,h_l)\in \mathscr{G}(n,1_h):H(n,1_h,h_l)\subset G^h_m\}.\]
   Then $\mathscr{O}^*(1,h)=\{g(n,1_h,h_l):l\in N\}$
   is a c.o.D family.
   Let $O^h_m=\cup\mathscr{O}^*(1,h)$.
\par

Then, by induction, we have had a c.o.D
   family $\mathscr{O}^*(1,h)$,
   a c.D family $\mathscr{H}^*(1,h)$ and an $O^h_m$ for each $h\in N$.
   Let $\mathscr{G}(n,m,0)=\cup_h\mathscr{O}^*(1,h)$
   and $\mathscr{H}(n,m,0)=\cup_h\mathscr{H}^*(1,h)$.
\par

We call the above induction \textbf{ Operation O1 on
   $\mathscr{H}(m,1)$ and $\mathscr{G}(m,1)$.}
\par

\begin{cl}
1. $\cup\mathscr{G}(n,m,0)=G_m$ and
 $g(n,h,k)$ is a c.o set in $(X,\rho)$ for every
 $g(n,h,k)\in\mathscr{G}(n,m,0)$.
\par

2. Every $\mathscr{O}^*(1,h)$ is c.o.D,
  and
  $\mathscr{G}(n,m,0)=\cup_h\mathscr{O}^*(1,h)$ is c.o.D in $(G_m,\rho)$.
\par

3. $\mathscr{H}(n,m,0)$ is a c.D family and
  $g[n,H^1_m]=\cup\{g(n,x): x\in H^1_m\}=G_m$
  for $\cup\mathscr{H}(n,m,0)=H^1_m$.
\par

4. Let $g(\ell,i'',l'')\subset G_m$
  with $\ell\geq n$.
  Then there exists a $g(n,i',l')\in\mathscr{G}(n,m,0)$
  such that $g(\ell,i'',l'')\subset g(n,i',l')$.
\par

5. Let $g(n,i,l)$ be the first member in $\mathscr{G}(n,m,0)$
  with $\Delta$-order. Then
  \[H(n,i,l)=H(n,1_1,1_1)\subset\cup\mathscr{H}^*(1,1)\subset H_0.\]
\end{cl}

\begin{proof}
To see 1 pick an $x\in G_m=\cup\mathscr{H}(m,1)$. Then there exists an
  $H(n, 1'_i, i_h)$ with $x\in H(n, 1'_i, i_h)$.
  Case 1, $i=1$. Then $ H(n, 1_1, 1_e)\in\mathscr{H}^*(1,1)$
  and $x\in H(n, 1_1, 1_e)\subset O^1_m$.
  Case 2, $i>1$.  Then $1_k<1'_i<1_{k+1}$.
  This implies $x\in O^k_m$, and
   $1'_i=1_{k+1}$ implies $x\in O^{k+1}_m$.
   And then it is easy to see that $g(n,h,k)$ is a c.o set
   in $(X,\rho)$.
   This implies 1.
\par

{\it The proof of 2.}
Note that $O^k_m\cap O^{k+1}_m=\emptyset$ and
 each $\mathscr{O}^*(1,h)$ is a c.o.D family.
 Then $\mathscr{G}(n,m,0)$ is a c.o.D family in $(X,\rho)$
 since $\cup\mathscr{G}(n,m,0)=G_m$ is a c.o set.
\par

{\it The proof of 4.}
 Let  $H(n,i',l')\cap G_m\neq\emptyset$.
  Then $g(n,i',l')\subset G_m$ by Fact 4.0.
  Then $H(n,i',l')\in\mathscr{H}(m,1)$.
  Then there exists an $1_i$ with $1_i\leq i'<1_{i+1}$.
  If $1_i= i'$, then $g(n,1_i,i_l)=g(n,i',l')$
  for some $i_l$.
  If $1_i< i'$, then $i'<1_{i+1}$
  and $H(n,i',l')\subset g(n,1_i,i_l)$ for some $i_l$
  by the definition of $1_i$.
  Then $g(n,i',l')\subset g(n,1_i,i_l)$.
\par
Let $\ell> n$.
  Then there exists an  $H(n,i',l')\in\mathscr{H}(m,1)$
  with $H(\ell,i'',l'')\subset H(n,i',l')$.
  Then there exists a
  $g(n,i^*,l^*)\in\mathscr{G}(n,m,0)$
  such that $ H(n,i',l') \subset g(n,i^*,l^*)$.
  So $g(\ell,i'',l'')\subset g(n,i^*,l^*)$
  by 4 and 5 of Proposition 3.3.
\par

{\it The proof of 5.} Let $g(n,i,l)\in\mathscr{G}(n,m,0)$
  be the first member in $\mathscr{G}(n,m,0)$.
  Then $H(n,i,l)\subset G_m=g(m,a^*,y_0)$.
  Then $i\geq 1_1$ since $\cup\mathscr{H}^*(1,1)\subset H_0$
  and $1_1$ is the least number.
\par

Suppose  $i>1_1$.
  Then $H(n,i,l)$ is in the $i$'th line in the Figure II.
  It is a contradiction since $H(n,i,l)$ is the first member in
  $\mathscr{G}(n,m,0)$
  with $\Delta$-order (following the direction of arrow in the Figure II).
\par

Then $i=1_1$ and $H(n,i,l)\subset\cup\mathscr{H}^*(1,1)\subset H_0$.
  Then, in the first line, we have $H(n,i,l)=H(n,1,1)=H(n,1_1,1_1)$.

\end{proof}

\textbf{Definition Ln.}
Call $\mathscr{G}(n)$
  \textbf{a least number cover ( Ln cover)
  on $(n,B)$ (or on $B$)} if and only if
  $\mathscr{G}(n)$ and $B$
  satisfy 1, 2 and 4 of Claim 4.1.
\par
It is easy to check the following fact.

\vspace{0.3cm}

\textbf{Fact *.} Let $\mathscr{G}(n)$ be a Ln cover on
   $(n,B)$.
  Then $\mathscr{G}'$ is a Ln cover on $(n,\cup\mathscr{G}')$
  if $\mathscr{G}'\subset\mathscr{G}(n)$.

\vspace{0.3cm}
\textbf{Definition Ln$'$.}
Let $H$ be a closed set in $(\cup\mathscr{G}',\rho)$ with
 $H\cap H(n,i',l')\neq\emptyset$
   for each $g(n,i',l')\in\mathscr{G}'$. Also call
   \textbf{ $\mathscr{G}'$  a Ln cover of $(n,H)$
   (or $H$).}
\vspace{0.3cm}

\textbf{Note Ln.}
Let $H$ be a closed set in $(X,\rho)$ and
  $\mathscr{G}(n,H)=\{g(n,x)\in\mathscr{G}_n: x\in H\}$.
  Then there exists $\mathscr{G}'(n,H)\subset\mathscr{G}(n,H)$
  such that $\cup\mathscr{G}'(n,H)=\cup\mathscr{G}(n,H)$
  and $\mathscr{G}'(n,H)$ is a Ln cover of $(n,H)$
  by 3 of Claim 4.1.
  Call $\mathscr{G}'(n,H)$  a Ln subcover of $(n,H)$ also.
  Then $\cup\mathscr{G}'(n,H)$ is c.o in $(X,\rho)$.
\vspace{0.3cm}

 \OE. ( \textit{Construction 4 is continued.})
  Assume that we have constructed
   a c.o.D family $\mathscr{G}(n_a,m,0)$ and \hbox{ }
   a c.D family $\mathscr{H}(n_a,m,0)$ in $(X,\rho)$
   such that Claim 4.1.
   In the following, we construct the next Ln cover by induction.
   To do it let $b=a+1$ and $n_b=n_a+1$.
\par

  Take a $g(n_a,a_j,j_h)\in\mathscr{G}(n_a,m,0)$.
  For $H(n,i,i_l)\in\mathscr{H}(m,1)$,
  let $H(n,b'_i,i_l)=H(n,i,i_l)\cap g(n_a,a_j,j_h)$
  if $H(n,i,i_l)\cap g(n_a,a_j,j_h)\neq\emptyset$, and let
  \[\mathscr{H}(n,a_j,j_h)=
  \{H(n,b'_i,i_l)=H(n,i,i_l)\cap g(n_a,a_j,j_h):
  H(n,i,i_l)\in\mathscr{H}(m,1)\}.\]
  Then $ b'_1<b'_2<...<b'_i<...$ by $\Delta$-order of $\mathscr{H}(m,1)$.
  Let
  \[\mathscr{H}'(n_b,a_j,j_h)=\{H(n_b,b_i,i_k)\in\mathscr{H}_{n_b}:
  H(n_b,b_i,i_k)\subset g(n_a,a_j,j_h)\}.\]
  Then $\cup\mathscr{H}'(n_b,a_j,j_h)=\cup\mathscr{G}'(n_b,a_j,j_h)=g(n_a,a_j,j_h)$.
  Give $\mathscr{H}'(n_b,a_j,j_h)$ a $\Delta$-order as a subsequence of $\mathscr{H}_{n_b}$.
  Then we have the following claim.

  \begin{cl}
 Let $H(n_b,b^1_1,1_e)$ be the first member in $\mathscr{H}'(n_b,a_j,j_h)$
  with $\Delta$-order, and $H(n,b'_1,1_l)$ be the first member in
  $\mathscr{H}(n,a_j,j_h)$ with $\Delta$-order. Then:
\par
1.  $H(n_b,b^1_1,1_e)\subset H(n_a,a_j,j_h)\subset H(n,b'_1,1_l)$.

\par
2.  $b'_1$ is the least number in $\mathscr{H}(n,a_j,j_h)$, and $b^1_1$
  is the least number in $\mathscr{H}'(n_b,a_j,j_h)$ with $b'_1<a_j<b^1_1$.
\end{cl}

\begin{proof}
At first, we prove $H(n_a,a_j,j_h)\subset H(n,b'_1,1'_l)$
 for the first member $H(n,b'_1,1'_l)\in\mathscr{H}(n,a_j,j_h)$.
 To do it Suppose $H(n_a,a_j,j_h)\subset H(n,b_1,1_l)$  with $b'_1\neq b_1$.
\par

Then $H(n_a,a_j,j_h)\subset H(n,b_1,1_l)$ implies
  $g(n_a,a_j,j_h)\subset g(n,b_1,1_l)$ since $n<n_a$.
  Then $H(n,b'_1,1'_l)\cap g(n,b_1,1_l)\neq\emptyset$.
  Then $b'_1\neq b_1$ implies $b_1<b'_1$ by $6'$ of Proposition 3.3.
  It is a contradiction to the definition of $b'_1$.
  Then $b_1=b'_1$.
\par

On the other hand, $1'_l\neq 1_l$
  implies $g(n,b'_1,1'_l)\cap g(n,b'_1,1_l)=\emptyset$
  since $\mathscr{G}(n,b'_1)$ is discrete by 7 of Proposition 3.3.
  Then $H(n,b'_1,1'_l)\cap g(n_a,a_j,j_h)=\emptyset$,
   a contradiction to
  $H(n,b'_1,1'_l)\in\mathscr{H}(n,a_j,j_h)$. This implies
  $H(n_a,a_j,j_h)\subset H(n,b_1,1_l).$
\par

We prove  $H(n_b,b^1_1,1_e)\subset H(n_a,a_j,j_h)$.
 To do it take the first member $H(n_b,b^1_1,1_e)$ from
 $\mathscr{H}'(n_b,a_j,j_h)$.
\par

Replace $H_0$ with $H(n,b_1,1_l)$, $\cup\mathscr{H}^*(1,1)$ with
  $H(n_a,a_j,j_h)$, and $\mathscr{G}(n,m,0)$ with $\mathscr{G}'(n_b,a_j,j_h)$
  in 5 of Claim 4.1.
  Here \[\mathscr{H}'(n_b,a_j,j_h)=\{g(n_b,x)\in\mathscr{G}(n_b,m,0):
  g(n_b,x)\subset g(n_a,a_j,j_h)\}.\]
  Then, by 5 of Claim 4.1, we have
  $H(n_b,b^1_1,1_e)\subset H(n_a,a_j,j_h)\subset H(n,b'_1,1_l)$.
  Then, by 4 of Proposition 3.1, $b'_1<a_j<b^1_1$ since $n_b>n_a>n\geq m$.
\end{proof}

\textit{Construction 4 is continued.} Take $\mathscr{G}'(n_b,a_j,j_h)$
   and $\mathscr{H}'(n_b,a_j,j_h)$. Then
   \[\mathscr{H}'(n_b,a_j,j_h)=\{H(n_b,b_i,i_k)\in\mathscr{H}_{n_b}:
   i,k\in N\},\]
   \[\cup\mathscr{H}'(n_b,a_j,j_h)=\cup\mathscr{G}'(n_b,a_j,j_h)=g(n_a,a_j,j_h)
  \ \hbox{ and }\ b_1<b_2<...<b_i<....\]
  Replace $n$ with $n_b$, $G_m$ with $g(n_a,a_j,j_h)$, and
  $\mathscr{H}(m,1)$ with $\mathscr{H}'(n_b,a_j,j_h)$ in operation O1.
  Then Condition of Operation O1 is satisfied. Then we have the following claim.

\begin{cl}
1. $\mathscr{G}(n_b,a_j,j_h)$ is a Ln cover on $g(n_a,a_j,j_h)$
  for every $g(n_a,a_j,j_h)\in\mathscr{G}(n_a,m,0)$.
\par

2. $\mathscr{H}(n_b,a_j,j_h)$
   is a c.D family in $(G_m,\rho)$.
\par

3. Let $\cup\mathscr{H}(n_b,a_j,j_h)=H$.
  Then $g[n_b,H]=\cup\mathscr{G}(n_b,a_j,j_h)=g(n_a,a_j,j_h)$.
\par

4.  Let $g(\ell,i',l')\subset G_m$ with $\ell\geq n_b$.
  Then there exists a $g(n_b,i^*,l^*)\in\mathscr{G}(n_b,m,0)$
  such that $g(\ell,i',l')\subset g(n_b,i^*,l^*)$.
\par

5. $\mathscr{G}(n_b,m,0)=
   \cup\{\mathscr{G}(n_b,a_j,j_h):g(n_a,a_j,j_h)\in\mathscr{G}(n_a,m,0)\}$
   is a Ln cover on $G_m$.
\end{cl}

\begin{proof}
Note 2 of Claim 4.1. Then $\mathscr{G}(n_b,a_j,j_h)
   =\cup\{\mathscr{G}(n_b,b^k,k):k\in N\}$ is a c.o.D family in $(X,\rho)$.
  Then $\cup\mathscr{G}(n_b,a_j,j_h)=g(n_a,a_j,j_h)$ by 1 of Claim 4.1
  if we replace $G_m$ with $g(n_a,a_j,j_h)$.
  Then $\mathscr{G}(n_b,a_j,j_h)$ is a Ln cover on $g(n_a,a_j,j_h)$ by 1,2 and 4 of Claim 4.1.
  Then 2 is a corollary of 1.
\par

To prove 3, replace $\cup\mathscr{H}(n,m,0)=H^1_m$ with $\cup\mathscr{H}(n_b,a_j,j_h)=H$
  in 3 of Claim 4.1. Then $g[n_b,H]=\cup\mathscr{G}(n_b,a_j,j_h)=g(n_a,a_j,j_h)$.
\par

We prove 4 by induction. Note $g(\ell,i',l')\subset G_m$ for $\ell\geq m$.
  Assume that $g(\ell,i',l')\subset G_m$ with $\ell\geq n_a$
  implies $g(\ell,i',l')\subset g(n_a,a_j,j_h)$ for some
  $g(n_a,a_j,j_h)\in\mathscr{G}(n_a,m,0)$.
\par

Let  $g(\ell,i',l')\subset G_m$ with $\ell\geq n_b$.
  Then there exists a $g(n_a,a_j,j_h)\in\mathscr{G}(n_a,m,0)$
  such that  $g(\ell,i',l')\subset g(n_a,a_j,j_h)$
  by inductive assumption.
  Note $g(\ell,i',l')\subset g(n_a,a_j,j_h)$
  and $\mathscr{G}(n_b,a_j,j_h)$ is a Ln cover on $g(n_a,a_j,j_h)$.
  Then there exists a
  $g(n_b,i^*,l^*)\in\mathscr{G}(n_b,a_j,j_h)\subset\mathscr{G}(n_b,m,0)$
  such that $g(\ell,i',l')\subset g(n_b,i^*,l^*)$ by the definition of Ln cover
  since $\mathscr{G}(n_b,a_j,j_h)$ is a Ln cover of $g(n_a,a_j,j_h)$.
  This implies 4.
\par

Note $\mathscr{G}(n_a,m,0)$ is a c.o.D family by inductive assumption.
  Then $\mathscr{G}(n_b,m,0)$ is a c.o.D family in $(G_m,\rho)$.
  So $\mathscr{G}(n_b,m,0)$ is a Ln cover on $(n_b,G_m)$.
\end{proof}

Note \AE \hbox{ and }\OE.
  We have completed induction on $b\in N$.
  Let $n_b=n_a+1$ and

  \[\mathcal{G}(n,m,L)=\{\mathscr{G}(n_i,m,0): n_i\geq n\}.\]
  Then we have proved the following claim.

\begin{cl}
1. $\mathcal{G}(n,m,L)$ is  a Ln covers sequence on $G_m$.
\par

2. Every $\mathscr{G}(n_b,m,0)$ is a Ln cover on $G_m$.
\par

3. Every $\mathscr{H}(n_b,m,0)$ is a c.D family in $(G_m,\rho)$.
\par

4. For every $g(n_a,i,j)\in\mathscr{G}(n_a,m,0)$,
  there exists a family
  $\mathscr{G}(n_b,i,j)\subset\mathscr{G}(n_b,m,0)$,
  $\cup\mathscr{G}(n_b,i,j)=g(n_a,i,j)$ and
  $[\cup(\mathscr{G}(n_b,m,0)-\mathscr{G}(n_b,i,j))]
  \cap g(n_a,i,j)=\emptyset$.
\par

5. Let $H(\ell,i',l')\cap G_m\neq\emptyset$ with $\ell\geq n_b$.
  Then there exists a $g(n_b,i^*,l^*)\in\mathscr{G}(n_b,m,0)$
  such that $g(\ell,i',l')\subset g(n_b,i^*,l^*)$.
\end{cl}
This completes Construction 4.
\par

\vspace{0.5cm}

In the following, we describe a Ln cover by coordinates.
  To do it let  $\mathscr{G}(n,m,0)\in\mathcal{G}(n,m,L)$.
  Then $\cup\mathscr{G}(n,m,0)=G_m$.
  Note, by Operation O1, $\mathscr{H}(n,m,0)=\cup_h\mathscr{H}^*(1,h)$
  and $\mathscr{H}^*(1,h)=
   \{H(n,1_h,h_l)\in\mathscr{H}(n,h):H(n,1_h,h_l)\subset G^h_m\}$.
\par

A. Take $\mathscr{H}^*(1,1)$. Note the definition of $g(n,i,l)$.
  Let $1_1$ be the least number and $a_1=n+1_1$.
  Take $g(n,1_1,1_l)\in\mathscr{O}^*(1,1)$. Let $1'=1_1$.
  Then \[g(n,1_1,1_l)=[q_{1_1},q_{1_1a_1})\times...\times[q_{1_n},q_{1_na_1})\times
  [q_{l_1},q_{l_1a_1})\times...\times[q_{l_{1'}},q_{l_{1'}a_1})\times X_{a_1}.\]
  Then $\mathscr{O}^*(1,1)=\{g(n,1_1,1_l):l\in N\}$.
\par

B. Assume $\mathscr{O}^*(1,k)=\{g(n,1_k,k_l):l\in N\}$ with $a_k=n+1_k$.
 Take $\mathscr{H}^*(1,h)$ for $h=k+1$ in Operation O1. Let $h'=n+1_h$,
 $g(n,1_h,h_l)\in\mathscr{O}^*(1,h)$ and $a_h=n+1_h$.
 Then $a_h>a_k>a_1$ and
 \[g(n,1_h,h_l)=[q_{h_1},q_{h_1a_h})\times...\times[q_{h_n},q_{h_na_h})\times
  [q_{\ell_1},q_{\ell_1a_h})\times...\times[q_{\ell_{h'}},q_{\ell_{h'}a_h})\times X_{a_h}.\]
  Then $\mathscr{O}^*(1,h)=\{g(n,1_h,h_l):l\in N\}$ is c.o.D in $(X,\rho)$.
  Then, by induction on $N$, we have $\mathscr{O}^*(1,h)$ for every $h\in N$.
\par

Take $Q^*_{a_h}$ from Proposition 2.2. Note $m^2=m+a^*$ for
  $G_m=g(m,a^*,y_0)$. Let $j\leq n$ and
  \[\mathscr{I}(j,G_m)=\{[q_{h_j},q_{h_ja_h}):q_{h_j}\in Q^*_{a_h}\cap [q_{i_j},q_{i_jm^2})\}.\]
  Then $\cup\mathscr{I}(n,G_m)=[q_{i_j},q_{i_jm^2})$ and $\mathscr{I}(j,G_m)$ is c.o.D in $(Q,\rho)$ for $j<n$ by 3 of Proposition 2.2.
\par

Let $g(n,1_h,h_l)\in\mathscr{O}^*(1,h)$,
   $J(n,1_h,h_l)=[q_{\ell_1},q_{\ell_1a_h})\times...\times
  [q_{\ell_{h'}},q_{\ell_{h'}a_h})\times X_{a_h}$ and
  \[\mathscr{J}(1_h,G_m)=\{J(n,1_h,h_l):l\in N\}.\]
  Then, for every $h\in N$, we have $\cup\mathscr{J}(1_h,G_m)=J(m,a^*,y_0)$
  since \[G_m=I(m,a^*,y_0)\times J(m,a^*,y_0).\]
Then we have the following Proposition.

\begin{prop}
Let $m^2=m+a^*$ and $\mathscr{G}(n,m,0)=\cup_h\mathscr{G}^*(1,h)$ be a Ln cover on $G_m$
  and $G_m=[q_{i_1},q_{i_1m^2})\times...\times[q_{i_m},q_{i_mm^2})\times
  [q_{l_1},q_{l_1m^2})\times...\times[q_{l_{1'}},q_{l_{1'}m^2})\times X_{m^2}.$
\par

1. Let $\mathscr{J}(1_h,G_m)=\{J(n,1_h,h_l):g(n,1_h,h_l)\in\mathscr{G}^*(1,h)\}$
  for every $h\in N$.
  Then $\cup\mathscr{J}(1_h,G_m)=J(m,a^*,y_0)$ and $\mathscr{J}(1_h,G_m)$
  is c.o.D in $(J(m,a^*,y_0),\rho)$.
\par

2.  Let $j\leq m$ and
  $\mathscr{I}(j,G_m)=\{[q_{h_j},q_{h_ja_h}):g(n,1_h,h_l)\in\mathscr{G}(n,m,0)\}.$
  Then $\cup\mathscr{I}(j,G_m)=[q_{i_j},q_{i_jm^2})$
  for $j\leq m$,
  and $\mathscr{I}(j,G_m)$ is c.o.D in $(Q,\rho)$.
\end{prop}

\textbf{Note 4.6.} When $g(n_a,i,j)\in\mathscr{G}(n_a,m,0)$
  with $b>a$, let
\begin{eqnarray}
  \mathscr{G}(n_b,i,j)=\{g(n_b,i',j')\in\mathscr{G}(n_b,m,0):
  g(n_b,i',j')\subset g(n_a,i,j)\} \hbox{ and }\nonumber\\
  \mathscr{H}(n_b,i,j)=\{H(n_b,i',j')\in\mathscr{H}(n_b,m,0):
  g(n_b,i',j')\in\mathscr{G}(n_b,i,j)\}.\quad\nonumber
\end{eqnarray}
  Then $\mathscr{G}(n_b,i,j)$ is a Ln cover on
  $(n_b,g(n_a,i,j))$.
 \textit{Both $\mathscr{G}(n_b,i,j)$ and $\mathscr{H}(n_b,i,j)$
  are used  always throughout this paper
  for $g(n_a,i,j)$}.
  Then we may replace $G_m$ with $g(n_a,i,j)$,
  and $\mathscr{G}(n_b,i,j)=\cup_h\mathscr{G}(n_b,h)$
  in Proposition 4.5. We have the following corollary.

\begin{cor}
Let $\mathscr{G}(n_b,i,j)=\cup_h\mathscr{G}(n_b,h)$ be a Ln cover on $g(n_a,i,j)$.
\par

1. Let $\mathscr{J}(1_h,n_b)=\{J(n_b,1_h,h_l):g(n_b,1_h,h_l)\in\mathscr{G}(n_b,h)\}$
  for every $h\in N$.
  Then $\cup\mathscr{J}(1_h,n_b)=J(n_a,i,j)$ and $\mathscr{J}(1_h,n_b)$
  is c.o.D in $(J(n_a,i,j),\rho)$.
\par

2.  Let $l\leq n_a$, $a_h=n_b+1_h$, $c=n_a+i$ for $g(n_a,i,j)$ and
  $\mathscr{I}(l,g(n_a,i,j))=\{[q_{h_l},q_{h_la_h}):g(n_b,1_h,h_l)\in\mathscr{G}(n_b,h)\}.$
  Then $\cup\mathscr{I}(l,g(n_a,i,j))=[q_{i_l},q_{i_lc})$
  for $l\leq n_a$, and $\mathscr{I}(l,g(n_a,i,j))$ is c.o.D in $(Q,\rho)$.
\end{cor}

\textbf{Note 4.7.} Note $\mathscr{G}(n_b,i,j)=\cup_h\mathscr{G}(n_b,h)$
  if $g(n_a,i,j)\in\mathscr{G}(n_a,m,0)$
  with $b>a$ by Corollary 4.6. Let
\begin{eqnarray}
  \mathscr{G}(n_b,i,j,H_a)=\{g(n_b,1_h,h_l)\in\mathscr{G}(n_b,i,j):
  H(n_b,1_h,h_l)\subset H(n_a,i,j)\} \hbox{ and }\nonumber\\
  \mathscr{H}(n_b,i,j,H_a)=\{H(n_b,1_h,h_l)\in\mathscr{H}(n_b,i,j):
  g(n_b,1_h,h_l)\in\mathscr{G}(n_b,i,j,H_a)\}.\quad\nonumber
\end{eqnarray}
  Then $\mathscr{G}(n_b,i,j,H_a)$ is a Ln cover of
  $H(n_a,i,j)$.
  Let \[\mathscr{G}(n_b,H_a,1_h)=\{g(n_b,1_h,h_l)\in\mathscr{G}(n_b,h):
  H(n_b,1_h,h_l)\subset H(n_a,i,j)\} \hbox{ and }\]
  \[\mathscr{H}(n_b,H_a,1_h)=\{H(n_b,1_h,h_l):
  g(n_b,1_h,h_l)\in \mathscr{G}(n_b,H_a,1_h)\}.\]

\begin{cor}
1. $\mathscr{G}(n_b,i,j,H_a)=\cup_h\mathscr{G}(n_b,H_a,1_h)$
  is a Ln cover of $H(n_a,i,j)$.
\par

2. $\mathscr{G}(n_b,H_a,1_h)$ is a Ln cover of $H(n_a,i,j)\cap H(n_b,1_h)$
  for every $h\in N$.
\end{cor}

We use Ln covers sequence
 $\mathcal{G}(n,m,L)=
 \{\mathscr{G}(n_i,m,0): n_i\geq n\}$ on $G_m$
 to prove the following proposition.

\begin{prop}
Let $H(n,i^*,l^*)\subset G_m$. Then there exists an
 $a\in N$ and a family
 $\mathscr{G}(a,i^*,l^*)\subset\mathscr{G}(a,m,0)$
 such that $\cup\mathscr{H}(a,i^*,l^*)\subset
 H(n,i^*,l^*)\subset\cup\mathscr{G}(a,i^*,l^*)$.
\end{prop}

\begin{proof}
 Let $H(n,i^*,l^*)\cap g(n,h,k)\neq\emptyset$,
 $H(n,i^*,l^*)\cap H(n,h,k)=\emptyset$ and
 \[\mathscr{G}(\ell,n,h,k)=\{g(\ell,h',k')
 \in\mathscr{G}(\ell,m,0):g(\ell,h',k')\cap H(n,h,k)\neq\emptyset\}.\]
We prove that there exists an $a_h$ with
 $H(n,i^*,l^*)\cap[\cup\mathscr{G}(\ell,n,h,k)]=\emptyset$
 for each $\ell\geq a_h$.
\par

At first, $H(n,i^*,l^*)\cap g(n,h,k)\neq\emptyset$ and
 $H(n,i^*,l^*)\cap H(n,h,k)=\emptyset$
 implies $i^*>h$ by 6$'$ of Proposition 3.3.
 Note that $H(n,i^*,l^*)\cap g(n,h,k)\neq\emptyset$
 implies $g(n,i^*,l^*)\subset g(n,h,k)$ by 4 of Proposition 3.3.
 Let $\mathscr{G}(n,f)=\{g(n,h,k)\in\mathscr{G}_n:
 H(n,i^*,l^*)\cap g(n,h,k)\neq\emptyset
 \hbox{ and } H(n,i^*,l^*)\cap H(n,h,k)=\emptyset\}$.
 Then $\mathscr{G}(n,f)=\{g(n,h_i,k_i): h_i<i^*\}$  is finite
 by 7 and 8 of Proposition 3.3.
\par

 And then, $\pi_j[H(n,i^*,l^*)]=\{q_{i^*_j}\}$ and
  $\pi_j[H(n,h,k)]=\{q_{h_j}\}$ with $q_{i^*_j}\neq q_{h_j}$
  for some $j\leq n$.
  Here $\pi_j$ is a projection from $X$ to $Q_j=Q$.
\par

 Suppose $q_{i^*_j}=q_{h_j}$ for each $j\leq n$.
 Then $P(n,i^*)=P(n,h)$.
 Then $H(n,i^*,l^*)\cap H(n,h,k)=\emptyset$
 implies $J(n,i^*,l^*)\cap J(n,h,k)=\emptyset$.
 So $g(n,i^*,l^*)\cap g(n,h,k)=\emptyset$,
 a contradiction to $H(n,i^*,l^*)\cap g(n,h,k)\neq\emptyset$.
\par

 Let $q_{i^*_j}\neq q_{h_j}$.
   Then $q_{i^*_j}\notin [q_{h_j}, q_{h_ja_i})$
   for some $a_i$.
   Pick an $\ell\geq a_i$.
   Let $x\in H(\ell,h',k')\subset H(n,h,k)$.
   Then $H(n,i^*,l^*)\cap g(\ell,x)
   =H(n,i^*,l^*)\cap g(\ell,h',k')=\emptyset$
   since $q_{i^*_j}\notin[q_{h_j}, q_{h_ja'})
   =\pi_j[g(\ell,x)]$ with $a'=\ell+h'>a_i$.
   Then $H(n,i^*,l^*)\cap[\cup\mathscr{G}(\ell,n,h,k)]=\emptyset$
 for each $\ell\geq a_i$.
\par

Let $a=\hbox{max}\{a_i: h_i<i^*\}$.
  Take $\mathscr{G}(a,n,i^*,l^*)$.
  Then $H(n,i^*,l^*)\subset\cup\mathscr{G}(a,n,i^*,l^*)$
  since $\mathscr{G}(a,m,0)\in\mathcal{G}(n,m,L)$.
  Take a $g(a,i,l)\in\mathscr{G}(a,n,i^*,l^*)$.
  Then $g(a,i,l)\cap H(n,i^*,l^*)\neq\emptyset$.
  Let $H(a,i,l)\subset H(n,i_a,l_a)$.
  Then $g(n,i_a,l_a)\cap H(n,i^*,l^*)\neq\emptyset$.
  Then $g(n,i^*,l^*)\subset g(n,i_a,l_a)$.
\par

  Suppose $H(n,i^*,l^*)\cap H(n,i_a,l_a)=\emptyset$.
  Then  $g(n,i_a,l_a)=g(n,h_i,k_i)\in\mathscr{G}(n,f)$
  for some $h_i<i^*$.
  Then,  for each $\ell\geq a_i$,
  \[H(n,i^*,l^*)\cap[\cup\mathscr{G}(\ell,n,i_a,l_a)]=
  H(n,i^*,l^*)\cap[\cup\mathscr{G}(\ell,n,h_i,k_i)]=\emptyset.\]
 Then $H(n,i^*,l^*)\cap[\cup\mathscr{G}(a,n,i_a,l_a)]=\emptyset$
 for $a\geq a_i$.
 Note $H(a,i,l)\subset H(n,i_a,l_a)$
 and $g(a,i,l)\in\mathscr{G}(a,m,0)$
 since $g(a,i,l)\in\mathscr{G}(a,n,i^*,l^*)$.
 Then $g(a,i,l)\in\mathscr{G}(a,n,i_a,l_a)$.
 Then $g(a,i,l)\cap H(n,i^*,l^*)=\emptyset$,
 a contradiction to $g(a,i,l)\in\mathscr{G}(a,n,i^*,l^*)$
 and $g(a,i,l)\cap H(n,i^*,l^*)\neq\emptyset$.
\par

So $H(n,i^*,l^*)\cap H(n,i_a,l_a)\neq\emptyset$.
 Then $H(n,i^*,l^*)=H(n,i_a,l_a)$.
 Then, for each $g(a,i,l)\in\mathscr{G}(a,n,i^*,l^*)$,
 we have $H(a,i,l)\subset H(n,i^*,l^*)$.
 So $\cup\mathscr{H}(a,i^*,l^*)\subset
 H(n,i^*,l^*)\subset\cup\mathscr{G}(a,n,i^*,l^*)$.
 Then $\mathscr{G}(a,n,i^*,l^*)=\mathscr{G}(a,i^*,l^*)$
 is desired.
\end{proof}

\begin{cor}
Let $H(n^*,i^*,l^*)\subset G_m$ with $n^*\geq n$.
 Then there exists a $k\in N$ and a family
 $\mathscr{G}(k,i^*,l^*)\subset\mathscr{G}(k,m,0)$
 such that $\cup\mathscr{H}(k,i^*,l^*)\subset
 H(n^*,i^*,l^*)\subset\cup\mathscr{G}(k,i^*,l^*)$.
\end{cor}

\begin{proof}
Let $H(n^*,i^*,l^*)\subset G_m$.
  Then $H(n^*,i^*,l^*)\subset\cup\mathscr{G}(n^*,n^*,i^*,l^*)$.
  Then $\mathscr{G}(n^*,n^*,i^*,l^*)=\{g(n^*,i',l')\}$
  by 5 of Claim 4.4.
  If $H(n^*,i',l')=H(n^*,i^*,l^*)$,
  then $\mathscr{G}(n^*,n^*,i^*,l^*)$
  is desired.
  If $H(n^*,i',l')\neq H(n^*,i^*,l^*)$,
  then $H(n^*,i',l')\cap H(n^*,i^*,l^*)=\emptyset$
  and $H(n^*,i^*,l^*)\subset g(n^*,i',l')$.
  We
  consider Ln covers sequence on $G_m\cap g(n^*,i',l')$:
 \[\mathscr{G}(n^*,m,0)|g(n^*,i',l'),\hbox{  }
 \mathscr{G}(n^*+1,m,0)|g(n^*,i',l'),\hbox{  }
 \mathscr{G}(n^*+2,m,0)|g(n^*,i',l'),....\]
 Then, \quad by Proposition 4.8, \quad  there
 \quad exists a
 \quad $k\in N$ \quad and \quad a \quad family \quad
 $\mathscr{G}(k,n^*,i^*,l^*)|g(n^*,i',l')
 \subset\mathscr{G}(k,m,0)|g(n^*,i',l')$
 such that \[\cup[\mathscr{H}(k,n^*,i^*,l^*)|g(n^*,i',l')]\subset
 H(n^*,i^*,l^*)\subset
 \cup[\mathscr{G}(k,n^*,i^*,l^*)|g(n^*,i',l')].\]
 Note \quad $\mathscr{G}(k,n^*,i^*,l^*)|g(n^*,i',l')\quad
 =\quad\mathscr{G}(k,n^*,i^*,l^*) \quad\subset
  \quad\mathscr{G}(k,m,0)$.
 \quad Then
 $\mathscr{G}(k,n^*,i^*,l^*)|g(n^*,i',l')
 =\mathscr{G}(k,n^*,i^*,l^*)=\mathscr{G}(k,i^*,l^*)$
 is desired.
\end{proof}

Professor Gary Gruenhage went through the primitive paper, and told us
  that some sections are too complicated and too hard to understand.
  So we are going to rewriting the paper from Section 5.

\section{To construct a Ln cover on $g_c(n,i,l)-H$}

In order to construct holes in neighborhoods,
  we have to construct a Ln covers on $g(n,i,l)-H$ also.
\par

Let family $\mathscr{H}$ with $\Delta$-order and
 $\mathscr{H}'\subset\mathscr{H}$.
 Call $\mathscr{H}'$ \textit{ with $\Delta$-order if
 $\mathscr{H}'$ as a subsequence of $\mathscr{H}$.}
 We always use  $\mathscr{H}'$ with $\Delta$-order throughout
 this paper.
\par

\begin{prop}
Fix an $n\in N$ and an $m<n$. Then, for arbitrary
  $g(n,i,l)\in\mathscr{G}_n$,
  there exists  $c\in N$ such that $k>c$ and
  $H(m,k,h)\cap g(n,i,l)\neq\emptyset$
  implies $g(m,k,h)\subset g(n,i,l)$.
\end{prop}

\begin{proof}
1. Let $g(n,i,l)=[q_{i_1},q_{i_1a})\times ...\times
  [q_{i_n},q_{i_na})\times [q_{l_1},q_{l_1a})\times...
  \times[q_{l_i},q_{l_ia})\times X_{n+i}$ and
   $c_1=a$.
  Assume $q_{i_j}\in S_{h'_j}$ for $1\leq j\leq n$,
  and $q_{l_j}\in S_{h'_{n+j}}$ for $1\leq j\leq i$.
  Let $h'=\hbox{max}\{h'_j: 1\leq j\leq n+i\}$.
  Then $q_{i_j}\in S_{h'_j}$ implies
  $q_{i_j}=q_{\ell_jk^2_j}\in S_{h'_j}(q_{\ell_j})$
  for some $q_{\ell_j}\in S_{\ell^*_j}$
  with $\ell_j^*\leq h'-1$.
  So do $q_{l_j}\in S_{h'_{n+j}}$ for $1\leq j\leq i$.
  If $q_{i_j}=q_1=0$ for some $j$, let $k^2_j=0$.
\par

  2. Let $c_2=\hbox{max}\{k^2_j: 1\leq j\leq n+i\}$.
  Then, for each $j$, $q_{\ell_j}\in S_{l_j^*}$
  implies
  $q_{\ell_j}=q_{m_jk^3_j}\in S_{l_j^*}(q_{m_j})$
  for some $q_{m_j}\in S_{m_j^*}$ with
  $m_j^*\leq\l_j^*-1\leq h'-2$.
  If $q_{\ell_j}=q_1=0$ for some $j$, let $k^3_j=0$.
\par

 3. Let $c_3=\hbox{max}\{k^3_j: 1\leq j\leq n+i\}$.
 Note $h'$ is finite. Then, by finite induction at most $h'$ times,
 we have $c_1, c_2,..., c_{h'} $.
 Let $c=\hbox{max}\{c_i: i\leq h'\}$.
 Then $q_{i_j},q_{\ell_j},...,q_{l_j}$ in $Q^*_c$
 for  $j\leq n+i$.
 Let $\mathscr{I}(n,i,l)$ be the family of all
 $[q_{\ell_j},q_{\ell_jc})$'s,...,
 $[q_{i_j},q_{i_jc})$'s and $[q_{l_j},q_{l_jc})$'s
 for  $j\leq n+i$.
 Then $\mathscr{I}(n,i,l)\subset\mathscr{I}_c$
 by 2 of Proposition 2.2.
\par

Let $H(m,k,h)\in\mathscr{H}_m$ with $H(m,k,h)\cap g(n,i,l)\neq\emptyset$, $k>c$,
  $b=m+k$ and
  \[H(m,k,h)=\{q_{k_1}\}\times...\times\{q_{k_m}\}
   \times[q_{h_1},q_{h_1b})\times...
  \times[q_{h_k},q_{h_kb})\times X_b.\]
\par

Case 1. $1\leq j\leq m$.
  $H(m,k,h)\cap g(n,i,l)\neq\emptyset$ implies
  $q_{k_j}\in [q_{i_j},q_{i_ja})$.
  $k>c\geq c_1$ implies $b=m+k>k>c\geq a$.
  If $q_{k_j}=q_{i_j}$, then
  $[q_{k_j},q_{k_jb})\subset [q_{i_j},q_{i_jc})
  \subset [q_{i_j},q_{i_ja})$.
  If $q_{k_j}\neq q_{i_j}$,
  then $q_{k_j}\in (q_{i_j},q_{i_ja})$.
  Then $[q_{k_j},q_{k_jb})\subset [q_{i_j},q_{i_jc})
  \subset [q_{i_j},q_{i_ja})$
  by 1 of Proposition 3.2.
\par

\textbf{Case 2.} $m<j\leq n$. Then $b=m+k>k>c\geq a$.
 Let $[q_{h_j},q_{h_jb})\cap[q_{i_j},q_{i_ja})\neq\emptyset$.
 Suppose $q_{h_j}\notin[q_{i_j},q_{i_ja})$.
 Then $q_{h_j}<q_{i_j}$ or $q_{i_ja}<q_{h_j}$.
\par

\textbf{Subcase 2.1.} $q_{i_ja}<q_{h_j}$. Then
 $[q_{i_j},q_{i_ja})\cap[q_{h_j},q_{h_jb})=\emptyset$.
 It is a contradiction to
 $[q_{h_j},q_{h_jb})\cap[q_{i_j},q_{i_ja})\neq\emptyset$.
\par

\textbf{Subcase 2.2.} $q_{h_j}<q_{i_j}$.
 (We want to know $q_{h_jb}<q_{i_ja}$ or $q_{i_ja}<q_{h_jb}$.)
\par

(a). $q_{i_j}=q_{\ell_jk^2_j}\in S_{h'_j}(q_{\ell_j})$
  for some $q_{\ell_j}\in S_{\ell_j^*}$.
  Then $q_{\ell_j}<q_{i_j}$.
\par

  And then $q_{\ell_j}=q_{m_jk^3_j}\in S_{l^*_j}(q_{m_j})$
  for some $q_{m_j}\in S_{m_j^*}$. Then $q_{m_j}<q_{\ell_j}$.
\par

(b). Note that $q_{i_j}\in S_{h'_j}$ and $h'$ is finite. Let $q_{h_j}\in S_{h_j^*-1}$.
\par

If $h_j^*-1\geq h'_j$, then $q_{h_j}<q_{i_j}$ implies
$[q_{h_j},q_{h_j,0})\cap[q_{i_j},q_{i_j,0})=\emptyset$ by (1) and (2) of B in Construction 1.
It is a contradiction to $[q_{h_j},q_{h_jb})\cap[q_{i_j},q_{i_ja})\neq\emptyset$.
So $h_j^*-1<h'_j$ by Construction 1
 since $q_{i_j}\in S_{h'_j}$ and $q_{h_j}\notin[q_{i_j},q_{i_ja})$.
 Then $h_j^*\leq h'_j$ and $q_{h_jb}\in S_{h_j^*}(q_{h_j})$.
 Then there is a $p\in N$ such that $p<h'$ and $c\geq c_p\geq k^p_j$,
 and there is a $q_{m'_j}\in S_{h_j^*-1}$ such that
 $q_{m'_j}<q_{m'_jk^p_j}=q_{m_j}<...<q_{\ell_j}<q_{i_j}$.
\par

\textbf{Subcase (b1).} $q_{h_j}\neq q_{m'_j}$. Then $q_{h_j0}\leq q_{m'_j}$
  by (2) of B in Construction 1 since both $q_{h_j}$ and $q_{m'_j}$ are in
  $S_{h_j^*-1}$ with $q_{h_j}<q_{i_j}$.
  Then $[q_{h_j},q_{h_jb})\cap[q_{i_j},q_{i_ja})=\emptyset$
  since $q_{h_j0}\leq q_{m'_j}<q_{m'_jk^p_j}=q_{m_j}<...<q_{\ell_j}<q_{i_j}$.
  It is a contradiction to $[q_{h_j},q_{h_jb})\cap[q_{i_j},q_{i_ja})\neq\emptyset$.
\par

\textbf{Subcase (b2).} $q_{h_j}=q_{m'_j}$. Note $b>c\geq c_p\geq k^p_j$.
  Then it is easy to see that
  $q_{h_jb}<q_{h_jk^p_j}=q_{m'_jk^p_j}=q_{m_j}<...<q_{\ell_j}<q_{i_j}$.
  Then $[q_{h_j},q_{h_jb})\cap[q_{i_j},q_{i_ja})=\emptyset$.
  It is a contradiction to $[q_{h_j},q_{h_jb})\cap[q_{i_j},q_{i_ja})\neq\emptyset$.
\par

Summarizing Subcase 2.1, Subcase (b1) and Subcase (b2), then $q_{h_j}<q_{i_j}$
 implies $q_{h_jb}<q_{i_j}$.
 It is a contradiction to $[q_{h_j},q_{h_jb})\cap[q_{i_j},q_{i_ja})\neq\emptyset$.
\par

Then we have  $q_{h_j}\in[q_{i_j},q_{i_ja})$.
   Then $b>a$ implies $[q_{h_j},q_{h_jb})\subset [q_{i_j},q_{i_ja})$.
   This complete a proof of Case 2.
\par
We use the above Case 2 throughout this paper.
\par

Case 3, $n<j\leq n+i$. Note $b=m+k>k>c>a$.
  So $[q_{h_j},q_{h_jb})\cap[q_{l_j},q_{l_ja})\neq\emptyset$
  implies  $[q_{h_j},q_{h_jb})\subset[q_{l_j},q_{l_jc})
  \subset [q_{l_j},q_{l_ja})$ by 5 of Proposition 2.2.
\par

Case 4, $a=n+i<j\leq m+k$. Then $[q_{h_j},q_{h_jb})\subset [0,1)$.
\end{proof}

\begin{cor}
Fix an $n\in N$ and an $m<n$. Then for arbitrary
  $g(n,i,l)\in\mathscr{G}_n$,
  there exists at most finitely many $k$'s such that
  $H(m,k,h)\cap g(n,i,l)\neq\emptyset$
  and $g(m,k,h)-g(n,i,l)\neq\emptyset$.
\end{cor}

\textbf{Note*}. By Proposition 5.1 and 5.2, we can construct
  a Ln cover $\mathscr{G}(m,i,l)$
  on $g(n,i,l)-\cup_{i\leq \ell}H(m_i,k_i)$
  with $m_i<n$  in the following Construction 5.
  Let $H(n,i,l)\subset...\subset H(1,i_1,l_1)$ with
  $H(n,i,l)\subset g(n,i,l)$. Then, by Proposition 5.1,
  there is a $c\in N$ such that $g(1,k,h)\subset g(n,i,l)$
  if $H(1,k,h)\cap g(n,i,l)\neq\emptyset$ with $k>c$.
  Note $H(m,k,h)=\{q_{k_1}\}\times...\times\{q_{k_m}\}
   \times[q_{h_1},q_{h_1b})\times...
  \times[q_{h_k},q_{h_kb})\times X_b.$
  Let
  \[\pi^a_{a+b}[H(m,k,h)]=\pi_a[H(m,k,h)]\times...\times\pi_b[H(m,k,h)].\]
  Then $\pi^1_m[H(m,k,h)]=\{q_{k_1}\}\times...\times\{q_{k_m}\}=\{(q_{k_1},...,q_{k_m})\}$. Denote
  $\pi^1_m[H(m,k,h)]$ by $(q_{k_1},...,q_{k_m})$ if $\pi^1_m[H(m,k,h)]$ is a point.
  Then we have
  \[\pi^2_{m+k}[H(m,k,h)]=\{q_{k_2}\}\times...\times\{q_{k_m}\}\times...
   \times[q_{h_1},q_{h_1b})\times...
  \times[q_{h_k},q_{h_kb})\quad\hbox{ and }\]
  \[\pi^{m+1}_{m+k}[H(m,k,h)]=[q_{h_1},q_{h_1b})\times...
  \times[q_{h_k},q_{h_kb}).\]
  Let $I_c(n,i,l)=[q_{i_1},q_{i_1c})\times...\times[q_{i_n},q_{i_nc})$,
  $J(n,i,l)=[q_{l_1},q_{l_1a})\times...\times[q_{l_i},q_{l_ia})
  \times X_a$ and
  $g_c(n,i,l)=I_c(n,i,l)\times J(n,i,l)$.

\begin{cor}
Let $\ell\leq m\leq n$, $g(n,i,l)\in\mathscr{G}_n$ and
 $g_c(n,i,l)=I_c(n,i,l)\times J(n,i,l)$.
\par

 1. If $H(\ell,k,h)\cap g_c(n,i,l)\neq\emptyset$ with $k>c$,
    then $g(\ell,k,h)\subset g_c(n,i,l)$.
\par

 2. If $H(m,k,h)\cap g_c(n,i,l)\neq\emptyset$ with
 $H(m,k,h)\cap H(\ell,i_\ell,l_\ell)=\emptyset$
 and $H(n,i,l)\subset H(\ell,i_\ell,l_\ell)$,
 then $g(m,k,h)\subset g_c(n,i,l)$ and
  \begin{align}
 g_c(n,i,l)-&H(\ell,i_\ell,l_\ell)\nonumber\\
           =&\cup\{H(m,k,h):H(m,k,h)\cap H(\ell,i_\ell,l_\ell)=\emptyset
                   \hbox{ and } H(n,i,l)\subset H(\ell,i_\ell,l_\ell)\}\nonumber\\
            =&\cup\{g(m,k,h):H(m,k,h)\cap H(\ell,i_\ell,l_\ell)=\emptyset
                   \hbox{ and } H(n,i,l)\subset H(\ell,i_\ell,l_\ell)\}.\nonumber
\end{align}
\end{cor}

\begin{proof}
To check 1 let $H(\ell,k,h)\cap g_c(n,i,l)\neq\emptyset$  with
  $k>c$ for arbitrary $\ell\in N$.
  Let $j\leq\ell$. Then $\pi_j[H(\ell,k,h)]
  =q_{k_j}\in[q_{i_j},q_{i_jc})$.
  Then $b=\ell+k>k>c$ implies
  $[q_{k_j},q_{k_jb})
  \subset[q_{i_j},q_{i_jc})$
  by 3 of Proposition 3.2 for $j\leq\ell$.
  Then, in the same way as Proposition 5.1,
  we can prove $g(\ell,k,h)\subset g_c(n,i,l)$
  by four Cases in the proof of Proposition 5.1. This implies 1.
\par

\textit{proof of 2.} Let $H(m,k,h)\cap g_c(n,i,l)\neq\emptyset$,
  $\ell\leq m\leq n$ and $H(m,k,h)\subset H(\ell,k_\ell,h_\ell)$
  for some $H(\ell,k_\ell,h_\ell)$.
  Note $g_c(n,i,l)\subset g(n,i,l)\subset g(\ell,i_\ell,l_\ell)$ and
  $H(\ell,k_\ell,h_\ell)\cap H(\ell,i_\ell,l_\ell)=\emptyset$.
  Then we have $H(\ell,k_\ell,h_\ell)\cap
  g(\ell,i_\ell,l_\ell)\neq\emptyset$.
  Then $k_\ell>i_\ell$ by 6$'$ of Proposition 3.3.
  Then $k\geq k_\ell>i_\ell$
  since $H(m,k,h)\subset H(\ell,k_\ell,h_\ell)$
  by 4 of Proposition 3.1.
  Then, by $k_\ell>i_\ell$, we have \[\pi^1_\ell[H(m,k,h)]=(q_{k'_1},q_{k'_2},...,q_{k'_\ell})\neq
  (q_{i'_1},q_{i'_2},...,q_{i'_\ell})=\pi^1_\ell[H(\ell,i_\ell,l_\ell)].\]
  Then there exists a $j\leq\ell$ with $q_{k'_j}\neq q_{i'_j}$.
  Then, for $j\leq\ell$,
  $\pi_j[H(\ell,k_\ell,h_\ell)]=q_{k'_j}\neq q_{i'_j}$
  and $\pi_j[H(\ell,k_\ell,h_\ell)]=q_{k'_j}\in(q_{i'_j},q_{i'_jc})$
  since $H(\ell,k_\ell,h_\ell)\cap g(\ell,i_\ell,l_\ell)\neq\emptyset$.
  Then $k_\ell\geq k'_j$ by 5 of Proposition 3.1,
  and $k'_j>c$ by 1 of Proposition 3.2.
  Then $g(\ell,k_\ell,h_\ell)\subset g_c(n,i,l)$
  since $k_\ell\geq k'_j>c$ and
  $H(\ell,k_\ell,h_\ell)\cap g_c(n,i,l)\neq\emptyset$
  by the above 1.
  Then $g(m,k,h)\subset g(\ell,k_\ell,h_\ell)\subset g_c(n,i,l)$
  by 5 of Proposition 3.3
  since $H(m,k,h)\subset H(\ell,k_\ell,h_\ell)$.
\end{proof}

Call  $c$ \textbf{on $ (g(n,i,l), m)$.}

To construct a Ln cover of $g(n,i,l)-H$,
 we take $H(n,i,l)\subset g(n,i,l)$.

\vspace{0.3cm}
\textbf{Construction 5.}
\vspace{0.3cm}

\textsf{Conditions.} 1. $H(n,i,l)\subset...\subset
 H(m,k,h)\subset...
 \subset H(2,i_2,l_2)\subset H(1,i_1,l_1)$.
\par

2.  $H(n,i,l)\subset g(n,i,l)$.
\par

\begin{align}
 \hbox{Let} \quad
\mathscr{H}^*(m,n,1)=\{H(m,k',h')\in\mathscr{H}_m:
    H(m,k',h')& \cap g(n,i,l)\neq\emptyset\hbox{ and } \nonumber\\
   & g(m,k',h')-g(n,i,l)\neq\emptyset\}.
  \quad\nonumber
\end{align}
 Then $H(m,k,h)\in\mathscr{H}^*(m,n,1)$ since $H(n,i,l)\subset H(m,k,h)$.  Let
 \[\mathscr{H}(m,k'_i)=\{H(m,k'_i,i_l)\in\mathscr{H}^*(m,n,1): H(m,k'_i,i_l)\subset
  H(m,k'_i)\}\]  and $H'(m,k'_i)=\cup\mathscr{H}(m,k'_i)$. Then
 $\mathscr{H}'(m,n,1)=\{H'(m,k'_j): j\leq k(m)\}$
 is finite by Corollary 5.2.
 Let \[H'(m,k_j,h_j)=H'(m,k'_j)\cap g(n,i,l)\quad
 \hbox{for}\quad 1\leq j\leq k(m),\]
 \[\mathscr{H}(m,n,f)=\{H'(m,k_j,h_j): 1\leq j\leq k(m)\},\]
 \[O(m,n,1)=g(n,i,l)-\cup\mathscr{H}(m,n,f),\]
  \[\mathscr{G}(m,n,1)=\{g(m,1'_i,i_h)\in\mathscr{G}_m:
 g(m,1'_i,i_h)\subset O(m,n,1)\} \quad\hbox{ and} \]
 \[\mathscr{H}(m,n,1)=\{H(m,1'_i,i_h)\in\mathscr{H}_m:
 g(m,1'_i,i_h)\in\mathscr{G}(m,n,1)\}.\]
 Then $\cup\mathscr{G}(m,n,1)=
 \cup\mathscr{H}(m,n,1)=O(m,n,1)$ and
 $Cl_{\rho}(\cup\mathscr{G}(m,n,1))=g(n,i,l)$ by Corollary 5.2.

\vspace{0.3cm}
\begin{fact}$\cup\mathscr{G}(m,n,1)=
   \cup\mathscr{H}(m,n,1)=O(m,n,1)$.
\end{fact}

\begin{proof}
Pick an $x\in g(n,i,l)-\cup\mathscr{H}(m,n,f)$.
  Then there exists an $H(m,k^*,l^*)\in\mathscr{H}_m$
  with $x\in H(m,k^*,l^*)$.
  Then $H(m,k^*,l^*)\cap g(n,i,l)\neq\emptyset$.
  Then we have $H(m,k^*,l^*)\subset g(n,i,l)\subset g(m,k,h)$
  since $H(m,k^*,l^*)\cap[\cup\mathscr{H}(m,n,f)]
  =\emptyset$.
  Suppose $g(m,k^*,l^*)\cap H(m,h_j,k_j)\neq\emptyset$
  for some $H(m,h_j,k_j)\subset H'(m,h_j,k_j)\in\mathscr{H}(m,n,f)$.
  Then $h_j>k^*$ by $6'$ of Proposition 3.3
  and $g(m,h_j,k_j)-g(n,i,l)\neq\emptyset$
  by the definition of $\mathscr{H}(m,n,f)$.
  Then $g(m,h_j,k_j)\subset g(m,k^*,l^*)$
  by 4 of Proposition 3.3.
  Then $g(m,k^*,l^*)-g(n,i,l)\neq\emptyset$,
  a contradiction to $H(m,k^*,l^*)\notin\mathscr{H}(m,n,f)$.
  So $g(m,k^*,l^*)\cap[\cup\mathscr{H}(m,n,f)]
  =\emptyset$.
  Then $g(m,k^*,l^*)\subset O(m,n,1)$.
\end{proof}
 We construct a Ln cover
 $\mathscr{G}(m,n,*)$ on $(m,O(m,n,1))$
 such that $\cup\mathscr{G}(m,n,*)=\cup\mathscr{G}(m,n,1)$
 by induction.
 Note $1'_1<1'_2<...<1'_i<...$.
\par

A. Take the least number $1'_1$ and denote $1'_1$ by $1_1$.
  Let \[\mathscr{G}(m,n,1_1)=\{g(m,1_1,1_e):
  g(m,1_1,1_e)\in\mathscr{G}(m,n,1)\}=\{g(m,1_1,1_e): e\in N\}.\]
  Then $\mathscr{G}(m,n,1_1)$
  is c.o.D family in $(X,\rho)$
  by 7 of Proposition 3.3.
  Let \[\mathscr{H}(m,n,1_1)=\{H(m,1_1,1_e): e\in N\}.\]
\par

B. Assume we have had $\mathscr{G}(m,n,1_i)$ and $\mathscr{H}(m,n,1_i)$
  for $i<k$. Let $O_k=\cup_{i<k}\cup\mathscr{G}(m,n,1_i)$.
\par

  Case 1, $H(m,1'_i,i_h)-O_k=\emptyset$ for each $H(m,1'_i,i_h)
  \in \mathscr{H}(m,n,1)$.
  Let $\mathscr{G}(m,n,*)=\cup_{i<k}\mathscr{G}(m,n,1_i)$
  and $\mathscr{H}(m,n,*)=\cup_{i<k}\mathscr{H}(m,n,1_i)$.
\par

  Case 2, $H(m,1'_i,i_h)-O_k\neq\emptyset$ for some $H(m,1'_i,i_h)
  \in\mathscr{H}(m,n,1)$.
  Let \[1_k=\hbox{min}\{1'_i:H(m,1'_i,i_h)-O_k\neq\emptyset\}.\]
  Note, for some $g(m,1_j,\ell)\in\cup_{i<k}\mathscr{G}(m,n,1_i)$,
  $H(m,1_k,k_h)\cap g(m,1_j,\ell)\neq \emptyset$ implies
  $H(m,1_k,k_h)\subset g(m,1_k,k_h)\subset g(m,1_j,\ell) \subset O_k$,
  a contradiction.  Then $H(m,1_k,k_h)-O_k\neq\emptyset$
  implies $H(m,1_k,k_h)\cap O_k=\emptyset$.
  Note $1_k>1_j$ for each
  $g(m,1_j,\ell)\in\cup_{i<k}\mathscr{G}(m,n,1_i)$.
  Then $g(m,1_k,k_h)\cap O_k=\emptyset$
  by 8 of Proposition 3.3.
\par

Let $\mathscr{G}(m,n,1_k)=\{g(m,1_k,k_h):
  H(m,1_k,k_h)-O_k\neq\emptyset\}$.
  Then $\mathscr{G}(m,n,1_k)$ is a c.o.D family in $(X,\rho)$
  and $[\cup\mathscr{G}(m,n,1_k)]\cap O_k=\emptyset$.
  Let $\mathscr{H}(m,n,1_k)=\{H(m,1_k,k_h): h\in N\}$.
\par

Then, by induction, we have c.o.D family
 $\mathscr{G}(m,n,1_k)$ and
 c.D family $\mathscr{H}(m,n,1_k)$ for each $k$.
 Let \[\mathscr{G}(m,n,*)=\cup_k\mathscr{G}(m,n,1_k)
 \quad\hbox{ and }\quad
 \mathscr{H}(m,n,*)=\cup_k\mathscr{H}(m,n,1_k).\]
 Then $\cup\mathscr{G}(m,n,*)=\cup\mathscr{G}(m,n,1)=
  O(m,n,1)=g(n,i,l)-\cup\mathscr{H}(m,n,f)$.
 Let $H=g(n,i,l)\cap[\cup\mathscr{H}(m,n,f)$.
 Then, by 1 of Proposition 3.5,
 $H\subset Cl_\rho[\cup\mathscr{G}(m,n,*)]$.
\par

Summarizing the above A and B, we have proved the following claim.

\begin{cl}
1. $\mathscr{G}(m,n,*)=\cup_k\mathscr{G}(m,n,1_k)$ is a c.o.D family in $(O(m,n,1),\rho)$
 and $g(m,h,k)$ is a c.o set in $(X,\rho)$
 for each $g(m,h,k)\in\mathscr{G}(m,n,*)$.
\par

2. $\mathscr{H}(m,n,*)=\cup_k\mathscr{H}(m,n,1_k)$ is a c.D family in $(O(m,n,1),\rho)$.
\par

3. $\mathscr{H}(m,n,f)=\{H'(m,k_j,h_j): 1\leq j\leq k(m)\}$ is finite.
\par

4. $\cup\mathscr{G}(m,n,*)=\cup\mathscr{G}(m,n,1)=
  O(m,n,1)=g(n,i,l)-\cup\mathscr{H}(m,n,f)$.
\par

5. Let $g(\ell,h,k)\subset O(m,n,1)$ with $\ell\geq m$.
  Then there exists a $g(m,i',l')\in\mathscr{G}(m,n,*)$
  such that $g(\ell,h,k)\subset g(m,i',l')$.
\par

6. Let $H=g(n,i,l)\cap[\cup\mathscr{H}(m,n,f)]$.
  Then $H\subset Cl_\rho[\cup\mathscr{G}(m,n,*)]$.
\end{cl}

Call the above process Construction 5 on $g(n,i,l)$.

Call $\mathscr{G}(m,n,*)$  \textbf{a least number cover ( Ln cover)
  on $(m,O(m,n,1))$ ( or on $O(m,n,1)$ )} if and only if
  $\mathscr{G}(m,n,*)$ and $O(m,n,1)$
  satisfy 1 and 4 - 5 of Claim 5.4.

\begin{cor}
There exists a Ln covers sequence $\mathcal{G}(n,m,L)$ on $O(m,n,1)$
   such that:
\par

1. Every $\mathscr{G}(\ell,m,*)\in\mathcal{G}(n,m,L)$ is a Ln cover on $O(m,n,1)$.
\par

2. Every $\mathscr{H}(\ell,m,*)$ is a c.D family in $(O(m,n,1),\rho)$.
\end{cor}

\vspace{0.3cm}
\textbf{Construction 5.1.}
\vspace{0.3cm}

 \textsf{Conditions.}
  $\{\mathscr{G}(\ell,n,0),\mathscr{G}(\ell+1,n,0)\}$ is a Ln covers
  sequence
  on $O(m,n,1)$.

  To introduce Operation $\ell^1_2$th, take $\mathscr{G}(\ell,n,0)$,
  and a  $g(\ell,i_h,h_j)\in\mathscr{G}(\ell,n,0)$.
  Let $\ell=n_a$ and $d=\ell+1=n_b$. Then we have
 \begin{eqnarray}
  \mathscr{H}'(d,i_h,h_j)=
  \{H(d,i_k,h_k)\in\mathscr{H}(d,n,0):
  H(d,i_k,h_k)\subset H(\ell,i_h,h_j)\} \hbox{ and }\nonumber\\
  \mathscr{G}'(d,i_h,h_j)
  =\{g(d,i_k,h_k)\in\mathscr{G}(d,n,0):
  H(d,i_k,h_k)\in\mathscr{H}'(d,i_h,h_j)\}.\quad\quad\nonumber
  \end{eqnarray}
 Then, by Corollary 4.7, we have the following claim.

\begin{cl}
1. $\mathscr{G}'(d,i_h,h_j)=\cup_u\mathscr{G}(d,i_h,e_u)$
  is a Ln cover of $H(\ell,i_h,h_j)$.
\par

2. $\mathscr{G}(d,i_h,e_u)$ is a Ln cover of $H(\ell,i_h,h_j)\cap H(d,k_e)$
  for every $e\in N$.
\par

3. $\mathscr{H}'(d,i_h,h_j)$ is infinite and $\mathscr{H}'(d,i_h,h_j)$
  has a $\Delta$-order as a subsequence
  of $\mathscr{H}_d$.
\end{cl}

\begin{proof}
1 and 2 is a corollary  of Corollary 4.7 for $n_a=\ell$ and $n_b=\ell+1=d$.
\par

To prove 3, note the definition of $H(n,i,l)$.
  Then $\pi_b[H(\ell,i_h,h_j)]=[0,1)$
  if $b=d+i_k>\ell+i_h=a$ by 4 of Proposition 3.1
  and $\pi_b[H(d,i_k,h_k)]=[q_{h'_b},q_{h'_bb})$.
  Let $\mathscr{J}^d_b=\{\pi_b[H(d,i_k,h_k)]:
  H(d,i_k,h_k)\in\mathscr{H}'(d,i_h,h_j)\}$.
  Then $\mathscr{J}^d_b$ is a c.o.D family
  in $([0,1),\rho')$ and
  $\cup\mathscr{J}^d_b=[0,1)$ by definitions of
  $H(d,i_k,h_k)$ and $H(\ell,i_h,h_j)$.
  So $\mathscr{J}^d_b$ is infinite.
  Then $\mathscr{H}'(d,i_h,h_j)$ is infinite.
  Give $\mathscr{H}'(d,i_h,h_j)$ a $\Delta$-order
  as a subsequence of $\mathscr{H}_d$. Then,
  for $d=\ell+1$ and $g(\ell,i_h,h_j)\in\mathscr{G}(\ell,n,0)$,
  we have
  \[\mathscr{H}'(d,i_h,h_j)=\{H(\ell^1_2,x_i):i\in N\}.\]
\end{proof}

$\Re1$.
   \textit{Condition:} $\{\mathscr{G}(\ell,n,0),\mathscr{G}(\ell+1,n,0)\}$ is
  a Ln covers sequence on $G_m$  and $g(\ell,i_h,h_j)\in\mathscr{G}(\ell,n,0)$.
\par
\vspace{0.3cm}
  Let
  \[\mathscr{H}_d(d,i_h,h_j)=
  \{H(\ell^1_2,x_i)\in\mathscr{H}'(d,i_h,h_j):i=1\}=\{H(\ell^1_2,x_1)\}.\]
  Then $\cup\mathscr{H}_d(d,i_h,h_j)\subset H(\ell,i_h,h_j)$
  and $\mathscr{G}_d(d,i_h,h_j)\subset \mathscr{G}(\ell+1,n,0)$
  by 2 of Claim 5.7.
  Let $\mathscr{G}_d(\ell+1,n)=\cup\{\mathscr{G}_d(d,i_h,h_j):
  g(\ell,i_h,h_j)\in\mathscr{G}(\ell,n,0)\}$.
\par

Call  $\mathscr{G}_d(\ell+1,n)$
  \textbf{ 1'th Ln family}.
\vspace{0.3cm}

Recall $\mathscr{A}$ a \textit{ mad } family on $N$
  in page 115 of Van Douwen \cite{dou}
  if it is a maximal pairwise almost disjoint subfamily of $[N]^\omega$.
  $\alpha, \beta, \delta$ and $\gamma$ are used to
  denote members in $\mathscr{A}$.
  Take an $\alpha$ from $\mathscr{A}$.
  Then
  $\alpha=\{\alpha(1), \alpha(2),...\}
  \subset N$.

\textbf{$\Re2$.} %d=dig=¨ª¨²%
Let $m+j=m_j$ and $\ell+1=m_2$.
Take $H(m_2,x)=H(\ell^1_2,x)$.
  Let $\mathscr{G}(m_3,x)$ with $\Delta$-order be a Ln cover of $H(m_2,x)$.
  Then \[\mathscr{G}(m_3,x)=\{g(m_3,x^i): i\in N\}.\]
  Let $H(m_3,x^i)=g(m_3,x^i)\cap H(m_2,x)$.
  Then $\cup_iH(m_3,x^i)=H(m_2,x)$.
  Let
  $\mathscr{G}(m_4,x^i)$ with $\Delta$-order be a Ln cover of $H(m_3,x^i)$.
  Then \[\mathscr{G}(m_4,x^i)=\{g(m_4,x^i_j): j\in N\}\]
  Let $H(m_4,x^i_j)=g(m_4,x^i_j)\cap H(m_2,x)$.
  Then \[\cup_jH(m_4,x^i_j)=H(m_3,x^i)\quad\hbox{ and }\quad
    \cup_{ij}H(m_4,x^i_j)=H(m_2,x).\]
  Take a $g(m_4,x^i_j)\in\mathscr{G}(m_4,x^i)$. Let
  $\mathscr{G}(m_5,x^i_j)$ with $\Delta$-order be a Ln cover of $H(m_4,x^i_j)$.
  Then \[\mathscr{G}(m_5,x^i_j)=\{g(m_5,x^i_{jl}): l\in N\}.\]
  Let $H(m_5,x^i_{jl})=H(m_2,x)\cap g(m_5,x^i_{jl})$,
  \[\mathscr{H}(m_5,x^i)=\{H(m_5,x^i_{jl}): j,l\in N\}
  \hbox{ and }\mathscr{G}(m_5,x^i)=\cup\{\mathscr{G}(m_5,x^i_j): j\in N\}.\]
  Then $\cup_lH(m_5,x^i_{jl})=H(m_4,x^i_j)\hbox{ and }
    \cup_{ijl}H(m_5,x^i_{jl})=H(m_2,x).$
\par
\vspace{0.3cm}

Fix an $\alpha=\{\alpha(1), \alpha(2),...\}
  \subset N$ for $\alpha\in\mathscr{A}$.  Let
\[\mathscr{H}_{d1}(m_5,\alpha x^i)=
  \{H(m_5,x^i_{jl})\in\mathscr{H}(m_5,x^i_j):
   l>\alpha(i)\hbox{ for } j\leq \alpha(i)\},\]
   \[\mathscr{H}_{d2}(m_5,\alpha x^i)=
  \{H(m_5,x^i_{jl})\in\mathscr{H}(m_5,x^i_j):
   l\leq\alpha(i)\hbox{ for } j>\alpha(i)\}
  \quad\hbox{ and }\]
  \[\mathscr{H}_d(m_5,\alpha x^i)=\mathscr{H}_{d1}(m_5,\alpha x^i)
  \cup\mathscr{H}_{d2}(m_5,\alpha x^i).\]
  Take $\mathscr{G}_d(m_5,\alpha x^i)=
  \{g(m_5,x^i_{jl})\in\mathscr{G}(m_5,x^i_j):
  H(m_5,x^i_{jl})\in\mathscr{H}_d(m_5,\alpha x^i)\}$.  Let
  \[\mathscr{G}_d(m_5,\alpha x)=\cup\{\mathscr{G}_d(m_5,\alpha x^i): i\in N\}\ \hbox{ and }\
  \mathscr{H}_C(m_5,\alpha x)=\cup\{\mathscr{H}_d(m_5,\alpha x^i): i\in N\}.\]
  Then $H_C(m_5,\alpha x)=\cup\mathscr{H}_C(m_5,\alpha x)\subset H(m_2,x)$.
\par

  Call \textbf{$H_C(m_5,\alpha x)$ a characterization set of
  $\alpha$ in $H(m_2,x)$
  (or $\alpha$ ch-set).} $\Box$
\vspace{0.3cm}
\par

Let
   \[\mathscr{H}_{-1}(\ell+1,n)=\mathscr{H}(\ell+1,n,0)-\mathscr{H}_d(\ell+1,n).\]
  Then both $\mathscr{H}_d(\ell+1,n)$ and
  $\mathscr{H}_{-1}(\ell+1,n)$
  are  discrete families of  closed sets
  in $(O(m,n,1),\rho)$.\ \ Take\ \ the\ relative\ \ $\mathscr{G}_d(\ell+1,n)$
  and\ $\mathscr{G}_{-1}(\ell+1,n)$.\ \
  Then both $\mathscr{G}_d(\ell+1,n)$ and $\mathscr{G}_{-1}(\ell+1,n)$
  are c.o.D families in $(O(m,n,1),\rho)$.

  Let
  $H^1_d=\cup\mathscr{H}_d(\ell+1,n)$ and
  $H_{-1}=\cup\mathscr{H}_{-1}(\ell+1,n)$.
  Then \[g[\ell+1,H^1_d]=\cup\mathscr{G}_d(\ell+1,n),\quad
  g[\ell+1,H_{-1}]=\cup\mathscr{G}_{-1}(\ell+1,n) \quad \hbox{and}\]
  \[g[\ell,H_{-1}]=g[\ell,H^1_d]=O(m,n,1).\]

\par

Summarizing the above Construction 5.1,
 we have proved the following claim.

\begin{cl}
A. If $\{\mathscr{G}(\ell,n,0),\mathscr{G}(\ell+1,n,0)\}$
  is a Ln covers sequence on $O(m,n,1)$, then
  there uniquely exists $1$'th Ln family $\mathscr{G}_d(\ell+1,n)$
  such that:
\par

1.   Both $\mathscr{G}_d(\ell+1,n)$ and
  $\mathscr{G}_{-1}(\ell+1,n)$ are infinite.
\par

2. $\mathscr{G}_d(\ell+1,n)$ is c.o.D, and  $\mathscr{H}_d(\ell+1,n)$ is c.D
 in $(O(m,n,1),\rho)$.
\par

3. Let $H_1=\cup\mathscr{H}(m,n,f)$.
  Then $H_1\subset Cl_\rho[\cup\mathscr{G}_d(\ell+1,n)]$.
\par

4. $H_1\cap O(m,n,1)=\emptyset$ and
  $H_1\cup O(m,n,1)=g(n,i,l)$.
\par

B. Let $g(\ell,i_h,h_j)\in\mathscr{G}(\ell,n,0)$.
\par

 1. $\mathscr{G}'(d,i_h,h_j)\cap
  \mathscr{G}_d(\ell+1,n)=\mathscr{G}_d(d,i_1,h_1)=\{g(d,i_1,h_1)\}=\{g(\ell^1_2,x_1)\}$.
\par

2. $\cup\mathscr{H}_d(\ell+1,i_h,h_j)\subset H(\ell,i_h,h_j)$.
\par

3.  $g[\ell+1,H^1_d]=\cup\mathscr{G}_d(\ell+1,n)$, \quad
  $g[\ell+1,\cup\mathscr{G}_d(\ell+1,n)]=\cup\mathscr{G}_d(\ell+1,n)$,
  \[g[\ell+1,H_{-1}]=\cup\mathscr{G}_{-1}(\ell+1,n)\quad\hbox{ and }\quad
  g[\ell,H_{-1}]=g[\ell,H^1_d]=O(m,n,1).\]
\end{cl}

\section{To construct holes in neighborhoods }
In this section, we use Operation $\ell^1_2$th to construct holes
 families which make bases of neighborhoods are not C.P.
\par

 Take $\mathscr{H}(n,i)$ with $\cup\mathscr{H}(n,i)=H(n,i)$ in Construction 3.2.
   Note, before Construction 4, we have $H_0=H(m,a^*,y_0)$ and $G_m=g(m,a^*,y_0)$.
  Let
  \[n\geq m, \quad H(n,i)\cap H_0=\emptyset,\]
    \[\mathscr{H}'(n,1_i)=\{H(n,1_i,i_l)\in\mathscr{H}(n,1_i): H(n,1_i,i_l)\subset
  H(n,1_i)\cap G_m\},\]
  \[ H'(n,1_i)=\cup\mathscr{H}'(n,1_i),\quad
  1_1<1_2<...<1_i<...\quad \hbox{ and}\]
  \[\mathscr{H}(n,G_m)=\{H'(n,1_i): 1_i\in N(G_m)\}.\]
  Then $\cup\mathscr{H}(n,G_m)=g(m,a^*,y_0)-H(m,a^*,y_0)$.

 \begin{prop}
 Let $\mathscr{G}(n+1,1_1)$ be Ln cover of $H'(n,1_1)$.
 Then there exists an open set $g(n+1,1_2)$ such that:
\par

1.  $g(n+1,1_2)\cap [\cup\mathscr{G}(n+1,1_1)]=\emptyset$.
\par

2. Let $\mathscr{G}(n,1_2)$ be a Ln cover of $H'(n,1_2)$.
  Then there exist infinitely many $g(n,x_j)$'s in $\mathscr{G}(n,1_2)$ such that
  $g(n,x_j)\subset g(n+1,1_2)$.
\par

3. Let $\mathscr{G}(n+1,1_2)=\cup_j\mathscr{G}(n+1,1_2,c_j)$ be a Ln cover on $H'(n,1_2)$.
 Then  $\cup\mathscr{G}(n+1,1_2,c_j)\subset g(n+1,1_2)$ for
 infinitely many $\mathscr{G}(n+1,1_2,c_j)$'s.
\end{prop}

\begin{proof}
  Take $H'(n,1_1)$.
  Then $H'(n,1_1)=H(n,1_1)\cap G_m$ by the definitions of $G_m$ and $H(n,i)$
  in (B) of Construction 3.1.
  Let
  \[H(n,1_1)=\{q_{1_1}\}\times...\times\{q_{1_{n-1}}\}
  \times\{q_{1_n}\}\times X_n.\]
  Note $m^2=m+a^*$, $n\geq m$, $G_m=g(m,a^*,y_0)=I(m,a^*,y_0)\times J(m,a^*,y_0)$
  and
  \[g(n,1_1,1_l)=[q_{1_1},q_{1_1b})\times...\times[q_{1_n},q_{1_nb})\times
  [q_{l_1},q_{l_1b})\times...\times[q_{l_{1'}},q_{l_{1'}b})\times X_b\]
  if $H(n,1_1,1_l)\in\mathscr{H}'(n,1_1)$ with $\cup\mathscr{H}'(n,1_1)=H'(n,1_1)$, $1'=1_1$
  and $b=n+1_1$.
  Let
  \[J_n(m,a^*,y_0)=[q_{l_1},q_{l_1m^2})\times[q_{l_2},q_{l_2m^2})\times...
  \times X_{m^2}.\]
  Note $G_m=[q_{i_1},q_{i_1m^2})\times...\times[q_{i_n},q_{i_nm^2})\times
  J_n(m,a^*,y_0)$. Then
  \[H'(n,1_1)=G_m\cap H(n,1_1)=\{q_{1_1}\}\times...
  \times\{q_{1_n}\}\times J_n(m,a^*,y_0)\hbox{ if }m\leq n<m^2\ \hbox{ and}\]
  \[H'(n,1_1)=G_m\cap H(n,1_1)=\{q_{1_1}\}\times...
  \times\{q_{1_n}\}\times X_n\quad\hbox{if } n\geq m^2.\]

\par
  Let $\mathscr{G}(n,1_1)=\{g(n,x_j): H(n,x_j)\subset H'(n,1_1)\}$
  be a Ln cover of $H'(n,1_1)$ and $H(n,x_j)\in\mathscr{H}(n,1_1)$.
  Then \[H(n,x_j)=H(n,1_1,1_j)\subset H'(n,1_1)=H(n,1_1)\cap G_m.\]
  Then $1_1>a^*$ since $n\geq m$, $H'(n,1_1)\subset G_m=g(m,a^*,y_0)$
  and $H'(n,1_1)\cap H_0=\emptyset$.
  Then $b=n+1_1>m^2=m+a^*$. Let $\pi_n:X\rightarrow Q_n$ be a project map.
  Then $\pi_n[H'(n,1_1)]=q_{1_n}\in[0,1)=\pi_n[G_m]$ for $n>m^2$, or
  \[\pi_n[H'(n,1_1)]=q_{1_n}\in[q_{i_n},q_{i_nm^2})=\pi_n[G_m]
  \quad\hbox{ for }\quad m<n\leq m^2.\]
  Then $[q_{1_n},q_{1_nb})\subset [q_{i_n},q_{i_nm^2})$ by $b>m^2$.
\par

A. Let
  $\mathscr{G}(n+1,1_1)$ be a Ln cover of $H'(n,1_1)$.
  Then, by 1 of Corollary 4.7,
  we have $\mathscr{G}(n+1,1_1)=\cup\{\mathscr{G}(n+1,1'_j):j\in N\}$.
  Take a $\mathscr{G}(n+1,1'_j)$.
  Let $g(n+1,x^i_j)\in\mathscr{G}(n+1,1'_j)$, $d=n+1$ and $e_j=(n+1)+1'_j$.
  Then $e_j>b=n+1_1$,
  \[g(n+1,x^i_j)=[q_{1_1},q_{1_1e_j})\times...\times[q_{1_n},
  q_{1_n e_j})\times[q_{j_d},q_{j_d e_j})\times
   [q_{j^i_1},q_{j^i_1e_j})\times...
  \times[q_{j^i_l},q_{j^i_le_j})\times X_{e_j},\]
  \[I(n+1,x^i_j)=[q_{1_1},q_{1_1e_j})\times...\times[q_{1_n},
  q_{1_n e_j})\times[q_{j_d},q_{j_d e_j}),\]
  \[I_n(n+1,x^i_j)=[q_{1_1},q_{1_1e_j})\times...\times[q_{1_n},
  q_{1_n e_j})\quad\hbox{ and}\]
  \[J(n+1,x^i_j)=
   [q_{j^i_1},q_{j^i_1e_j})\times...
  \times[q_{j^i_l},q_{j^i_le_j})\times X_{e_j}.\]
  Then, for every $g(n+1,x^h_j)\in\mathscr{G}(n+1,1'_j)$,
  we have
  \[I_n(n+1,x^h_j)=I_n(n+1,x^i_j)=I_n(n+1,e_j)
  =[q_{1_1},q_{1_1e_j})\times...\times[q_{1_n},
  q_{1_n e_j})\hbox{ and}\]
  \[I(n+1,x^h_j)=I(n+1,x^i_j)=I(n+1,e_j)=
  [q_{1_1},q_{1_1e_j})\times...\times[q_{1_n},
  q_{1_n e_j})\times[q_{j_d},q_{j_d e_j}).\]
  Let
  $\mathscr{I}(n,1_1)=\{I_n(n+1,x^i_j):  g(n+1,x^i_j)\in\mathscr{G}(n+1,1_1)\}.$
  Then
  \[\mathscr{I}(n,1_1)=\{I_n(n+1,e_j):  j\in N\}.\]
  Note the definition of $e_j$.
  Then $\cap_jI_n(n+1,e_j)=p(n,1_1)=\{q_{1_1}\}\times...\times\{q_{1_n}\}.$
  Let $p(n,1_2)=\{q_{2_1}\}\times...\times\{q_{2_n}\}$
  and $H(n,1_2)=p(n,1_2)\times X_n$.
  Then \[H'(n,1_2)=H(n,1_2)\cap G_m=p(n,1_2)\times J_n(m,a^*,y_0).\]
  Then $p(n,1_2)\neq p(n,1_1)$.
  Then $q_{1_h}\neq q_{2_h}$ for some $h\leq n$.
\par

B. Note $b=n+1_1$. Let $g(n,1_1,b)=I(n,1_1,b)\times J(n,1_1,b)$. Here
  \[I(n,1_1,b)=[q_{1_1},q_{1_1b})\times...\times[q_{1_n},q_{1_nb})
  \hbox{ and }J(n,1_1,b)=[q_{l_1},q_{l_1b})\times...
  \times[q_{1_b'},q_{1_b'b})\times X_b.\]
  Let $c=n+1_2$. Then $c>b=n+1_1$. Let $g(n,1_2,c)=I(n,1_2,c)\times J(n,1_2,c)$.
  Here
  \[I(n,1_2,c)=[q_{2_1},q_{2_1c})\times...\times[q_{2_n},q_{2_nc})
  \hbox{ and }J(n,1_2,c)=[q_{l'_1},q_{l'_1c})\times...
  \times[q_{l_c'},q_{l_c'c})\times X_c.\]
\par

Case 1.  There exists an $h\leq n$ with
  $[q_{1_h},q_{1_hb})\cap [q_{2_h},q_{2_hc})=\emptyset$.
  Then we have $I(n,1_1,b)\cap I(n,1_2,c)=\emptyset$.
  Then it is easy to see $I_n(n+1,x^i_j)\cap I(n,1_2,c)=\emptyset$ if
  $g(n+1,x^i_j)\in\mathscr{G}(n+1,1_1)$
  with $\pi_h[g(n+1,x^i_j)]=[q_{1_h},q_{1_he_j})\subset[q_{1_h},q_{1_hb})$.
  Then $g(n+1,x^i_j)\cap g(n,1_2,c)=\emptyset$.
\par

Case 2. $[q_{1_h},q_{1_hb})\cap [q_{2_h},q_{2_hc})\neq\emptyset$ for each $h\leq n$.
  Note $c>b>m^2$. Then, by 3 of Proposition 3.2,
  $[q_{1_h},q_{1_hb})\supset[q_{2_h},q_{2_hc})$ for each $h\leq n$.
  Note $p(n,1_2)\neq p(n,1_1)$.
  Then there exists an $h\leq n$ with $q_{2_h}\in(q_{1_h},q_{1_hb})$.
  Let
  \[\mathscr{I}(n,1_1,\in)=\{I_n(n+1,e_j):
  q_{2_h}\in (q_{1_h},q_{1_he_j})=\pi_h[I_n(n+1,x^i_j)]-\{q_{1_h}\}\}\quad\hbox{ and}\]
   \[\mathscr{I}(n,1_1,\notin)=\{I_n(n+1,e_j): q_{2_h}\notin (q_{1_h},q_{1_he_j})=\pi_h[I_n(n+1,x^i_j)]-\{q_{1_h}\}\}.\]
   Then $\mathscr{I}(n,1_1,\in)\cap\mathscr{I}(n,1_1,\notin)=\emptyset$
   and $\mathscr{I}(n,1_1,\in)\cup\mathscr{I}(n,1_1,\notin)=\mathscr{I}(n,1_1)$.
   Then $\mathscr{I}(n,1_1,\in)$ is finite since
   $q_{2_h}\notin \cap_j\pi_h[I_n(n+1,e_j)]=\{q_{1_h}\}$.
   Then we have
   $I_n(n+1,x^i_j)\cap I(n,1_2,c)=\emptyset$ for every $I(n+1,x^i_j)\in\mathscr{I}(n,1_1,\notin)$.
   Let
  \[\mathscr{I}(n+1,1_1,\in)=\{I(n+1,e_j): I_n(n+1,e_j)\in\mathscr{I}(n,1_1,\in)\},\]
  \[\mathscr{G}(n+1,1_1,\in)=\{g(n+1,x^i_j)\in\mathscr{G}(n+1,1_1):
  I(n+1,e_j)\in\mathscr{I}(n+1,1_1,\in)\},\]
  \[\mathscr{I}(n+1,1_1,\notin)=\{I(n+1,e_j):
  I_n(n+1,e_j)\in\mathscr{I}(n,1_1,\notin)\}\quad\hbox{ and}\]
  \[\mathscr{G}(n+1,1_1,\notin)=\{g(n+1,x^i_j)\in\mathscr{G}(n+1,1_1):
  I(n+1,e_j)\in\mathscr{I}(n+1,1_1,\notin)\}.\]
   Then $g(n+1,x^i_j)\cap g(n,1_2,c)=\emptyset$
   for every $g(n+1,x^i_j)\in\mathscr{G}(n+1,1_1,\notin)$.
  Then we have the following claim:
\par

\begin{cl}
1. Both $\mathscr{I}(n,1_1,\in)$ and $\mathscr{I}(n+1,1_1,\in)$ are finite.
\par

2. $g(n+1,x^i_j)\cap g(n,1_2,c)=\emptyset$ for every $g(n+1,x^i_j)\in\mathscr{G}(n+1,1_1,\notin)$.
\end{cl}
\vspace{0.3cm}

C. Note  $c>b>m^2$. Take $\mathscr{I}_b$ from 1 of Proposition 2.2.
  Let $\mathscr{I}_{bm^2}=\mathscr{I}_b|[q_{l_1},q_{l_1m^2})$.
  Then $\mathscr{I}_{bm^2}$ is infinite.
   Let
\begin{align}
\mathscr{F}_{bm^2}(1)=\{[q_{i_d},q_{i_db})\in\mathscr{I}_{bm^2}:
 & \hbox{ there exits an } I(n+1,e_j)\in\mathscr{I}(n+1,1_1,\in)\nonumber\\
 &\hbox{ with } \pi_d[I(n+1,e_j)]=[q_{j_d},q_{j_de_j})\subset
  [q_{i_d},q_{i_db})\}.\quad\nonumber
\end{align}
\begin{cl}
1. $\mathscr{F}_{bm^2}(1)$ is finite.
\par
2. $\mathscr{I}_{bm^2}-\mathscr{F}_{bm^2}(1)$ is infinite.
\par

3.  There exists a $q_{l_1p}\geq\hbox{max}[\cup\mathscr{F}_{bm^2}(1)]$ with $p>a=m^2=m+a^*$
  such that $[q_{l_1p},q_{l_1a})\cap[\cup\mathscr{F}_{bm^2}(1)]=\emptyset$.
\end{cl}
\vspace{0.3cm}

D. Note $J_n(m,a^*,y_0)=[q_{l_1},q_{l_1m^2})\times[q_{l_2},q_{l_2m^2})\times...
  \times X_{m^2}$.
  Let  \[O(b,m^2)=\cup(\mathscr{I}_{bm^2}-\mathscr{F}_{bm^2}(1)),\quad
  J_{n+1}(m,a^*,y_0)=[q_{l_2},q_{l_2m^2})\times...
  \times X_{m^2}\quad\hbox{ and}\]
  \[g(n+1,1_2)=I(n,1_2,c)\times O(b,m^2)\times J_{n+1}(m,a^*,y_0).\]
  Then $g(n+1,1_2)\subset G_m$, and
  \[H'(n,1_2)\cap g(n+1,1_2)=\{q_{2_1}\}\times...
  \times\{q_{2_n}\}\times O(b,m^2)\times J_{n+1}(m,a^*,y_0).\]
  Let $J(n+1,1_2)=O(b,m^2)\times J_{n+1}(m,a^*,y_0)$.
  Then \[g(n+1,1_2)=I(n,1_2,c)\times J(n+1,1_2).\]
  Let $g(n+1,x)\in\mathscr{G}(n+1,1_1,\in)$.
  Note $d=n+1$.
  Then \[\pi_d[g(n+1,x)]=[q_{j_d},q_{j_de_j})\subset
  [q_{j'_d},q_{j'_db})\in\mathscr{F}_{bm^2}(1).\]
  Then  $J(n+1,x)\cap J(n+1,1_2)=\emptyset$.
  Then $g(n+1,x)\cap g(n+1,1_2)=\emptyset$.
  Let $g(n+1,x)\in\mathscr{G}(n+1,1_1,\notin)$.
  Then $I(n+1,x)=I(n+1,e_j)$
  and $I_n(n+1,e_j)\cap I(n,1_2,c)=\emptyset$.
  Then $g(n+1,x)\cap g(n+1,1_2)=\emptyset$.
\par

\begin{cl}
$g(n+1,1_2)\cap[\cup\mathscr{G}(n+1,1_1)]=\emptyset$.
\end{cl}

\vspace{0.3cm}

E. Let $\mathscr{G}(n+1,1_2)$ be a Ln cover of $H'(n,1_2)$.
  Then, by 1 of Corollary 4.7,
  $\mathscr{G}(n+1,1_2)=\cup_j\mathscr{G}(n+1,1_2,c_j)$. Note $d=n+1$.
  Then
  \[O(b,m^2)\subset[q_{l_1},q_{l_1m^2})=\pi_d[H'(n,1_2)]
  \subset\pi_d[\cup\mathscr{G}(n+1,1_2)].\]
  Let $[q_h,q_{hb})\in\mathscr{I}_{bm^2}-\mathscr{F}_{bm^2}(1)$.
  Then $q_h\in[q_h,q_{hb})\subset O(b,m^2)\subset\pi_d[\cup\mathscr{G}(n+1,1_2)]$.
  Then there exists a $g(n+1,2_j,y_h)\in\mathscr{G}(n+1,1_2,c_j)$ with
  \[g(n+1,2_j,y_h)=[q_{2_1},q_{2_1c_j})\times...\times[q_{2_n},q_{2_n c_j})
  \times[q_{l'_1},q_{l'_1 c_j})\times...
  \times[q_{l'_{2'}},q_{l'_{2'}c_j})\times X_{c_j}\]
  and $q_h\in\pi_d[g(n+1,2_j,y_h)]=[q_{l'_1}, q_{l'_1 c_j})$.
  Here $q_{l'_h}\in [q_{l_h},q_{l_hm^2})$ for $1\leq h\leq a^*$,
  $2'=2_j$ and $c_j=(n+1)+2_j>n+1_2=c>b=n+1_1$ by 4 of Proposition 3.1 since
  $H(n+1,2_j,y_h)\subset H'(n,1_2)$.
\par

Note $c_j>b$ implies $Q^*_b\subset Q^*_{c_j}$ by 2 of Proposition 2.2,
  and $q_i\in Q^*_b$ with $[q_i,q_{ib})\in\mathscr{I}_b$
  implies $q_i\in Q^*_{c_j}$ with $[q_i,q_{ic_j})\in\mathscr{I}_{c_j}$
  and $[q_i,q_{ic_j})\subset [q_i,q_{ib})$.
  Then $q_h\in[q_h,q_{hb})\in\mathscr{I}_{bm^2}\subset\mathscr{I}_b$
  implies $q_h\in Q^*_b\subset Q^*_{c_j}$.
  Then, by 3 of Proposition 3.2,
  \[q_h=q_{l'_1}\in[q_{l'_1},q_{l'_1 c_j})=[q_h,q_{h c_j})\subset[q_h,q_{hb}).\]
  Let
  $I(n+1,2_j,y_h)=[q_{2_1},q_{2_1c_j})\times...\times[q_{2_n},q_{2_n c_j})
  \times[q_{l'_1},q_{l'_1 c_j})$ and
  \[J(n+1,2_j,y_h)=[q_{l'_2},q_{l'_2 c_j})\times...
  \times[q_{l'_{2'}},q_{l'_{2'}c_j})\times X_{c_j}.\]
  And then let $\mathscr{I}(c_j,m^2,l_h)=\mathscr{I}_{c_j}|[q_{l_h},q_{l_hm^2})$
  for $2\leq h\leq 2_j=2'$.
  Then, by 3 of Proposition 2.2,
  $\cup\mathscr{I}(c_j,m^2, l_h)=[q_{l_h},q_{l_hm^2})$.
  Let \[\mathscr{J}(n+1,c_j)=\mathscr{I}(c_j,m^2, l'_2)\times...\times\mathscr{I}(c_j,m^2,l'_{2'})
  \times\{X_{c_j}\}.\]
  Then $J(n+1,2_j,y_h)\in\mathscr{J}(n+1,c_j)$ by the definition of $g(n,i,l)$
  in A-C of Construction 3.2.
  Let \[\mathscr{G}^*(n+1,1_2,c_j)=\{I(n+1,2_j,y_h)\times J(n+1,2_j,j_l):
  J(n+1,2_j,j_l)\in\mathscr{J}(n+1,c_j)\}.\]
  Then $g(n+1,2_j,y_h)\in\mathscr{G}^*(n+1,1_2,c_j)$ since $J(n+1,2_j,y_h)\in\mathscr{J}(n+1,c_j)$.
  Then $\mathscr{G}^*(n+1,1_2,c_j)=\mathscr{G}(n+1,1_2,c_j)$
  by the definition of $\mathscr{G}(n,i)$ in C of Construction 3.2 and 2 of Corollary 4.7.
  Then there exist infinitely many $\mathscr{G}(n+1,1_2,c_j)$'s such that
  $\cup\mathscr{G}(n+1,1_2,c_j)\subset g(n+1,1_2)$.
\par

Note Claim 6.4 and the definition of $g(n+1,1_2)$. It is easy to see 1.
\par

Note the definitions of $I(n,1_2,c)$ and $g(n+1,1_2)$, and 2 of Claim 6.3.
  Then it is easy to see 2.
\end{proof}

\begin{prop}
Let $\mathscr{G}(n+1,1_i)$ be Ln cover of $H'(n,1_i)$
 for $i<j$.
 There exists an open set $g(n+1,1_j)$ such that:
\par

1.  $g(n+1,1_j)\cap [\cup\mathscr{G}(n+1,1_i)]=\emptyset$
  for every $i<j$.
\par

2. Let $\mathscr{G}(n,1_j)$ be a Ln cover of $H'(n,1_j)$.
  Then there exist infinitely many $g(n,x_i)$'s in $\mathscr{G}(n,1_j)$ such that
  $g(n,x_i)\subset g(n+1,1_j)$.
\par

3.  Let $\mathscr{G}(n+1,1_j)=\cup_b\mathscr{G}(n+1,1_j,c_b)$ be a Ln cover of $H'(n,1_j)$.
 Then  $\cup\mathscr{G}(n+1,1_j,c_b)\subset g(n+1,1_j)$ for
 infinitely many $\mathscr{G}(n+1,1_j,c_b)$'s.
\end{prop}

\begin{proof}
Let $\mathscr{G}(n+1,1_j)=\cup_b\mathscr{G}(n+1,1_j,c_b)$
  be a Ln cover of $H'(n,1_j)$ with $\Delta$-order.
\par

Fix an $H'(n,1_i)$ for $i<j$. Then, by Claim 6.2,
  there exists $\mathscr{I}(n+1,1_i,\in)$,
  and a set $g(n,1_j,c)$ with $c=n+1_j$
  such that $\mathscr{I}(n+1,1_i,\in)$ is finite,
  and $g(n+1,y)\cap g(n,1_j,c)=\emptyset$
  for every $g(n+1,y)\in\mathscr{G}(n+1,1_i,\notin)$.
  Note $b=n+1_1>m+a^*=m^2$. Let $\mathscr{G}(n+1,1_i)$ be Ln cover of $H'(n,1_i)$.
  Then $\mathscr{G}(n+1,1_i)=\cup_h\mathscr{G}(n+1,1_i,i_h)$
  by 1 of Corollary 4.7.
  Let $g(n+1,i_h,x)=g(n+1,x)\in\mathscr{G}(n+1,1_i,i_h)$.
  Then $e^i_h=(n+1)+i_h>n+1_i>n+1_1=b$.
  Take $\mathscr{F}_{bm^2}=\mathscr{I}_b|[q_{l_1},q_{l_1m^2})$ from 2 of Claim 6.3.
  Let
    \begin{align}
\mathscr{F}_{bm^2}(i)=\{[q_{j_d},q_{j_db})\in\mathscr{I}_{bm^2}:
 \hbox{ there exits a } & [q_{j'_d},q_{j'_de^i_h})
 \in\mathscr{I}(n+1,1_i,\in) \nonumber\\
 &\hbox{ such that }\ \ [q_{j'_d},q_{j'_de^i_h})\subset
  [q_{j_d},q_{j_db})\}.\nonumber
\end{align}
 In the same way, $\mathscr{F}_{bm^2}(i)$  is finite for each $i<j$.
 Then $\cup_{i<j}\mathscr{F}_{bm^2}(i)$ is finite.
 Then $\mathscr{I}_{bm^2}-\cup_{i<j}\mathscr{F}_{bm^2}(i)$
   is infinite.
 Let \[O(b,m^2,j)=\cup[\mathscr{I}_{bm^2}-\cup_{i<j}\mathscr{F}_{bm^2}(i)]
  \quad\hbox{ and}\]
 \[g(n+1,1_j)=I(n,1_j,c)\times O(b,m^2,j)\times [q_{l_2},q_{l_2m^2})\times...
  \times X_{m^2}. \]
Then, in the same way as the proof of C, D and E,
  open set $g(n+1,1_j)$
  satisfies 1, 2 and 3.
\end{proof}

Recall $\mathscr{A}$ a \textit{ mad } family on $N$
  in page 115 of Van Douwen \cite{dou}
  if it is a maximal pairwise almost disjoint subfamily of $[N]^N$.
  $\alpha, \beta, \delta$ and $\gamma$ are used to
  denote members in $\mathscr{A}$.
  Take a $\beta$ from $\mathscr{A}$.
  Then
  $\beta=\{\beta(1), \beta(2),...\}
  \subset N$.
\par
Recall the \textit{ quasi-order } $\leq^*$ on $[N]^N$
 in page 115 of Van Douwen \cite{dou} by
  \[\alpha\leq^*\beta \hbox{ if }
   \alpha(n)\leq\beta(n)\hbox{ for all but finitely
   many } n\in N.\]
\par

A subset  $\mathscr{A}'$  of $\mathscr{A}$ is called
 \textit{ unbounded } if  $\mathscr{A}'$  is unbounded
 in $(\mathscr{A}, \leq^*)$.
\par

Call an $\alpha\in[N]^N$ \textit{ strictly increasing } if $\alpha(i)<\alpha(j)$
  if and only if $i<j$.

\par
\vspace{0.3cm}
\textbf{Construction 6.1.}
\vspace{0.3cm}

\textbf{A.} Note $g(m,a^*,y_0)=G_m$. Let $O(m,y_0)=g(m,a^*,y_0)-H(m,a^*,y_0)$
   and $\mathcal{G}(m,m,L)=\{\mathscr{G}(m_i,m,0): m_i\geq m\}$ be  a Ln covers sequence on $O(m,y_0)$ which satisfies 1 of Corollary 5.6.
\par

Take $\mathscr{H}'(n,y_0)=\mathscr{H}(n,G_m)$  for every $n\geq m$ before Proposition 6.1. Let
  \[\mathcal{H}'(m,m,H)=\{\mathscr{H}'(n,y_0): n\geq m\}.\]
  Take $\mathscr{H}'(m,y_0)\in\mathcal{H}'(m,m,H)$ for $n=m$.
  Then
  \[\mathscr{H}'(m,y_0)=\{H'(m,1_h):h\in N\} \hbox{ and }
   1_1<1_2<...<1_i<... \] with $\Delta$-order in Construction 3.1.
  Then we have the following Condition $B^*_1$.
\par

 \textbf{Condition $B^*_1$:}
  1. $O(m,y_0)=\cup\mathscr{H}'(m,y_0)$.
\par

2. $\mathscr{H}'(m,y_0)=\{H'(m,1_h): h\in N\}$ with
    $ 1_1<1_2<...<1_h<... $.
\vspace{0.3cm}

\textbf{B.}  Take $H'(m,1_j)\in\mathscr{H}'(m,y_0)$.
  Let $m_1=m+1$ and $\mathscr{G}(m_1,1_j)$ be a Ln cover of $H'(m,1_j)$.
  Take a $g(m_1,z)\in\mathscr{G}(m_1,1_j)$. Let $m_2=m_1+1$,
  \[H(m_1,jz)=H'(m,1_j)\cap g(m_1,z)\quad\hbox{and}\]
  \[\mathscr{G}(m_2,z)=\{g(m_2,z_l)\in\mathscr{G}_{m_2}:
  l\in N\}\quad\hbox{with}\quad \Delta\hbox{-order }\]
  be a Ln cover of $H(m_1,jz)$.
  Take an $g(m_2,z_n)\in\mathscr{G}(m_2,z)$.
  Let \[H(m_2,jz_n)=H(m_1,jz)\cap g(m_2,z_n)=H'(m,1_j)\cap g(m_2,z_n).\]
  Let  $m_3=m_1+2$ and
  $\mathscr{G}(m_3,jz_n)$
  be a Ln cover of $H(m_2,jz_n)$. Let
  \[\mathscr{G}(m_3,1_j,z)=
  \cup\{\mathscr{G}(m_3,jz_n): g(m_2,z_n)\in\mathscr{G}(m_2,z)\}.\]
  Then $\mathscr{G}(m_3,1_j,z)$ is a Ln cover of $H(m_1,jz)$. Take an $H(m_2,jz_n)$.
  Let \[\mathscr{G}(m_3,jz_n)=\{g(m_3,z^n_v):  v\in N\}
  \quad\hbox{with}\quad \Delta\hbox{-order}, \]
  \[H(m_3,jz^n_v)=H(m_2,jz_n)\cap g(m_3,z^n_v)=H'(m,1_j)\cap g(m_3,z^n_v)\quad\hbox{ and}\]
  \[\mathscr{H}(m_3,jz_n)=\{H(m_3,jz^n_v):
  v\in N\}.\]
\begin{cl}
1. $\cup_vH(m_3,jz^n_v)=H(m_2,jz_n)$ and $\cup_nH(m_2,jz_n)=H(m_1,jz)$.
\par

2. $\mathscr{G}(m_3,1_j,z)$ is a Ln cover of $H(m_1,jz)$.
\par

3. $\mathscr{G}(m_3,jz_n)$ is a Ln cover of $H(m_2,jz_n)$
  for every $n\in N$.
\end{cl}

\textbf{B1.} Take an $H(m_3,jz^n_v)$.
Let $v^\alpha_0=\alpha(m+v)$, $v^\alpha_1=\alpha(m+v)+1$ and
   \[\mathcal{G}(v^\alpha_0,jz^n_v)
  =\{\mathscr{G}(v^\alpha_0,jz^n_v),
  \mathscr{G}(v^\alpha_1,jz^n_v)\}\]
  be a covers sequence of $H(m_3,jz^n_v)$.
  Then there
  exists $1$'th Ln family $\mathscr{H}^*_d(v^\alpha_1,jz^n_v)$
   with $\Delta$-order by $\Re1$.
   Then $\mathscr{G}^*_d(v^\alpha_1,jz^n_v)$ is  c.o.D in $(X,\rho)$
  by 2 of A in Claim 5.8.
\par

\textbf{B2.} For the same $H(m_3,jz^n_v)$ as B1, take $\mathscr{H}(m_5,x^v)$ and
  $\mathscr{G}(m_5,x^v)$ for $v=i$
  from $\Re2$. Let $x=z_n$ and $x^v=z^n_v$.
  Then, by $\Re2$, there exists an $\alpha$ ch-set $H_C(m_5,\alpha z_n)$,
  $\mathscr{G}(m_5, z^n_v)$  and  $\mathscr{G}_d(m_5,\alpha z^n_v)$.
  Note both $\mathscr{G}(m_5, z^n_v)$ and $\mathscr{G}(v^\alpha_1,jz^n_v)$
  are Ln covers of $H(m_3,jz^n_v)$. Let $G(5,\alpha z^n_v)=\cup\mathscr{G}_d(m_5,\alpha z^n_v)$.
  Then $g(v^\alpha_1,t)\cap G(5,\alpha z^n_v)\neq\emptyset$ implies $g(v^\alpha_1,t)\subset G(5,\alpha z^n_v)$
  if $v^\alpha_1\geq m_5$. Let $v(v^\alpha_5)$ be the least number
  such that $v\geq v(v^\alpha_5)$ implies $v^\alpha_1\geq m_5$.
  Let
  \[\mathscr{G}_d(v^\alpha_1,jz^n_v)=\{g(v^\alpha_1,t)\in\mathscr{G}^*_d(v^\alpha_1,jz^n_v):
  g(v^\alpha_1,t)\cap G(5,\alpha z^n_v)\neq\emptyset\} \quad \hbox{ if }v\geq v(v^\alpha_5),\]
  \[\mathscr{G}_d(v^\alpha_1,jz^n_v)=\mathscr{G}_d(m_5,\alpha z^n_v)
   \quad \hbox{ if }v< v(v^\alpha_5),\]
    \[\mathscr{G}_{-1}(v^\alpha_1,jz^n_v)=\{g(v^\alpha_1,t)\in\mathscr{G}(v^\alpha_1,jz^n_v):
  g(v^\alpha_1,t)\cap[\cup\mathscr{G}_d(v^\alpha_1,jz^n_v)]=\emptyset\},\]
   \[\mathscr{G}_d(n\sigma^\alpha_1,jz,p\sigma)=\cup\{\mathscr{G}_d(v^\alpha_1,jz^n_v):
  \ v\geq p \hbox{ and }g(m_3,z^n_v)\in\mathscr{G}(m_3,jz_n)\}\quad\hbox{ and }\]
  \[\mathscr{G}_{-1}(n\sigma^\alpha_1,jz,p\sigma)
  =\cup\{\mathscr{G}_{-1}(v^\alpha_1,jz^n_v):
  v\geq p \hbox{ and }g(m_3,z^n_v)\in\mathscr{G}(m_3,jz_n)\}.\]
    Note that $\mathscr{G}(m_1,1_j)$ is a Ln cover of $H'(m,1_j)$
    for $H'(m,1_j)\in\mathscr{H}'(m,y_0)$. Let
  \[\mathscr{G}_d(n\sigma^\alpha_1,1_j,p\sigma)=
  \cup\{\mathscr{G}_d(n\sigma^\alpha_1,jz,p\sigma):
  g(m_1,z)\in\mathscr{G}(m_1,1_j)\},\]
  \[\mathscr{G}_{-1}(n\sigma^\alpha_1,1_j,p\sigma)
  =\cup\{\mathscr{G}_{-1}(n\sigma^\alpha_1,jz,p\sigma):
  g(m_1,z)\in\mathscr{G}(m_1,1_j)\},\]
  \[\mathscr{G}_d(n\sigma^\alpha_1,p\sigma)=\cup\{\mathscr{G}_d(nj^\alpha_1,1_j,p\sigma):
  H'(m,1_j)\in\mathscr{H}'(m,y_0)\}\quad\hbox{ and}\]
  \[\mathscr{G}_{-1}(n\sigma^\alpha_1,p\sigma)=\cup\{\mathscr{G}_{-1}(nj^\alpha_1,1_j,p\sigma):
  H'(m,1_j)\in\mathscr{H}'(m,y_0)\}.\]
\par

\begin{cl}
Let $H(n\sigma^\alpha_1,p\sigma)=\cup\mathscr{H}_d(n\sigma^\alpha_1,p\sigma)$.
\par

 1. $H(n\sigma^\alpha_1,q\sigma)\subset H(n\sigma^\alpha_1,p\sigma)$ if $q>p$.
\par

 2. $\cap_pH(n\sigma^\alpha_1,p\sigma)=\emptyset$
  and
 $H(m_3,jz^n_u)\cap H(m_3,jz^n_v)=\emptyset$ if $u\neq v$.
\par

3. $\cup\mathscr{H}(n\sigma^\alpha_1,jz,p\sigma)\subset H(m_2,jz_n)$.
\par

4.  Let $x\in H(m_2,jz_n)$ for some $jz_n$.
  Then there uniquely exists $u$ such that
  $x\in g(u^\alpha_1,t)\in
  \mathscr{G}_d(u^\alpha_1,jz^n_u)\cup\mathscr{G}_{-1}(u^\alpha_1,jz^n_u)$.
\par

5. $\mathscr{G}(n\sigma^\alpha_1,1_j,p\sigma)=\mathscr{G}_d(n\sigma^\alpha_1,1_j,p\sigma)
  \cup\mathscr{G}_{-1}(n\sigma^\alpha_1,1_j,p\sigma)$ is c.o.D in $(X,\rho)$.
\end{cl}
\vspace{0.3cm}
\par

\textbf{B3.} Take
  $\mathscr{H}(1\sigma^\alpha_1,p\sigma)=\cup\{\mathscr{H}(1j^\alpha_1,1_j,p\sigma):
  H'(m,1_j)\in\mathscr{H}'(m,y_0)\}$.
  Note $\cup\mathscr{H}_d(v^\alpha_1,jz^1_v)\subset H(m_3,jz^1_v)$
  for every $v,1_j\in N$.
We construct families by induction on $\mathscr{H}'(m,y_0)$.
\vspace{0.3cm}
\par

\textbf{B3.1.}
  Let $H'(m,1_1)\in\mathscr{H}'(m,y_0)$
  and $\mathscr{G}(m_1,1_1)$ be a Ln cover of $H'(m,1_1)$.
  For $j=1$ and $n=1$, take $\mathscr{G}_d(1\sigma^\alpha_1,1_1,p\sigma)$
  and $\mathscr{G}_{-1}(1\sigma^\alpha_1,1_1,p\sigma)$  from B2.
  Let \[H(m,1_1)=H'(m,1_1)-\cup\mathscr{G}_d(1\sigma^\alpha_1,1_1,p\sigma).\]
\par

\textbf{B3.2.} For $n=1$, assume that we have had
 $\mathscr{G}_d(1\sigma^\alpha_1,1_\ell,p\sigma)$,
  $\mathscr{G}_{-1}(1\sigma^\alpha_1,1_\ell,p\sigma)$ and $H(m,1_\ell)$ for each $\ell<j$.
 Take $H'(m,1_j)$. Let
  \[H^*(m,1_j)=H'(m,1_j)-\cup_{\ell<j}\cup\mathscr{G}_d(1\sigma^\alpha_1,1_\ell,p\sigma).\]
  Then $H^*(m,1_j)\neq\emptyset$ by 2 and 3 of Proposition 6.5.
  Let $\mathscr{G}(m_1,1_j)$
  be a Ln covers sequence of $H'(m,1_j)$ and
  \[\mathscr{G}^*(m_1,1_j)=\{g(m_1,z)\in\mathscr{G}(m_1,1_j):
  g(m_1,z)\cap H^*(m,1_j)\neq\emptyset\}.\]
  Take  $\mathscr{G}_d(1\sigma^\alpha_1,jz,p\sigma)$,
  $\mathscr{G}_{-1}(1\sigma^\alpha_1,jz,p\sigma)$ for
  every $g(m_1,z)\in\mathscr{G}^*(m_1,1_j)$ from the above B2.
  Let
  \[\mathscr{G}_d(1\sigma^\alpha_1,1_j,p\sigma)=
  \cup\{\mathscr{G}_d(1\sigma^\alpha_1,jz,p\sigma):
  g(m_1,z)\in\mathscr{G}^*(m_1,1_j)\}¡ê?\]
  \[\mathscr{G}_{-1}(1\sigma^\alpha_1,1_j,p\sigma)
  =\cup\{\mathscr{G}_{-1}(1\sigma^\alpha_1,jz,p\sigma):
  g(m_1,z)\in\mathscr{G}^*(m_1,1_j)\}\quad\hbox{ and}\]
  \[H(m,1_j)=H'(m,1_j)-\cup_{i\leq j}\cup\mathscr{G}_d(1\sigma^\alpha_1,1_i,p\sigma).\]
\vspace{0.3cm}

Then, by induction for every $j\in N$, there exist $\mathscr{G}_d(1\sigma^\alpha_1,1_j,p\sigma)$,
 $\mathscr{G}_{-1}(1\sigma^\alpha_1,1_j,p\sigma)$, $H(m_1,1_j)$,
 $\mathscr{G}^*(m_1,1_j)$,  and
  \[\mathcal{G}(v^\alpha_0,jz^1_v)
  =\{\mathscr{G}(v^\alpha_0,jz^1_v),
  \mathscr{G}(v^\alpha_1,jz^1_v)\} \hbox{ for every } v\in N.\]
    Let $\mathscr{H}(m,G_1)=\{H(m,1_j):j\in N\}$,
    \[\mathscr{G}^*(m_1,G_1)=\cup_j\mathscr{G}^*(m_1,1_j),
    \quad \mathscr{G}_{-1}(1\sigma^\alpha_1,G_1)
    =\{\mathscr{G}_{-1}(1\sigma^\alpha_1,1_j,p\sigma): j\in N\},\]
    \[ \mathscr{G}(1\sigma^\alpha_1,G_1)=\cup_j\mathscr{G}_d(1\sigma^\alpha_1,1_j,p\sigma)
    \quad\hbox{ and }\quad
    H(1\sigma^\alpha_1,G_1)=\cup\mathscr{H}(1\sigma^\alpha_1,G_1).\]

\begin{cl}

 Let $D(1\sigma^\alpha_1,G_1)=\cup\mathscr{G}(1\sigma^\alpha_1,G_1)$
  and $H(m,G_1)=\cup\mathscr{H}(m,G_1)$. Then:

1. $\mathscr{G}_d(1\sigma^\alpha_1,1_j,p\sigma)$ is c.o.D in $(X,\rho)$
  for every $j\in N$.
\par

2.  $D(1\sigma^\alpha_1,G_1)\cup H(m,G_1)=O(m,y_0)$ and
  $D(1\sigma^\alpha_1,G_1)\cap H(m,G_1)=\emptyset$.

3.  $Cl_\rho D(1\sigma^\alpha_1,G_1)=g(m,y_0)$.
\par

4. $H(m,1_j)=H'(m,1_j)-D(1\sigma^\alpha_1,G_1)
  =H'(m,1_j)-\cup_{i\leq j}\cup\mathscr{G}(1\sigma^\alpha_1,1_i,p\sigma)$
  for every $H'(m,1_j)\in\mathscr{H}'(m,y_0)$.
\par

5.  $g[p,H(1\sigma^\alpha_1,G_1)]=O(m,y_0)$ if
  $p=m_1$.
\par

6. Let $t\in g[p,D(1\sigma^\alpha_1,G_1)]-D(1\sigma^\alpha_1,G_1)$.
  Then, there exists a $k$ such that  $H(\ell,t)\cap g[q,D(1\sigma^\alpha_1,G_1)]=\emptyset$
  if $\ell,q>k$.
\end{cl}

\begin{proof}
It is easy to see 1, 2 and 4. \quad   5 follows from 3 of B in Claim 5.8.
\par
\textit{Proof of 3.}
To do it pick an $r\in H(m,1_j)$. Then,
  by the above definition of $H(m,1_j)$ and 4 of Claim 6.7, there uniquely exists
  $g(v^\alpha_1,r')\in\mathscr{G}(v^\alpha_1,jz^1_v)$
  such that $H(v^\alpha_1,r)\subset H(v^\alpha_1,jr')=g(v^\alpha_1,r')\cap H(m,1_j)
  \subset H(m,1_j)$ and
  \[g(v^\alpha_1,r')\cap [\cup_{i\leq j}\cup\mathscr{G}_d(1\sigma^\alpha_1,1_i,p\sigma)]=\emptyset.\]
\par

Let $p>v^\alpha_1$ and $r\in H(p,y')\subset H(v^\alpha_1,jr')$.
  Then there exists a $c$ with
  $g(m,1_\ell,t)\subset g_c(p,y')$ if $1_\ell>c$ and
  $H(m,1_\ell,t)\cap g_c(p,y')\neq\emptyset$
  by 1 of Corollary 5.3.
  Let $1_\ell$ be the least number for $m<p$.
  Then, for every $t\in H'(m,1_\ell)\cap g_c(p,y')$, we have
  \[g(m,1_\ell, t)\subset g_c(p,y')\subset g(p,y').\]
\par

  And then we may prove the following Fact.
\par

\begin{fact}
Let $\mathscr{G}(m_1,1_k)$ be a Ln cover of $H'(m,1_k)$ for
  $H'(m,1_k)\in\mathscr{H}'(m,y_0)$,
  $G(m_1,1_\ell)=\cup\{\cup\mathscr{G}(m_1,1_k): 1_j<1_k<1_\ell\}$ and
  $H(m_1,1_\ell,t)\subset g_c(p,y')$. Then
  $g(m_1,1_\ell,t)\cap G(m_1,1_\ell)=\emptyset$.
\end{fact}
\begin{proof}
Note $1_j<1_\ell$. Let $1_j<1_k<1_\ell$ for some $1_k$.
  Let $g(m_1,1_k,s)\in\mathscr{G}(m_1,1_k)$,
  \[I_c(p,y')=[q_1,q_{1c})\times...\times[q_{m_1},q_{m_1 c}),\quad
  g_c(p,y')=I_c(p,y')\times J(p,y'),\]
  \[I(m_1,1_j,r)=[q_1,q_{1c_j})\times...\times[q_{m_1},q_{m_1 c_j}),\quad
  g(m_1,1_j,r)=I(m_1,1_j,r)\times J(m_1,1_j,r),\]
  \[I(m_1,1_\ell,t)=[q'_1,q'_{1c_\ell})\times...\times[q'_{m_1},q'_{m_1 c_\ell}),\quad
  g(m_1,1_\ell,t)=I(m_1,1_\ell,t)\times J(m_1,1_\ell,t),\]
  \[I(m_1,1_k,s)=[q''_1,q''_{1c_k})\times...\times[q''_{m_1},q''_{m_1 c_k}) \quad\hbox{ and }\]
  \[g(m_1,1_k,s)=I(m_1,1_k,s)\times J(m_1,1_k,s).\]
  Note $H(m_1,1_\ell,t)\subset H'(m_1,1'_\ell)\subset H'(m,1_\ell)$,
  $H(m_1,1_k,s)\subset H'(m_1,1'_k)\subset H'(m,1_k)$ and
  $H(p,y')\subset H(v^\alpha_1,r)\subset H(m_1,1_j,r)\subset H'(m_1,1'_j)\subset H'(m,1_j)$
  such that
  $H'(m_1,1'_k)=[\{q''_1\}\times...\times\{q''_{m_1}\}\times X_{m_1}]\cap g(m,y_0)$ and $g_c(p,y')\subset g(m,y_0)$.
\par

Case 1.  $g(m_1,1_k,s)\cap g_c(p,y')=\emptyset$.
  Then $g(m_1,1_k,s)\cap g(m_1,1_\ell,t)=\emptyset$
  since $g(m,1_\ell,t)\subset g_c(p,y')$.
\par

Case 2.  $g(m_1,1_k,s)\cap g_c(p,y')\neq\emptyset$.
  Note that $1_\ell$ is the least number such that $1_\ell>c$,
  $1_\ell>1_j$ and $H'(m,1_\ell)\cap g_c(p,y')\neq\emptyset$.
  Then $1_j<1_k<1_\ell$ implies \[H'(m,1_k)\cap g_c(p,y')=\emptyset.\]
  Then, for every $H'(m_1,1'_k)\subset H'(m,1_k)$, we have
  $H'(m_1,1'_k)\cap g_c(p,y')=\emptyset$.
  Then $H'(m_1,1'_k)\cap g_c(p,y')=\emptyset$ implies $q''_i\notin[q_i,q_{ic})$
  for some $i\leq m_1$ since $g(m_1,1_k,s)\in\mathscr{H}(m_1,1_k)$
  and $g(m_1,1_k,s)\cap g_c(p,y')\neq\emptyset$.
\par

Note $1_j<1_k$. This implies $c_j=m+1_j<m+1_k=c_k$.
  Note \[H(p,y')\subset H(v^\alpha_1,r)\subset H(m,1_j).\]
  Suppose $g(m_1,1_k,s)\cap g(m_1,1_\ell,t)\neq\emptyset$ for some $1_k<1_\ell$.
  Then $g(m_1,1_\ell,t)\subset g_c(p,y')\subset g(m_1,1_j,r)$
  implies $g(m_1,1_k,s)\cap g(m_1,1_j,r)\neq\emptyset$.
  Then $1_j<1_k$ implies
  $H'(m,1_j)\cap g(m_1,1_k,s)=\emptyset$ by 6 of Proposition of 3.3, and
  implies
  $g(m_1,1_k,s)\subset g(m_1,1_j,r)$ by 8 of Proposition 3.3.
  Then $[q''_b,q''_{bc_k})\subset [q_b,q_{bc_j})$ for each $b\leq m_1$.
  Then $q''_b\in [q_b,q_{bc_j})$ for each $b\leq m_1$,
  and $q''_i\notin[q_i,q_{ic})$
  for some $i\leq m_1$.
  Then $q''_i\in [q_i,q_{ic_j})$ and $q''_i\notin[q_i,q_{ic})$.
  Then $q_{ic}<q''_i$.
  Then $[q_i,q_{ic})\cap[q''_i,q''_{ic_k})=\emptyset$.
  Note $g_c(p,y')\cap g(m_1,1_k,s)\neq\emptyset$.
  Then $[q_i,q_{ic})\cap[q''_i,q''_{ic_k})\neq\emptyset$
  for each $i\leq m_1$, a contradiction.
  Then  $g(m_1,1_k,s)\cap g(m_1,1_\ell,t)=\emptyset$ for each $1_k<1_\ell$.
  Then $G(m_1,1_\ell)\cap g(m_1,1_\ell,t)=\emptyset$.
\end{proof}

\begin{fact}
Let $H(m+1,s)\subset H'(m,1_\ell)$ and $g(m+1,s)\cap g_c(p,y')\neq\emptyset$. Then
  $H(m+1,s)\cap g_c(p,y')\neq\emptyset$,
  $g(m+1,s)\subset g_c(p,y')$ and $g(m,s)\subset g_c(p,y')$.
\end{fact}
\begin{proof}
Note $g(m,a^*,y_0)=I(m,a^*,y_0)\times J(m,a^*,y_0)$.
Let \[J(m,a^*,y_0)=[q_{l_1},q_{l_1m^2})\times[q_{l_2},q_{l_2m^2})\times...
  \times X_{m^2}.\]
  Let $H(m,1_\ell)=\{q_{1_1}\}\times...
  \times\{q_{1_m}\}\times X_m$ and $P(m,1_\ell)=\{q_{1_1}\}\times...
  \times\{q_{1_m}\}$. Then
  \[H'(m,1_\ell)=G_m\cap H(m,1_\ell)=P(m,1_\ell)\times J(m,a^*,y_0).\]
  Let $g_c(p,y')=I_m(p,y')\times J_m(p,y')$.
  Note that $1_\ell$ is the least number such that
  $H'(m,1_\ell)\cap g_c(p,y')\neq\emptyset$.
  Then \[P(m,1_\ell)\in I_m(p,y')=[q_1,q_{1c})\times...\times[q_{m},q_{m c}).\]
  Let $\mathscr{G}(m_1,1_\ell)$ be an Ln cover of $H'(m,1_\ell)$.
  Then there exists a $g(m_1,s)\in\mathscr{G}(m_1,1_\ell)$
  such that $H(m_1,s)\subset H'(m,1_\ell)$ and $g(m_1,s)\cap g_c(p,y')\neq\emptyset$.
  Then we have $H(m_1,s)\subset H(m,s)$ and $g(m,s)\cap g_c(p,y')\neq\emptyset$.
  Note \[g(m,s)=g(m,1_\ell,s)=I(m,1_\ell,s)\times J(m,1_\ell,s).\]
  Then $J(m,1_\ell,s)\cap J_m(p,y')\neq\emptyset$.
  Note $H(m,1_\ell,s)=P(m,1_\ell)\times J(m,1_\ell,s)$ by
  the definition of $H(n,i,l)$  in A of Construction 3.2.
  Then $H(m,1_\ell,s)\cap g_c(p,y')\neq\emptyset$.
  Then $g(m,s)\subset g_c(p,y')$ by the definition $1_\ell$.
  Then $g(m+1,s)\subset g_c(p,y')$.

\end{proof}

\textit{The proof of 3 is continued.}
Let $t^*\in H''(m_1,1_\ell)=H'(m_1,1_\ell)\cap g_c(p,y')$ with
  $H''(m,1_\ell)\cap H(p,y')=\emptyset$,
  and $\mathscr{G}(m_1,1_\ell)$ be an Ln cover of $H'(m,1_\ell)$.
  Then there exists a $g(m_1,x)\in\mathscr{G}(m_1,1_\ell)$
  such that \[t^*\in H''(m,1_\ell)\cap g(m_1,x)
  \subset g_c(p,y')\cap g(m_1,x).\]
  Then $g(m_1,x)\cap g_c(p,y')\neq\emptyset$.
  Then, by Fact 6.10, we have $g(m_1,x)\subset g_c(p,y')$
  and $x\in H'(m,1_\ell)\cap g(m_1,x)\subset H''(m,1_\ell)$.
\par

 Let $\mathscr{G}^{**}(m_1,y'1_\ell)=\{g(m_1,x)\in\mathscr{G}(m_1,1_\ell):
  g_c(p,y')\cap g(m_1,x)\neq\emptyset\}$
  and
  \[\mathscr{G}^*(m_1,y'1_\ell)=\{g(m_1,x)\in\mathscr{G}(m_1,1_\ell):
  g(m_1,x)\subset g_c(p,y')\}.\]
  Then $\mathscr{G}^{**}(m_1,y'1_\ell)=\mathscr{G}^*(m_1,y'1_\ell)\neq\emptyset$.
  Let $\mathscr{G}(m,1_\ell)$ be a Ln cover of $H'(m_1,1_\ell)$ and
   \[\mathscr{G}^*(m,y'1_\ell)=\{g(m,x)\in\mathscr{G}(m,1_\ell):
  g(m,x)\subset g_c(p,y')\}.\]
  Then $\mathscr{G}^*(m,y'1_\ell)\neq\emptyset$ by Fact 6.10.
\par

A. Note $\mathscr{G}_d(1\sigma^\alpha_1,1_i,p\sigma)=
  \cup\{\mathscr{G}_d(1\sigma^\alpha_1,iz,p\sigma):
  g(m_1,z)\in\mathscr{G}^*(m_1,1_i)\}$ for $i\in N$.
  Let $g(m_1,z)\in\mathscr{G}^*(m_1,y'1_\ell)$
  Then $g(m_1,z)\cap [\cup_{i\leq j}
  \cup\mathscr{G}_d(1\sigma^\alpha_1,1_i,p\sigma)]=\emptyset$
  since $g(m_1,z)\subset g_c(p,y')\subset g(v^\alpha_1,r')$
  and $g(v^\alpha_1,r')\cap [\cup_{i\leq j}\cup\mathscr{G}_d(1\sigma^\alpha_1,1_i,p\sigma)]=\emptyset$.
  Note $g(m_1,z)\cap G(m_1,1_\ell)=\emptyset$ for $j<k<\ell$ by Fact 6.9.
  Let $i>\ell$, $\mathscr{G}(m_1,1_i)$ be a Ln cover of $H'(m_1,1_i)$
  and $G(m_1,1_i)=\cup\mathscr{G}(m_1,1_i)$.
  Then,
  by 6 of Proposition 3.3, we have $G(m_1,1_i)\cap [
  \cup\mathscr{H}_d(1\sigma^\alpha_1,1_\ell,p\sigma)]=\emptyset$.
  Note $\mathscr{G}(m_1,1_\ell)$ is an Ln cover of $H'(m,1_\ell)$
  and $\mathscr{G}^*(m_1,y'1_\ell)\subset\mathscr{G}(m_1,1_\ell)$.
  Then $\mathscr{G}_d(1\sigma^\alpha_1,1_\ell,p\sigma)|g(m_1,z)
  =\mathscr{G}_d(1\sigma^\alpha_1,\ell z,p\sigma)$
  for every $g(m_1,z)\in\mathscr{G}^*(m_1,y'1_\ell)$
  by the definition of $\mathscr{H}_d(1\sigma^\alpha_1,\ell z,p\sigma)$
  in B3.2.
\par

B. Let $g(m_1,z)\in\mathscr{G}^*(m_1,y'1_\ell)$.
  Then $\cup\mathscr{G}_d(1\sigma^\alpha_1,\ell z,p\sigma)
  \subset g(m_1,z)\subset g_c(p,y')$.
  Let $G_d(1\sigma^\alpha_1,\ell z,p)
  =\cup\mathscr{G}_d(1\sigma^\alpha_1,\ell z,p\sigma)$.

  Then we have the following fact:
\begin{fact}
1. $\emptyset\neq\mathscr{G}^*(m_1,y'1_\ell)\subset\mathscr{G}^*(m_1,1_\ell)$
  and $\mathscr{G}^*(m,y'1_\ell)\neq\emptyset$.
\par

2. $\mathscr{H}_d(1\sigma^\alpha_1,p\sigma)|H(m_1,z)
  =\mathscr{H}_d(1\sigma^\alpha_1,\ell z,p\sigma)$
  for every $g(m_1,z)\in\mathscr{G}^*(m_1,y'1_\ell)$.
\par

3. $H(m_1,z)-G_d(1\sigma^\alpha_1,\ell z,p)
  \subset  g_c(p,y')$ and
  $G_d(1\sigma^\alpha_1,\ell z,p)\subset g_c(p,y')$
  for every $g(m_1,z)\in\mathscr{G}^*(m_1,y'1_\ell)$.
\end{fact}

Call $H(m_1,z)$ \textbf{full  if $g(m_1,z)\in \mathscr{G}^*(m_1,y'1_\ell)$.}

Pick an arbitrary $x_{dj}\in H(m_1,z)-G_d(1\sigma^\alpha_1,\ell z,p)$.
  Let \[A^*_{d\ell}=\{x_{dz}:x_{dz}\in H(m_1,z)-G_d(1\sigma^\alpha_1,\ell z,p)
  \hbox{ and }g(m_1,z)\in\mathscr{G}^*(m_1,1_\ell)\}\quad\hbox{ and}\]
   \[A^*_d=\{A^*_{d\ell}:H'(m,1_\ell)\in\mathscr{H}'(m,y_0)\}.\]
  Then we have the following fact:
\begin{fact}
Let $H(v^\alpha_1,t)\in\mathscr{H}_{-1}(1\sigma^\alpha_1,p\sigma)$. Then:
\par

1. $H(v^\alpha_1,t)\subset Cl_\rho A^*_d$.
\par

2. $H(v^\alpha_1,t)\subset Cl_\rho[\cup_{\ell>j}H(m,1_\ell)]$,
  $H(v^\alpha_1,t)\subset Cl_\rho[\cup_{i> j}
  \cup\mathscr{H}_d(1\sigma^\alpha_1,1_i,p\sigma)]$
  and
  $H(v^\alpha_1,t)\subset Cl_\rho[\cup_{i> j}\cup\mathscr{G}_d(1\sigma^\alpha_1,1_i,p\sigma)]$.
\end{fact}
\textit{The proof of 3 is continued.}
 Then,  by the above Fact 6.12, $Cl_\rho D(1\sigma^\alpha_1,G_1)=g(m,y_0)$.
\par

\textit{Proof of 6.}
Let $t\in g[p,D(1\sigma^\alpha_1,G_1)]-D(1\sigma^\alpha_1,G_1)$.
  Then there exists an
  $H'(m,1_j)\in\mathscr{H}'(m,y_0)$ such that
  \[t\in H'(m,1_j)\cap[g[p,D(1\sigma^\alpha_1,G_1)]-D(1\sigma^\alpha_1,G_1)].\]
\par

A.  Let $D(1,<)=\cup_{i<j}
  \cup\mathscr{G}_d(1\sigma^\alpha_1,1_i,p\sigma)$ and $D(1,=)=\cup_{i\leq j}
  \cup\mathscr{G}_d(1\sigma^\alpha_1,1_i,p\sigma)$.
  Then, by 4 of Claim 6.8,
  \[t\in H(m,1_j)=H'(m,1_j)-D(1\sigma^\alpha_1,G_1)=H'(m,1_j)-D(1,=)
  \subset H'(m,1_j)-D(1,<).\]
  Then there exits an $H(v^\alpha_1,r)\subset H(v^\alpha_1,jr)
  =g(v^\alpha_1,r)\cap H(m,1_j)\subset H(m,1_j)$
  such that $t\in H(v^\alpha_1,jr)\subset g(v^\alpha_1,r)\in\mathscr{G}_{-1}(v^\alpha_1,jz^1_v)$ and
  $H(v^\alpha_1,jr)\cap D(1,<)=\emptyset$.
  Then $H(v^\alpha_1,t)\subset H(v^\alpha_1,jr)\subset H(m,1_j)$.
\par

 Let $\mathscr{H}(m, 1_j-)
   =\{H'(m,1_l)\in\mathscr{H}'(m,y_0):
   1_l<1_j\}$.
   Then $\mathscr{H}(m, 1_j-)$ is finite.
   Let $H(m, 1_j-)=\cup\mathscr{H}(m, 1_j-).$
   Then $\rho(H(m, 1_j-), H'(m, 1_j))=r>0$
   by 3 of Proposition 3.1.
   Then there exists a $q$ such that
   \[g[q,H(m, 1_j-)]\cap H'(m, 1_j)=\emptyset.\]
   Let $g(v^\alpha_1,s)\in\mathscr{G}_d(1\sigma^\alpha_1,i,p\sigma)$
   for $i<j$.
\par

Case 1, $v^\alpha_1\leq q$. Then $g[q,g(v^\alpha_1, s)]=g(v^\alpha_1,s)$
  and $g(v^\alpha_1,s)\cap H(v^\alpha_1,t)=\emptyset$.
\par

Case 2,  $v^\alpha_1>q$. Then $g[q,g(v^\alpha_1, s)]\cap H'(m, 1_j)=\emptyset$
  since $H(v^\alpha_1, s)\subset H'(m, 1_i)\subset H(m, 1_j-)$
  and $g[q,H(m, 1_j-)]\cap H'(m, 1_j)=\emptyset$.
\par

Then  $g[q,D(1,<)]\cap H(v^\alpha_1,t)=\emptyset$.
 Let $q_1=q$.
\par

B. Let $\mathscr{H}(m, 1_j+)
   =\{H'(m,1_l)\in\mathscr{H}'(m,y_0):
   1_l>1_j\}$ and $H(m, 1_j+)=\cup\mathscr{H}(m, 1_j+).$
   Then $H'(m,1_j)\cap g[m_1,H(m_1, 1_j+)]=\emptyset$.
   Let $q_2>m_1$.
\par

C. Note $\mathscr{G}(m_1,1_j)$ is a Ln cover of $H'(m,1_j)$
 and $\mathscr{G}(m_3,1_j,z)$ is a Ln cover of $H(m_1,jz)$
 for every $g(m_1,z)\in\mathscr{G}(m_1,1_j)$ by 2 of Claim 6.6.
 Then there uniquely exists an $g(m_3,z^1_v)\in\mathscr{G}(m_3,1_j,z)$
 such that $t\in H(m_3,jz^1_v)=g(m_3,z^1_v)\cap H(m_1,jz)$.
 Then $t\in g[p,\cup\mathscr{G}_d(v^\alpha_1,jz^1_v)]-\cup\mathscr{G}_d(v^\alpha_1,jz^1_v)$
 since $\mathscr{G}(m_3,1_j,z)$ is a Ln cover of $H(m_1,jz)$.
 Note, by 3 of B of Claim 5.8, $q\geq v^\alpha_1$ implies
 \[g[q,\cup\mathscr{G}_d(v^\alpha_1,jz^1_v)]=\cup\mathscr{G}_d(v^\alpha_1,jz^1_v).\]
 Then there exists an $\ell>v^\alpha_1$ such that
 $H(l,t)\cap g[q,\cup\mathscr{G}_d(v^\alpha_1,jz^1_v)]=\emptyset$
 if $l,q>\ell$. Let $q_3>\ell$.
\par

   Let $l,q>\hbox{max}\{q_1,q_2,q_3\}$.
   Then $H(l,t)\cap g[q,D(1\sigma^\alpha_1,G_1)]=\emptyset$.
\end{proof}

To calculate $g[p, D(1\sigma^\alpha_1,G_1)]$ for $p\geq m_5$, we prove the following proposition.

\begin{prop} Let $p\geq m_5$,
  $\mathscr{H}'(m,y_0)=\{H'(m,1_h): h\in N\}$ with $O(m,y_0)=\cup\mathscr{H}'(m,y_0)$. Then:
\par

1.  $g[p,D(1\sigma^\alpha_1,G_1)]=D(p,m^\alpha)=\cup_hD(p,m^\alpha,1_h)$.
\par

2. $Cl_\rho D(p,m^\alpha)=g(m,y_0)$.
\par

3. $D(p,m^\alpha)\cap H(p,m^\alpha)=\emptyset$ and
  $D(p,m^\alpha)\cup H(p,m^\alpha)=O(m,y_0)$.
\par

4. Let $G_d(1\alpha,j)=\cup\mathscr{G}_d(1\sigma^\alpha_1,1_j,q\sigma)$
  and $g[p,G_d(1\alpha,j)]=D(p,m^\alpha,1_j)$.
  Then $D(p,m^\alpha,1_j)$ is a c.o set in $(X,\rho)$ for every $j\in N$,
  $D(p,m^\alpha)=\cup_jD(p,m^\alpha,1_j)$
  and
  $H(m,1_j)=H'(m,1_j)-D(p,m^\alpha)=H'(m,l_j)-[\cup_{i\leq j}D(p,m^\alpha,1_i)]$.
\par

5. There exists a c.o.D family $\mathscr{G}_d(p,m^\alpha,1_j)$ with
 $D(p,m^\alpha,1_j)=\cup\mathscr{G}(p,m^\alpha,1_j)$.
\par

6. Let $H(l,t)\subset H(p,m^\alpha)=\cup_jH(m,1_j)$.
  Then
  $H(l,t)\subset Cl_\rho[\cup_{j>h}H(m,1_j)]$
  and
  $H(l,t)\subset Cl_\rho[\cup_{j>h}D(p,m^\alpha,1_j)]$
  for arbitrary $h\in N$.
\end{prop}

\begin{proof}
A. Take $\mathscr{H}_d(1\sigma^\alpha_1,1_j,p\sigma)=
  \cup\{\mathscr{H}_d(1\sigma^\alpha_1,jz,p\sigma):
  g(m_1,z)\in\mathscr{G}^*(m_1,1_j)\}$ from B3.2.
  Take $\mathscr{H}_d(v^\alpha_1,jz^1_v)\subset
  \mathscr{H}_d(1\sigma^\alpha_1,1_j,p\sigma)$.
  Let $v^\alpha_1>q$.
  Then \[D(q,v)=g[q,\cup\mathscr{H}_d(v^\alpha_1,jz^1_v)]
  =g[q,\cup\mathscr{G}_d(v^\alpha_1,jz^1_v)]\] is a c.o set
  in $(X,\rho)$ by the definition of $\mathcal{G}(v^\alpha_0,jz^n_v)$
  and 3 of B of Claim 5.8.
  Then
  $D(q,v)\subset g(m_3,z^1_v)$.
  Let $\mathscr{D}(q+,z,1_j)=\{D(q,v):v^\alpha_1>q\}$.
  Then $\mathscr{D}(q+,z,1_j)$ is c.o.D in $(X,\rho)$.
  Let $v^\alpha_1\leq q$.
  Then $g[q,g(v^\alpha_1,t)]=g(v^\alpha_1,t)$.
  Let $\mathscr{D}(q-,z,1_j)=\cup\{\mathscr{G}_d(v^\alpha_1,jz^1_v):v^\alpha_1\leq q\}$.
  Then $\mathscr{D}(q-,z,1_j)$ is c.o.D in $(X,\rho)$.
  Let
  \[\mathscr{D}(q,z,1_j)=\mathscr{D}(q-,z,1_j)\cup\mathscr{D}(q+,z,1_j)\]
  \[\mathscr{D}_d(q,m^\alpha,1_j)
  =\cup\{\mathscr{D}(q,z,1_j):
  g(m_1,z)\in\mathscr{G}^*(m_1,1_j)\}.\]
  Then $\mathscr{D}_d(q,m^\alpha,1_j)$ is c.o.D in $(X,\rho)$.
  Let $G_d(1\alpha,j)=\cup\mathscr{G}_d(1\sigma^\alpha_1,1_j,p\sigma)$.
  Then $g[q,G_d(1\alpha,j)]=\cup\mathscr{D}_d(q,m^\alpha,1_j)=D(q,m^\alpha,1_j)$
  is a c.o set for every $j\in N$.
\par

B. I. Let $k=1$. Note $H(m,1_1)=H'(m,1_1)-D(q,m^\alpha,1_1)$.
  Then we have $H(m,1_1)=H'(m,1_1)-g[q,G_d(1\alpha,1_1)]$.
   Then $H(m,1_1)\cap g[q,\cup_{j>1}G_d(1\alpha,j)]=\emptyset$.
   Then $H(m,1_1)\cap \cup_{j>1}D(q,m^\alpha,1_j)=\emptyset$
   by the definition of $D(q,m^\alpha,1_j)$.
  Then \[H(m,1_1)=H'(m,1_1)-g[q,\cup_jG_d(1\alpha,j)]
  =H'(m,1_1)-D(q,m^\alpha,1_1).\]
\par

II. Assume $g[q,\cup_{j\leq n}G_d(1\alpha,j)]=\cup_{j\leq n}D(q,m^\alpha,1_j)$.
  Let $k=n+1$. Take $H'(m,1_k)$.
  Then, by 1 and 2 Proposition 6.5, there exist  infinitely many
  $g(m_1,z)$ such that $H(m_1,z)\subset H'(m,1_k)$
  and $g(m_1,z)\cap g[q,\cup_{n<k}H'(m,1_n)]=\emptyset$.
  Let \[H''(m,1_k)=H'(m,1_k)-g[q,\cup_{j\leq n}G_d(1\alpha,j)]
  \quad\hbox{ and }\]
  \[H^{**}(m,1_k)=H'(m,1_k)-\cup_{j\leq n}D(q,m^\alpha,1_j).\]
  Then $H^{**}(m,1_k)=H''(m,1_k)\neq\emptyset$
  by inductive assumption.
  Then $H''(m,1_k)\subset H^*(m,1_k)$ by the definition of $H^*(m,1_k)$
  in B3.2.
  Take $\mathscr{G}^*(m_1,1_k)$ from B3.2.
  Let
  \[\mathscr{G}''(1\alpha,k)=\{g(m_1,z):
  g(m_1,z)\in\mathscr{G}^*(m_1,1_k)\hbox{ and }H(m_1,z)\cap H''(m,1_k)\neq\emptyset\}.\]
  \[\mathscr{G}(1\alpha,k)=\cup\{\mathscr{G}_d(1\sigma^\alpha_1,kz,p\sigma):
  g(m_1,z)\in\mathscr{G}''(1\alpha,k)\},\]
  \[\mathscr{G}(q,m^\alpha,1_k)=\{g[q,g(v^\alpha_1,t)]: g(v^\alpha_1,t)
  \in\mathscr{G}(1\alpha,k)\}\ \hbox{and}\
  G(1\alpha,k)=\cup\mathscr{G}(1\alpha,k).\]
  Then $g[q,G(1\alpha,k)]=\cup\mathscr{G}(q,m^\alpha,1_k)=D(q,m^\alpha,1_k),$
  \[g[q,G(1\alpha,k)]\cup g[q,\cup_{j\leq n}G_d(1\alpha,j)]
  =g[q,\cup_{j\leq k}G_d(1\alpha,j)],\]
  \[D(q,m^\alpha,1_k)\cup[\cup_{j\leq n}D(q,m^\alpha,1_j)]
  =\cup_{j\leq k}D(q,m^\alpha,1_j)\quad\hbox{and}\]
  \[g[q,G(1\alpha,k)]\cup g[q,\cup_{j\leq n}G_d(1\alpha,j)]=
  D(q,m^\alpha,1_k)\cup[\cup_{j\leq n}D(q,m^\alpha,1_j)].\]
  Then $H(m,1_k)=H'(m,1_k)-g[q,\cup_{j\leq k}G_d(1\alpha,j)]
  =H'(m,1_k)-\cup_{j\leq k}D(q,m^\alpha,1_j)$
  in the same way as the above I.
\par

Then we have $\mathscr{G}''(1\alpha,k)$, $\mathscr{G}(q,m^\alpha,1_k)$
  and $H(m,1_k)$ for every $k\in N$.
  Let $\mathscr{H}(q,m^\alpha,1)=\{H(m,1_k): k\in N\}$
  and $\mathscr{D}(q,m^\alpha,1)=\{D(q,m^\alpha,1_k): k\in N\}$.
  This implies 1, 3, 4 and 5.
  Let $\mathcal{G}''(1,\alpha)=\{\mathscr{G}''(1\alpha,k): k\in N\}$.
\par

Let \[A^*_{dj}=\{x_{dz}:x_{dz}\in H(m_1,z)-[\cup\mathscr{D}(q,z,1_j)]
  \hbox{ and }g(m_1,z)\in\mathscr{G}''(1\alpha,j)\}\quad\hbox{ and}\]
   \[A^*_d=\{A^*_{dj}:H'(m_1,1_j)\in\mathscr{H}'(m,y_0)\}.\]
  Then, in the same way as the proof of 3 of Claim 6.8, we have the following fact:
\begin{fact}
Let $H(l,t)\subset H(q,m^\alpha)=\cup\mathscr{H}'(q,m^\alpha,1)$. Then:
\par

1. $H(l,t)\subset Cl_\rho A^*_d$.
\par

2. $H(l,t)\subset Cl_\rho[\cup_{\ell>j}H(m,1_\ell)]$,
  $H(l,t)\subset Cl_\rho[\cup_{i> j}
  \cup\mathscr{G}(1\alpha,i)]$
  and
  $H(l,t)\subset Cl_\rho[\cup_{i> j}\cup\mathscr{H}(1\alpha,i)]$.
\end{fact}
This implies 2 and 6. Then we complete the proof of proposition.
\end{proof}

  Let $p>m_5$.
Let $D=D(1\sigma^\alpha_1,G_1)=\cup\mathscr{G}(1\sigma^\alpha_1,G_1)$.
  Then $D$ is an open set in $(X,\rho)$.
  Let $\partial D=[Cl_\rho D]-D$.
  Then $D\cap \partial D=\emptyset$.
  Take $g[p,D(1\sigma^\alpha_1,G_1)]$ from 1 of Proposition 6.13.
  Let $\hbox{max}\{p,m_5+1\}=q_2$.
  Take $\mathscr{H}_d(2\sigma^\alpha_1,q_2\sigma)$  in B2. Let
  \[H(2\sigma^\alpha_1,q_2\sigma)=[\cup\mathscr{H}_d(2\sigma^\alpha_1,q_2\sigma)]
  -g[p,D(1\sigma^\alpha_1,G_1)].\]
  Note Face 6.12 and Fact 6.14. Then we have the
  following denotation.
\par

\vspace{0.3cm}
  \textbf{Denotation 1.}
Denote  $g[p,D(1\sigma^\alpha_1,G_1)]\cup H(2\sigma^\alpha_1,q_2\sigma)$ by
  $Cl_\tau g[p,D(1\sigma^\alpha_1,G_1)]$ for $p>m_5$.
\par

\begin{cl} Let $p>m_5$.
Then:
\par

1.  $\cap_p Cl_\tau g[p,D(1\sigma^\alpha_1,G_1)]= D(1\sigma^\alpha_1,G_1)$.
\par

2. Let $t\in Cl_\tau g[p,D(1\sigma^\alpha_1,G_1)]-D(1\sigma^\alpha_1,G_1)$.
  Then there exists a $k>p$ such that
  $H(\ell,t)\cap Cl_\tau g[q,D(1\sigma^\alpha_1,G_1)]=\emptyset$
  if $\ell,q>k$.
\par

3.  $Cl_\tau g[p,D(1\sigma^\alpha_1,G_1)]=g[p,D(1\sigma^\alpha_1,G_1)]\cup
  H(2\sigma^\alpha_1,q_2\sigma)$.
\par
\end{cl}

\begin{proof} To see 2 let $E(\alpha2,p)=Cl_\tau g[p,D(1\sigma^\alpha_1,G_1)]$.
  Then, for $q_2=\hbox{max}\{p,m_5+1\}$,
  \[E(\alpha2,p)=g[p,\cup\mathscr{G}(1\sigma^\alpha_1,G_1)]
  \cup H(2\sigma^\alpha_1,q_2\sigma).\]
  Let $t\in E(\alpha2,p)-D(1\sigma^\alpha_1,G_1)$.
  Then $t\in g[p,\cup\mathscr{G}(1\sigma^\alpha_1,G_1)]$ or
  $t\in H(2\sigma^\alpha_1,q_2\sigma)$.
\par

Case 1, $t\in g[p,\cup\mathscr{G}(1\sigma^\alpha_1,G_1)]$.
  Then $t\in g[p,\cup\mathscr{G}(1\sigma^\alpha_1,G_1)]-
  \cup\mathscr{G}(1\sigma^\alpha_1,G_1)$.
   Then, by 6 of Claim 6.8, there exists a $q>p$ such that  $H(\ell,t)\cap g[q,D]=\emptyset$
  if $\ell>q$. Here $D=D(1\sigma^\alpha_1,G_1)$.
  Then we have the following Fact.

\begin{fact}
Let $t\in g[p,\cup\mathscr{G}(1\sigma^\alpha_1,G_1)]-
  \cup\mathscr{G}(1\sigma^\alpha_1,G_1)$.
  Then there exists a $k_1\in N$ such that
  $H(\ell,t)\cap g[q,D(1\sigma^\alpha_1,G_1)]=\emptyset$ if $\ell,q>k_1$.
\end{fact}

Case 2. $t\in H(2\sigma^\alpha_1,q_2\sigma)$.
  Then $t\notin D$ since $H(2\sigma^\alpha_1,q_2\sigma)\cap
  g[p,D(1\sigma^\alpha_1,G_1)]=\emptyset$.
  Then $H(\ell,t)\cap g[q,D]=\emptyset$ by Fact 6.16.
\par

On the other hand, by the definition of $H(2\sigma^\alpha_1,q_2\sigma)$,
  there exists an $H(u^\alpha_1,t')\in\mathscr{H}(2\sigma^\alpha_1,q_2\sigma)$ such that $x\in H(u^\alpha_1,t')=H(u^\alpha_1,t)$.
  Note the definition of $\mathscr{H}(2\sigma^\alpha_1,q_2\sigma)$.
  Let $v_p=u^\alpha_1$ and $q>v_p$.
  Then $H(u^\alpha_1,t)\notin\mathscr{H}(2\sigma^\alpha_1,q\sigma)$
  and $H(u^\alpha_1,t)\cap [\cup\mathscr{H}(2\sigma^\alpha_1,q\sigma)]=\emptyset$.
Then we have the following Fact.

\begin{fact}
Let $t\in  H(2\sigma^\alpha_1,q_2\sigma)$.
  Then there exists a $v_p\in N$ such that
  $H(\ell,t)\cap H(2\sigma^\alpha_1,q\sigma)=\emptyset$ if $q,\ell>v_p$.
\end{fact}
\vspace{0.3cm}
Let $q>n(t)=\hbox{max}\{k_1,v_p\}$.
  Then, by Denotation 1, we have
  \[g[q,D(1\sigma^\alpha_1,G_1)]\cup H(2\sigma^\alpha_1,q\sigma)
  =Cl_\tau g[q,D(1\sigma^\alpha_1,G_1)].\]
  Then $H(\ell,t)\cap Cl_\tau g[q,D(1\sigma^\alpha_1,G_1)]=\emptyset$
  if $\ell>n(t)$.

\par
{\sl The proof of 2 is continued}.
Then, by Case 1 and  Case 2, we have proved 2.
  Then it is easy to see
  $\cap_pCl_\tau g[p,D(1\sigma^\alpha_1,G_1)]\subset D(1\sigma^\alpha_1,G_1)$.
  This implies 1.
\end{proof}

\vspace{0.3cm}

\textbf{Construction 6.1 is continued.}

\textbf{Q.}
  Assume, for  $n=h+1$, and $p_n\geq p_h\geq  ...\geq p_2\geq  p>m_5$,
  we have constructed
  \[E(\alpha n,p_n)=Cl_\tau g[p_n,E(\alpha h,p_h)]
  =g[p_n,E(\alpha h,p_h)]\cup H(n\sigma^\alpha_1,q_n\sigma).\]
  Here $q_n=\hbox{max}\{p_n,q_h\}$ and
  $H(n\sigma^\alpha_1,q_n\sigma)=[\cup\mathscr{H}_d(n\sigma^\alpha_1,q_n\sigma)]-
  g[p_n,E(\alpha h,p_h)]$.
  By induction in the same as method of B1 and B2, we construct
  \[E(\alpha k,p_k)=Cl_\tau g[p_k,E(\alpha n,p_n)]
  =g[p_k,E(\alpha n,p_n)]\cup H(k\sigma^\alpha_1,q_k\sigma)\hbox{ such that }\]
 \begin{align}
 E(\alpha k,p_k)=D(p_2,m^\alpha)\cup &
         g[p_3,H(2\sigma^\alpha_1,q_2\sigma)]\cup...\cup \nonumber\\
                            &\cup  g[p_k,H(n\sigma^\alpha_1,q_n\sigma)]\cup
                               H(k\sigma^\alpha_1,q_k\sigma)\hbox{ and}\quad\nonumber
\end{align}
$H(k\sigma^\alpha_1,q_k\sigma)=[\cup\mathscr{H}_d(k\sigma^\alpha_1,q_k\sigma)]-
  g[p_k,E(\alpha n,p_n)]$.
Here $k=n+1$, $q_k=\hbox{max}\{p_k,q_n\}$ and
  $p_k\geq p_n\geq  ...\geq p_1>m_5$.
\par

By Proposition 6.13, assume we have had $\mathscr{H}(p_n,m^\alpha,n)=\{H(m,n_i): i\in N\}$,
  $\mathscr{D}(p_n,m^\alpha,n)=\{D(p_n,m^\alpha,n_i): i\in N\}$ and
  $\mathcal{G}''(n,\alpha)=\{\mathscr{G}''(n\alpha,i): i\in N\}$ with
  \[H(m,n_j)=H'(m,1_j)-D(\alpha n,p_n)=H'(m,1_j)-[\cup_{i\leq j}D(p,m^\alpha,n_i)].\]
  Here $D(\alpha n,p_n)=\cup\mathscr{D}(p_n,m^\alpha,n)$.
\par

\vspace{0.3cm}

\textbf{Q1.}
Take $H'(m,1_1)\in\mathscr{H}'(m,y_0)$
  and $D(p_n,m^\alpha,n_1)\in\mathscr{D}(p_n,m^\alpha,n)$.
   Note $k=n+1$.
   Let $\mathscr{G}(m_1,k_1)$ be a Ln cover of $H'(m,1_1)$,
  \[H^*(m,k_1)=H'(m,1_1)-D(p_n,m^\alpha,n_1)\quad\hbox{ and}\]
  \[\mathscr{G}^*(m_1,k_1)=\{g(m_1,z)\in\mathscr{G}(m_1,k_1):
  g(m_1,z)\cap H^*(m_1,k_1)\neq\emptyset\}.\]
Let  $g(m_1,z)\in\mathscr{G}^*(m_1,k_1)$ and $H(m_1,1z)=g(m_1,z)\cap H'(m,1_1)$.
 Then we have $H(m_1,1z)\cap E(\alpha n,p_n)\subset\cup_{i\leq n}H(m_2,1z_i)$
 since $\cup\mathscr{H}(i\sigma^\alpha_1,1z,p\sigma)\subset H(m_2,1z_i)$
 by 1 of Claim 6.6 and 3 of Claim 6.7, and $p_i>m_5$ for $i\leq n$.
 Let $p_k\geq p_n$ and $q_k=\hbox{max}\{p_k,q_n\}$.
 Take $\mathscr{G}_d(k\sigma^\alpha_1,1z,q_k\sigma)$ and
  $\mathscr{G}_{-1}(k\sigma^\alpha_1,1z,q_k\sigma)$ from B2
  for every $H(m_1,z)\in\mathscr{G}^*(m_1,k_1)$.
  Let
   \[\mathscr{G}_d(k\sigma^\alpha_1,1_1,q_k\sigma)=
  \cup\{\mathscr{G}_d(k\sigma^\alpha_1,1z,q_k\sigma):
  g(m_1,z)\in\mathscr{G}^*(m_1,k_1)\},\]
  \[H(k\sigma^\alpha_1,1_1,q_k\sigma)=\cup\mathscr{H}_d(k\sigma^\alpha_1,1_1,q_k\sigma),\]
  \[\mathscr{G}_{-1}(k\sigma^\alpha_1,1_1,q_k\sigma)
  =\cup\{\mathscr{G}_{-1}(k\sigma^\alpha_1,1z,q_k\sigma):
  g(m_1,z)\in\mathscr{G}^*(m_1,k_1)\}\quad\hbox{ and}\]
  \[D(p_k,m^\alpha,k_1)=D(p_n,m^\alpha,n_1)\cup g[p_k,H(k\sigma^\alpha_1,1_1,q_k\sigma)].\]
\par

\vspace{0.3cm}
\textbf{Q2.} Assume we have had $\mathscr{H}(k\sigma^\alpha_1,1_i,q_k\sigma)$
  and $D(p_k,m^\alpha,k_i)$ for $i<j$.  Take $H'(m,1_j)\in\mathscr{H}'(m,y_0)$
  and $D(p_n,m^\alpha,n_j)\in\mathscr{D}(q_n,m^\alpha,n)$.
   Let $\mathscr{G}(m_1,k_j)$ be a Ln cover of $H'(m,1_j)$,
  \[H^*(m,k_j)=H'(m,1_j)-\cup_{i<j}D(p_k,m^\alpha,k_i)\quad\hbox{ and}\]
  \[\mathscr{G}^*(m_1,k_j)=\{g(m_1,z)\in\mathscr{G}(m_1,k_j):
  g(m_1,z)\cap H^*(m,k_j)\neq\emptyset\}.\]
  Then, by 2 and 3 of Proposition 6.5, there exists infinitely many $g(m_1,z)\in\mathscr{G}^*(m_1,k_j)$
  with $H(m_1,z)\subset H^*(m,k_j)$.
  Then, in the same way as Q1, take
  $\mathscr{H}_d(k\sigma^\alpha_1,jz,q_k\sigma)$
  and $\mathscr{G}_{-1}(k\sigma^\alpha_1,1_1,q_k\sigma)$
  for every $g(m_1,z)\in\mathscr{G}^*(m_1,k_j)$.
  Let
  \[\mathscr{G}_d(k\sigma^\alpha_1,1_j,q_k\sigma)=
  \cup\{\mathscr{G}_d(k\sigma^\alpha_1,jz,q_k\sigma):
  g(m_1,z)\in\mathscr{G}^*(m_1,k_j)\},\]
  \[H(k\sigma^\alpha_1,1_j,q_k\sigma)=\cup\mathscr{H}_d(k\sigma^\alpha_1,1_j,q_k\sigma),\]
  \[\mathscr{G}_{-1}(k\sigma^\alpha_1,1_j,q_k\sigma)
  =\cup\{\mathscr{G}_{-1}(k\sigma^\alpha_1,jz,q_k\sigma):
  g(m_1,z)\in\mathscr{G}^*(m_1,k_j)\}\quad\hbox{ and}\]
  \[D(p_k,m^\alpha,k_j)=D(p_n,m^\alpha,n_j)\cup g[p_k,H(k\sigma^\alpha_1,1_j,q_k\sigma)].\]
\par

Then, by induction, we have $D(q_k,m^\alpha,k_j)$,
  $\mathscr{H}_d(k\sigma^\alpha_1,1_j,q_k\sigma)$
  and $\mathscr{H}_{-1}(k\sigma^\alpha_1,1_j,q_k\sigma)$ for every
  $H'(m,1_j)\in\mathscr{H}'(m,y_0)$.
  Let $\mathscr{G}^*(m_1,k_\sigma)=\cup_j\mathscr{G}^*(m_1,k_j)$,
  \[D(\alpha k,p_k)=\cup_jD(p_k,m^\alpha,k_j)\quad\hbox{ and }\quad
  H(m,k_j)=H'(m,k_j)-D(\alpha k,p_k).\]
  Then $H(m,k_j)=H'(m,1_j)-[\cup_{i\leq j}D(p_k,m^\alpha,k_i)]$.
\par

  Then $\mathscr{H}_d(k\sigma^\alpha_1,1_j,q_k\sigma)$
  and $g[p_k,H(n\sigma^\alpha_1,q_n\sigma)]$
  satisfy Proposition 6.13.
 Let
 \[\mathscr{H}_d(k\sigma^\alpha_1,q_k\sigma)
 =\cup\{\mathscr{H}_d(k\sigma^\alpha_1,1_j,q_k\sigma):j\in N\},\]
 \[H(k\sigma^\alpha_1,q_k\sigma)=[\cup\mathscr{H}_d(k\sigma^\alpha_1,q_k\sigma)]-
  g[p_k,E(\alpha n,p_n)]\hbox{ and}\]
 \[E(\alpha k,p_k)=Cl_\tau g[p_k,E(\alpha n,p_n)]
  =g[p_k,E(\alpha n,p_n)]\cup H(k\sigma^\alpha_1,q_k\sigma).\]

We prove that the definition of $E(\alpha k,p_k)
  =Cl_\tau g[p_k,E(\alpha n,p_n)]$ is reasonable to definite closure operations.
\begin{prop}
Let $p'_k>p_k$. Then $Cl_\tau g[p'_k,E(\alpha n,p_n)]\subset Cl_\tau g[p_k,E(\alpha n,p_n)]$.
\end{prop}
\begin{proof}
By 1 of Claim 6.7, we have
  $H(1\sigma^\alpha_1,q'_1\sigma)\subset H(1\sigma^\alpha_1,q_1\sigma)$ if $q'_1>q_1$.
\par

A. Let $p'_2>p_2$. Then
  $g[p'_2, D(1\sigma^\alpha_1,G_1)]\subset g[p_2, D(1\sigma^\alpha_1,G_1)]$.
  Let $q'_2=\hbox{max}\{p'_2,q_1\}$ and $q_2=\hbox{max}\{p_2,q_1\}$.
  Then $q'_2\geq q_2$.
  Then, by B2 in Construction 6.1, we have $\mathscr{H}_d(2\sigma^\alpha_1,q'_2\sigma)
  \subset\mathscr{H}_d(2\sigma^\alpha_1,q_2\sigma)$.
  Then
\begin{align}
Cl_\tau g[p'_2,D(1\sigma^\alpha_1,G_1)]
       &=g[p'_2,D(1\sigma^\alpha_1,G_1)]\cup H(2\sigma^\alpha_1,q'_2\sigma)\nonumber\\
       &=g[p'_2,D(1\sigma^\alpha_1,G_1)]\cup
          [\cup\mathscr{H}_d(2\sigma^\alpha_1,q'_2\sigma)-g[p'_2,D(1\sigma^\alpha_1,G_1)]]  \nonumber\\
       &=g[p'_2,D(1\sigma^\alpha_1,G_1)]\cup
          [\cup\mathscr{H}_d(2\sigma^\alpha_1,q'_2\sigma)] \nonumber\\
       &\subset g[p_2,D(1\sigma^\alpha_1,G_1)]\cup
          [\cup\mathscr{H}_d(2\sigma^\alpha_1,q_2\sigma)] \nonumber\\
       &=g[p_2,D(1\sigma^\alpha_1,G_1)]\cup
          [\cup\mathscr{H}_d(2\sigma^\alpha_1,q_2\sigma)-
          g[p_2,D(1\sigma^\alpha_1,G_1)] \nonumber\\
       &=g[p_2,D(1\sigma^\alpha_1,G_1)]\cup H(2\sigma^\alpha_1,q_2\sigma)\nonumber\\
       &=Cl_\tau g[p_2,D(1\sigma^\alpha_1,G_1)].\nonumber
\end{align}

B. Assume that we have $Cl_\tau g[p'_n,E(\alpha h,p_h)]\subset Cl_\tau g[p_n,E(\alpha h,p_h)]$
  if $p'_n>p_n$ for $n=h+1$. Let $p'_k>p_k$.
  Then $q'_k=\hbox{max}\{p'_k,q_n\}\geq q_k=\hbox{max}\{p_k,q_n\}$.
  Then $g[p'_k,E(\alpha n,p_n)]\subset g[p_k,E(\alpha n,p_n)]$
  and $\mathscr{H}_d(k\sigma^\alpha_1,q'_k\sigma)\subset
  \mathscr{H}_d(k\sigma^\alpha_1,q_k\sigma)$. Then
  \[E(\alpha k,p'_k)=Cl_\tau g[p'_k,E(\alpha n,p_n)]\subset Cl_\tau g[p_k,E(\alpha n,p_n)]
  \subset E(\alpha k,p_k)\]
  in the same way as the above A.
\end{proof}

Then we have $E(\alpha1,p_1)=D(1\sigma^\alpha_1,G_1)$,
  \[E(\alpha2,p_2)=g[p_2,G(1\sigma^\alpha_1,G_1)]
  \cup H(2\sigma^\alpha_1,q_2\sigma)\quad\hbox{ with }\quad
  q_2=\hbox{max}\{p_2, q_1\}\quad\hbox{ and}\]
  \[E(\alpha k,p_k)=Cl_\tau g[p_k,E(\alpha n,p_n)]
  =g[p_k,E(\alpha n,p_n)]
  \cup H(k\sigma^\alpha_1,q_k\sigma)\]
  with $q_k=\hbox{max}\{p_k,q_n\}.$

\section{To construct a stratifiable space $(Y,\tau)$}

\vspace{0.3cm}
 \textbf{Construction 7.}

\vspace{0.3cm}

 We define a topological space $(Y,\tau)$. To do it take $g(m,a^*,y_0)$. Let
    \[Y=G_m=g(m,a^*,y_0)\quad\hbox{ and }\quad H_0=g(m,a^*,y_0)\cap H'(m,n_0). \]
  Then $H_0=H(m,a^*,y_0)$.
  For every $x\in Y$, let $\mathscr{H}_x=\{H(l,x)\in\mathscr{H}: l\geq m\}$ and
  $\mathscr{G}_x=\{g(l,x)\in\mathscr{G}: l\geq m\}$.
  Then \[\{x\}=\cap_{l\geq m} H(l,x)=\cap_{l\geq m} g(l,x).\]
\par

Let
  \[\mathscr{H}_Y=\cup\{\mathscr{H}_x: x\in Y\}
  \quad \hbox{ and } \quad
  \mathscr{G}_Y=\cup\{\mathscr{G}_x: x\in Y\}.\]

To define topological space $(Y,\tau)$, take
  $\mathscr{G}(1\sigma^\alpha_1,G_1).$ Let $D(1\sigma^\alpha_1,G_1)=\cup\mathscr{G}(1\sigma^\alpha_1,G_1)$ for $\alpha\in\mathscr{A}$.
 Let $k=n+1$, $q_k=\hbox{max}\{p_k,q_n\}$ for $p_k\geq p_n\geq  ...\geq p_1>m_5,$
\[H(k\sigma^\alpha_1,q_k\sigma)=[\cup\mathscr{H}_d(k\sigma^\alpha_1,q_k\sigma)]-
  g[p_k,E(\alpha n,p_n)],\]
 \begin{align}
 E(\alpha k,p_k)=g[p_2,D(1\sigma^\alpha_1,G_1)]\cup &
         g[p_3,H(2\sigma^\alpha_1,q_2\sigma)]\cup...\cup\nonumber\\
                            &\cup  g[p_k,H(n\sigma^\alpha_1,q_n\sigma)]\cup
                               H(k\sigma^\alpha_1,q_k\sigma),\quad\nonumber
\end{align}
  \[D(\alpha k,p_k)=g[p_2,D(1\sigma^\alpha_1,G_1)]\cup g[p_3,H(2\sigma^\alpha_1,q_2\sigma)]\cup...\cup
  g[p_k,H(n\sigma^\alpha_1,q_n\sigma)].\]
  Then $E(\alpha k,p_k)=D(\alpha k,p_k)\cup H(k\sigma^\alpha_1,q_k\sigma)$
  and $E(\alpha k,p_k)\cap H_0=\emptyset$.

\vspace{0.3cm}

\textbf{Definition A}.
1. Let $g(l,x)$
  be closed and open if $g(l,x)\in\mathscr{G}_Y$.
\par

2. Let $Cl_\tau g[p_2,D(1\sigma^\alpha_1,G_1)]=G_m$
  if $p_2\leq m_5$, and
  let \[Cl_\tau g[p_2,D(1\sigma^\alpha_1,G_1)]=
  g[p_2,D(1\sigma^\alpha_1,G_1)]\cup H(2\sigma^\alpha_1,q_2\sigma)=E(\alpha 2,p_2)\]
  be closed and open if $p_2>m_5$.
\par

3. Let
  \[U(\alpha\ell 2,p_2x)=g(\ell,x)-E(\alpha 2,p_2)\]
  be a closed and open neighborhood  of $x$
  if $x\in Y-E(\alpha2,p_2)$ and $\ell>m_5$.
\par

\textbf{Definition B}.  Let $k=n+1>2$.
\par

1. Let $Cl_\tau D(\alpha k,p_k)=G_m$
  if $p_k\leq m_5$, and
  let \[Cl_\tau D(\alpha k,p_k)
  =D(\alpha k,p_k)\cup H(k\sigma^\alpha_1,q_k\sigma)=E(\alpha k,p_k)\]
  be closed and open if $p_k>m_5$.
\par

2. Let
  \[U(\alpha\ell k,p_kx)=g(\ell,x)-E(\alpha k,p_k)\]
  be a closed and open neighborhood  of $x$ if $x\in Y-E(\alpha k,p_k)$ and $\ell>m_5$.
\par

\textbf{Definition C}.
  Let $p_h>p_k$ and \[g[p_h,U(\alpha\ell k,p_kx)]\]
  be a closed and open if
  $U(\alpha\ell k,p_kx)=g(\ell,x)-E(\alpha k,p_k)$.\quad\quad$\Box$
\par

\vspace{0.3cm}

\begin{prop}
Definition C is reasonable.
\end{prop}

\begin{proof}

Let $y'\in H(v^\alpha_1,t)\in\mathscr{H}_d(v^\alpha_1,jz^n_v)$.
  Then $H(v^\alpha_1,t)\subset H(n\sigma^\alpha_1,p_n\sigma)
  \subset E(\alpha n,p_n)$.
  Then $H(v^\alpha_1,t)\cap[Y-E(\alpha n,p_n)]=\emptyset$.
  Let $H(v^\alpha_1,t)\subset H'(m,l_{j0})$, $1_j=l_{j0}$ and
  \[H(mv^\alpha_1,-t)=[H'(m,l_{j0})\cap g(v^\alpha_1,t)] -H(v^\alpha_1,t).\]
  Then $H(mv^\alpha_1,-t)\subset Y-E(\alpha n,p_n)$ by the definition of $E(\alpha n,p_n)$.
  Take a $g(\ell,s)$ such that $g(v^\alpha_1,t)\subset g(\ell,s)$.
  Then we have
  \[H(mv^\alpha_1,-t)\subset U_s=U(\alpha\ell n,p_ns)=g(\ell,s)-E(\alpha n,p_n)\quad\hbox{ with }\quad H(v^\alpha_1,t)\cap U_s=\emptyset.\]
\par

\textbf{I.}  $g_c(p,y')\cap\Big[Y-[g[\ell,U_s]
  \cup H(v^\alpha_1,t)]\Big]\neq\emptyset$ if  $H(v^\alpha_1,t)\cap g[\ell,H(mv^\alpha_1,-t)]=\emptyset$
   for every $p,\ell\geq\hbox{max}\{v^\alpha_1,p_1,p_2,...,p_n,q\}$.
\par

In fact,  $y'\in H(v^\alpha_1,t)\subset H'(m,l_{j0})$.
  Let $\mathscr{G}(m_1,l_{j0})$ be a Ln cover of $H'(m,l_{j0})$,
  \[H^*(m,n_{j0})=H'(m,l_{j0})-\cup_{1_i<l_{j0}}D(p_n,m^\alpha,n_i)\quad\hbox{ and}\]
  \[\mathscr{G}^*(m_1,l_{j0})=\{g(m_1,z)\in\mathscr{G}(m_1,l_{j0}):
  g(m_1,z)\cap H^*(m,l_{j0})\neq\emptyset\}.\]
  Then, by the definition of $D(p_n,m^\alpha,n_i)$,
  there is a $p\geq\hbox{max}\{v^\alpha_1,p_1,p_2,...,p_n,q\}$
  with $g_c(p,y')\cap [\cup_{1_i<l_{j0}}D(p_n,m^\alpha,n_i)]=\emptyset$.
  Here $\mathscr{G}(\ell,1_i)$ is a Ln cover on $H'(m,1_i)$
  for $1_i<l_{j0}$, and
  \[q=\hbox{min}\{\ell:[\cup\mathscr{G}(\ell,1_i)]\cap H'(m,l_{j0})=\emptyset
  \hbox{ for } 1_i<l_{j0}\}.\]
  Let $1_\ell$ be the least number with
  $g_c(p,y')\cap H'(m,1_\ell)\neq\emptyset$ just as the same definition before
  Fact 6.9.
  Then, by 1 of Fact 6.11, $\emptyset\neq\mathscr{G}^*(m_1,y'1_\ell)\subset\mathscr{G}^*(m_1,1_\ell)$.
\par

Let $g(m_1,z)\in\mathscr{G}^*(m_1,y'1_\ell)$. Then $g(m_1,z)\subset g_c(p,y')$.
  Take an $H(v^\alpha_1,t_z)$ from $\mathscr{H}_d(v^\alpha_1,jz^n_v)$
  for $j=1_\ell$ and $H(v^\alpha_1,t_z)\subset H(m_1,z)$.
  Then, by Fact 6.19, we have $g(m_1,z)\cap G(m_1,1_\ell)=\emptyset$
  for $G(m_1,1_\ell)=\cup\{\cup\mathscr{G}(m_1,1_k): l_{j0}<1_k<1_\ell\}$.
\par

  And then $H'(m,1_\ell)\cap[\cup\mathscr{G}(m_1,1_i)]=\emptyset$
  for $1_i>1_\ell$ by 6 of Proposition 3.3.
\par

Note $H(v^\alpha_1,t)\cap g[\ell,H(mv^\alpha_1,-t)]=\emptyset$
   for every $p,\ell\geq\hbox{max}\{v^\alpha_1,p_1,p_2,...,p_n,q\}$.
   Then $H(v^\alpha_1,t_z)\cap g[\ell,U_s]=\emptyset$
   if $\ell\geq\hbox{max}\{v^\alpha_1,p_1,p_2,...,p_n,q\}$.
   Then it is easy to see
   $H(v^\alpha_1,t_z)\subset Y-g[\ell,U_s]$ and $H(v^\alpha_1,t_z)\subset g_c(p,y')$.
   Then  \[g_c(p,y')\cap\Big[Y-[g[\ell,U_s]
  \cup H(v^\alpha_1,t)]\Big]\neq\emptyset.\]
\par
\vspace{0.3cm}

\textbf{II.} We prove that: if
  $g_c(p,y')\cap\Big[Y-[g[\ell,U_s]
  \cup H(v^\alpha_1,t)]\Big]\neq\emptyset$ for every $p,\ell\geq\hbox{max}\{v^\alpha_1,p_1,p_2,...,p_n,q\}$,
  then Definition C is reasonable.
\par

  In fact, $y'\in H(v^\alpha_1,t)\in\mathscr{H}_d(v^\alpha_1,jz^n_v)$ and  $H(v^\alpha_1,t)\cap[Y-E(\alpha n,p_n)]=\emptyset$.
  Suppose  $g_c(p,y')\cap\Big[Y-[g[\ell,U_s]
  \cup H(v^\alpha_1,t)]\Big]=\emptyset$ for some $p,\ell\geq\hbox{max}\{v^\alpha_1,p_1,p_2,...,p_n,q\}$.
  Then we have $g_c(p,y')-H(v^\alpha_1,t)\subset g[\ell,U_s]$.
  Note that $H(v^\alpha_1,t)\cap[g[\ell,U_s]=\emptyset$, and $g[\ell,U_s]$ is closed  by Definition C in Section 7.
  Then we have $H(v^\alpha_1,t)\cap g_c(p,y')=g_c(p,y')-g[\ell,U_s]$ is open.
  Then $H(v^\alpha_1,t)$ is open since $y'$ is arbitrary in $H(v^\alpha_1,t)$.
  Then, by the definition of $D(\alpha n,p_n)$ in Q2 of Construction 6.1,
  this implies $H(n\sigma^\alpha_1,q_n\sigma)=\cup\{H(v^\alpha_1,t)\in
  \mathscr{H}_d(n\sigma^\alpha_1,q_n\sigma):H(v^\alpha_1,t)\cap D(\alpha n,p_n)
  =\emptyset\}$. Then $H(n\sigma^\alpha_1,q_n\sigma)$ is open.
  On the other hand, by 1 of Definition B in Section 7,
  \[Cl_\tau D(\alpha n,p_n)-D(\alpha n,p_n)=H(n\sigma^\alpha_1,q_n\sigma).\]
  This is a contradiction.
\par

\vspace{0.3cm}

\textbf{III.} So, in the following we must prove that
   $H(v^\alpha_1,t)\cap g[\ell,H(mv^\alpha_1,-t)]=\emptyset$
   for every $\ell\geq\hbox{max}\{v^\alpha_1,p_1,p_2,...,p_n,q\}$.

To do it let $k=v^\alpha_1$. Then $H(v^\alpha_1,t)=H(k,t)$
  and $g(v^\alpha_1,t)=g(k,t)$.
  Let  $H(k,t)=H(k,1_1,t)\subset H'(m,l_{j0})$, $a=k+1_1=k+d$,
  \[H(mv^\alpha_1,t)=
  H'(m,l_{j0})\cap g(v^\alpha_1,t),\]
  \[g(k,t)=[q_{1_1},q_{1_1a})\times ...\times
  [q_{1_k},q_{1_ka})\times [q_{h_1},q_{h_1a})\times...
  \times[q_{h_d},q_{h_da})\times X_{k+1_1}\ \hbox{ and}\]
  \[H(k,t)=\{q_{1_1}\}\times...\times\{q_{1_m}\}\times...\times\{q_{1_k}\}\times
   [q_{h_1},q_{h_1a})\times...
  \times[q_{h_d},q_{h_da})\times X_{k+1_1}.\]
  Note $H(k,t)\subset H'(m,l_{j0})$.
  Then $H'(m,l_{j0})=[\{q_{1_1}\}\times...
  \times\{q_{1_m}\}\times X_m]\cap G_m$ and $P(m,l_{j0})=\{q_{1_1}\}\times...
  \times\{q_{1_m}\}=(q_{1_1},...,q_{1_m})$.
  Let $b=m+1$. Then
  \[H(mv^\alpha_1,t)=H'(m,l_{j0})\cap g(k,t)=P(m,l_{j0})\times
  [q_{1_b},q_{1_ba})\times...\times[q_{1_k},q_{1_ka})\times J(k,1_1,t).\]
   Let $I(mv^\alpha_1,t)=(q_{1_1},...,q_{1_m})\times
  [q_{1_b},q_{1_ba})\times...\times[q_{1_k},q_{1_ka})$.
\par

\textbf{1.} Take a Ln cover of $H(mv^\alpha_1,-t)$ for $k=n+1$.
  To do it note Figure I in (B) of Construction 3.1.
  Let $P(k,1_i)=P(n,1_l)\times \{q_{1_j}\}$ in Figure I
  satisfy
  \[P(k,1_i)=(q_{1_1},...,q_{1_m})\times \{q_{i_b}\}\times...\times\{q_{i_k}\}.\]
  Then $P(k,1_i)\times X_k=H(k,1_i)$
  satisfies $\cup_iH(k,1_i)=H(mv^\alpha_1,t).$
  Then top $m$ coordinate of $H(k,1_i)$ is $(q_{1_1},...,q_{1_m})$
  for every $i$.
\par

\textbf{1A.} Let $H(k,t)=H(k,1_1,t)$,
  $g(k,1_1,t)=I(k,1_1,t)\times J(k,1_1,t)$ and
  \[H(k,1_1)=P(n,1)\times \{q_1\}\times X_k=P(k,1_1)\times X_k.\]
  Then $H(k,1_1)\cap g(k,1_1,t)=H(k,1_1,t)$.
  Note we take a Ln cover of $H(mv^\alpha_1,-t)$ for $k=n+1$.
  Then we can not take $P(n,1)\times \{q_1\}$ in
  Figure I in (B) of Construction 3.1 since $H(k,t)=H(k,1_1)\cap g(k,1_1,t)$
  and \[P(n,1)\times \{q_1\}=P(k,1_1)
  =(q_{1_1},...,q_{1_k})\in I(k,1_1,t).\]
\par

\textbf{1B.} We must take $P(n,1)\times \{q_2\}=P(k,1_2)$
  since $1_2$ is the least number such that
  $H(k,1_2)=P(n,1)\times \{q_2\}\times X_k=P(k,1_2)\times X_k$,
  \[H(k,1_2)\cap g(k,t)\subset H(mv^\alpha_1,-t)\quad\hbox{ and }\quad
  H(k,1_1)\cap H(k,1_2)=\emptyset.\]
  Let $a_2=k+1_2$, $P(k,1_2)=(q_{2_1},...,q_{2_k})$
   and $I(k,1_2)=[q_{2_1},q_{2_1a_2})\times ...\times
  [q_{2_k},q_{2_ka_2})$.
  Then $P(k,1_2)=(q_1,...,q_m,q_{2_b},...,q_{2_k})$
  by the definition of $P(k,1_i)$ in Figure I.
  Then
  \[I(k,1_2)=[q_1,q_{1a_2})\times ...\times
  [q_m,q_{ma_2})\times[q_{2_b},q_{2_ba_2})\times ...\times
  [q_{2_k},q_{2_ka_2}).\]
  Take $\mathscr{I}_{a_2}$ from 1 of Proposition 2.2. Note $a_2=k+1_2>k+1_1=a$.
  Let
  \[\mathscr{I}_{a_2j}=\mathscr{I}_{a_2}|[q_{h_j},q_{h_ja})\quad
  \hbox{ for }\quad 1\leq j\leq k, \hbox{ and}\]
  \[I(k,a_2,h)=[q_{h'_1},q_{h'_1a_2})\times ...\times
  [q_{h'_k},q_{h'_ka_2}) \hbox{ with }[q_{h'_j},q_{h'_ja_2})
  \in\mathscr{I}_{a_2j} \hbox{ for }1\leq j\leq k\]
  and $\mathscr{I}^*(k,a_2)$ be the family of all $I(k,a_2,h)$'s.
  Then \[\mathscr{I}^*(k,a_2)=\{I(k,a_2,h):h\in N\}.\]
  Then $\cup\mathscr{I}^*(k,a_2)=I(k,1_1,t)$ by the definition of $\mathscr{I}_{a_2j}$ for
  $1\leq j\leq k$.
  Let
  \[\mathscr{I}^+(k,a_2)=\{I(k,a_2,h)\in\mathscr{I}^*(k,a_2):
  I(k,a_2,h)\cap I(mv^\alpha_1,t)\neq\emptyset\}.\]
  Then $\mathscr{I}^+(k,a_2)=\{I(k,a_2,h):h\in N\}$ with $\Delta$-order.
  Then \[P(k,1_1)\in I(k,a_2,1)\quad \hbox{ and }\quad P(k,1_2)\in I(k,a_2,2)=I(k,1_2).\]
  Then $I(k,a_2,1)\cap I(k,a_2,2)=\emptyset$ since $\mathscr{I}_{a_2j}$ is pairwise
  disjoint for $1\leq j\leq k$. Then
  \[I(k,a_2,1)=[q_{1_1},q_{1_1a_2})\times ...\times[q_{1_m},q_{1_ma_2})\times
  [q_{1_b},q_{1_ba_2})\times...\times
  [q_{1_k},q_{1_ka_2}).\]
\par
\vspace{0.3cm}

 Let
  \[\mathscr{I}_{a_2j}=\mathscr{I}_{a_2}|[q_{1_j},q_{1_ja})\quad
  \hbox{ for }\quad k<j\leq a_2=k+,\]
  \[J(k,a_2,h)=[q_{h'_{k+1}},q_{h'_{k+1}a_2})\times ...\times
  [q_{h'_{k+}},q_{h'_{k+}a_2})\times X_{a_2}\hbox{ with }[q_{h'_j},q_{h'_ja_2})
  \in\mathscr{I}_{a_2j}\] for $k<j\leq a_2=k+$,
  and $\mathscr{J}(k,a_2)$ be the family of all $J(k,a_2,h)$'s.
  Then \[\mathscr{J}(k,a_2)=\{J(k,a_2,h):h\in N\}.\]
  Then $\cup\mathscr{J}(k,a_2)=J(k,1_1,t)$ by the definition of $\mathscr{I}_{a_2j}$.
  Let \[\mathscr{G}(k,1_2)=\{I(k,1_2)\times J(k,a_2,l):J(k,a_2,l)\in\mathscr{J}(k,a_2)\}.\]
  Then $\mathscr{G}(k,1_2)$ is a Ln cover of $H(k,1_2)\cap g(k,t)$
  and $H(k,1_2)\cap g(k,t)\subset H(mv^\alpha_1,-t)$.
\par
\vspace{0.3cm}

\textbf{1C.} Assume we have had $\mathscr{G}(k,1_i)$ for $i<j$.
  Following $\Delta$-order (direction of the arrow) in Figure I,
  take the first point $P(n,i)\times\{q_{j'}\}$ with
   \[P(n,i')\times\{q_{j'}\}=P(k,1_j)\in I(k,1_1,t)-\cup_{i<j}I(k,1_i)
   \hbox{ and }P(k,1_j)\neq P(k,1_1).\]
   Let $a_j=k+1_j$.
   Then, in the same way as 1B, we have $I(k,1_j)$, $\mathscr{I}^+(k,a_j)$,
   $\mathscr{J}(k,a_j)$
   and $\mathscr{G}(k,1_j)$.
\par

Then $I(k,1_j)\cap [\cup_{i<j}I(k,1_i)]=\emptyset$.
\par

In fact, suppose $I(k,1_j)\cap [\cup_{i<j}I(k,1_i)]\neq\emptyset$.
  Then $I(k,1_j)\cap I(k,1_i)\neq\emptyset$ for some  $i<j$.
  Then $I(k,1_j)\subset I(k,1_i)$ by 8 of Proposition 3.3 since $i<j$.
  Then $P(k,1_j)\in I(k,1_j)\subset I(k,1_i)$,
  a contradiction to $P(k,1_j)\in I(k,1_1,t)-\cup_{i<j}I(k,1_i)$.
\par

 Then, by induction, we have $I(k,1_j)$, $\mathscr{I}^+(k,a_j)$
   and $\mathscr{G}(k,1_j)$ every $j\in N$ satisfying
   claim:
   \par

1. $P(k,1_1)\in I(k,a_j,1)$, $P(k,1_j)\in I(k,a_j,\ell)=I(k,1_j)$
  \[(q_{1_1},...,q_{1_m})\in I_m(k,a_j,1)\cap I_m(k,1_j)
  \quad\hbox{ and}\]
   $I(k,a_j,1)\cap I(k,a_j,\ell)=\emptyset$.
\par

2. $\mathscr{G}(k,1_j)$ is a Ln cover of $H(k,1_j)\cap g(k,t)$
  and $H(k,1_j)\cap g(k,t)\subset H(mv^\alpha_1,-t)$.
\par

3. $I(k,1_j)\cap [\cup_{i<j}I(k,1_i)]=\emptyset$.
\par

Then, by induction, we have $I(k,1_j)$, $I(k,a_j,1)$ and $\mathscr{G}(k,1_j)$
  for $j\in N$.
Let \[\mathscr{I}(k,-t)=\{I(k,1_j): j\in N\},\]
 \[\mathscr{I}(k,a_\sigma, t)=\{I(k,a_j,1): j\in N\}\quad\hbox{ and  }\quad
 \mathscr{G}(k,-t)=\cup_j\mathscr{G}(k,1_j).\]
  Then
 $\mathscr{G}(k,-t)$ is a Ln cover of $H(mv^\alpha_1,-t)$.
\par

In fact, pick an $x\in H(mv^\alpha_1,-t)$.
  Let
   \[I(mv^\alpha_1,-t)=(q_{1_1},...,q_{1_m})\times
  [q_{1_b},q_{1_ba})\times...\times[q_{1_k},q_{1_ka})-P(k,1_1),\]
  $x=P(k,1_j)\times(q'_{k+1},q'_{k+2},...)$ and
  $P(k,1_j)\in I(mv^\alpha_1,-t)$.
  If $P(k,1_j)\in\cup_{i<j} I(k,1_j)$, then
  $x\in\cup_{i<j}\mathscr{G}(k,1_i)$.
  If  $P(k,1_j)\notin\cup_{i<j} I(k,1_j)$.
  Then $P(k,1_j)\in I(k,1_j)$ and $x\in \cup\mathscr{G}(k,1_j)$.
  Then $H(mv^\alpha_1,-t)\subset \cup\mathscr{G}(k,1_j)$.
\vspace{0.3cm}

\textbf{2.}
Take an $I(k,1_j)=[q_{1_1},q_{1_1a_j})\times ...\times[q_{1_m},q_{1_ma_j})
  \times ...\times
  [q_{j_k},q_{j_ka_j})\in\mathscr{I}(k,-t)$.
  Let $I_m(k,1_j)=[q_{1_1},q_{1_1a_j})\times ...\times[q_{1_m},q_{1_ma_j})$ and
  \[\mathscr{I}_m(k,-t)=\{I_m(k,1_h):
  I(k,1_h)\in\mathscr{I}(k,-t)\}.\]
\vspace{0.3cm}

\textbf{3.}
  Let $p>v^\alpha_1$ and $y'\in H(v^\alpha_1,t)$.
  Then $H(p,y')\subset H(v^\alpha_1,t)$.
  Then there exists a $c$ with
  $g(m,l_\ell,t)\subset g_c(p,y')$ if $l_\ell>c$ and
  $H(m,l_\ell,t)\cap g_c(p,y')\neq\emptyset$
  by 1 of Corollary 5.3.
  Let $l_\ell$ be the first number.
  Then $H'(m,l_\ell)\cap g_c(p,y')\neq\emptyset$
  and $H'(m,l_\ell)\cap H'(m,l_{j0})=\emptyset$.
  Then $l_\ell>l_{j0}$ since
  $H(v^\alpha_1,t)\subset H(mv^\alpha_1,t)=g(v^\alpha_1,t)\cap H'(m,l_{j0})$.
  Then $H'(m,l_\ell)=[(q'_1,...,q'_m)\times X_m]\cap G_m$,
  \[P(m,l_\ell)=(q'_1,...,q'_m)\neq(q_{1_1},...,q_{1_m})=P(m,l_{j0})\quad\hbox{ and}\]
  \[H'(m,l_\ell)\cap g(k,t)
  =(q'_1,...,q'_m)\times
  [q_{1_b},q_{1_ba})\times ...\times
  [q_{1_k},q_{1_ka})\times J(k,t).\]
  Then there exists an $h\leq m$ with $q'_h\neq q_h$.
  Note $\cap_jI_m(k,1_j)=\{(q_{1_1},...,q_{1_m})\}$.
  Let $\pi_h:Y\rightarrow Q_h$ be a project map.
  Then there exists an $n_0$ such that
  \[q'_h\notin [q_{1_h},q_{1_ha_j})=\pi_h[I_m(k,1_j)],
  \quad\hbox{ and }\quad q_{1_ha_j}< q'_h
  \hbox{ for every } j\geq n_0.\]
\par

Then $q'_h\in (q_{1_h},q_{1_ha_j})\subset [q_{1_h},q_{1_ha_j})=\pi_h[I_m(k,1_j)]$ if $j<n_0$.
  Then we may assume
   \[P(m,l_\ell)=(q'_1,...,q'_m)\in I_m(k,1_j)
  \quad\hbox{ for every }\quad j<n_0.\]
\par

\textbf{3A.} Note  $g_c(p,y')=I(p,y',c)\times J(p,y',c).$
  Let $\mathscr{G}(m,l_\ell)$ be a Ln cover of $H'(m,l_\ell)$
  and
   \[\mathscr{G}^*(m,y'l_\ell)=\{g(m,x)\in\mathscr{G}(m,l_\ell):
  g(m,x)\subset g_c(p,y')\}.\]
  Then, by Fact 6.11 and Fact 6.10, we have
  $\mathscr{G}^*(m,y'l_\ell)\neq\emptyset$.
  Let $e=m+l_\ell$ and $g(m,x)=g(m,l_\ell,x)\in \mathscr{G}^*(m,y'l_\ell)$.
  Then $g(m,l_\ell,x)=I(m,l_\ell,x)\times J(m,l_\ell,x)$,
  \[P(m,l_\ell)=(q'_1,...,q'_m)\in I(m,l_\ell,x)=[q'_1,q'_{1e}) \times...\times
 [q'_m,q'_{me})\quad \hbox{ and}\]
 \[I(m,l_\ell,x)\cap I_m(k,1_j)=\emptyset\quad \hbox{ if }  I_m(k,1_j)\in\mathscr{I}_m(k,-t)\hbox{ with } j\geq n_0.\]
\par

Let $\mathscr{G}^*(m_1,y'l_\ell)=\{g(m_1,z)\in\mathscr{G}(m_1,l_\ell):
  g(m_1,z)\subset g_c(p,y')\}.$
  Then, by Fact 6.11 and Fact 6.10, we have
  $\mathscr{G}^*(m_1,y'l_\ell)\neq\emptyset$.
\par
\vspace{0.3cm}

\textbf{3B.} Let $b=m+1$. Note $H(p,y')\subset H(v^\alpha_1,t)=
  (q_{1_1},...,q_{1_m},q_{1_b},...,q_{1_k})\times J(k,t)$
  and $g(m,l_\ell,x)\subset g_c(p,y')$.
Let  \[I(p,y',c)= [q_{1_1},q_{1_1c}) \times...\times
 [q_{1_m},q_{1_mc})\times...\times
 [q_{1_k},q_{1_kc})\times...\times[q_{1_p},q_{1_p c}),\]
 \[g_c(p,y')=I(p,y',c)\times J(p,y',c)\quad\hbox{ and}\]
 \[H(mp,y',c)=(q'_1,...,q'_m)\times
 [q_{1_b},q_{1_bc})\times...\times
 [q_{1_p},q_{1_pc})\times J(p,y',c).\]
 Then $H(mp,y',c)=g_c(p,y')\cap H'(m,l_\ell)$.
 Note $g(m,l_\ell,x)\subset g_c(p,y')$.
 Then
 \[H(m,l_\ell,x)=H'(m,l_\ell)\cap g(m,l_\ell,x)
 \subset H'(m,l_\ell)\cap g_c(p,y')=H(mp,y',c).\]
 Let $\mathscr{G}^*(m_1,xy')=\{g(m_1,z)\in\mathscr{G}^*(m_1,y'l_\ell):
 H(m_1,z)\cap H(m,l_\ell,x)\neq\emptyset\}$. Then
 $\mathscr{G}^*(m_1,xy')=\cup\{\mathscr{G}(m_1,xl'_\ell):
 l'_\ell\in N_1\}$ such that
 $N_1$ is infinite by Corollary 4.7, and
 $\mathscr{G}(m_1,xl'_\ell)=\{g(m_1,l'_\ell,z_i)\in\mathscr{G}^*(m_1,y'l_\ell):
 i\in N(l'_\ell)\}\neq\emptyset$ for every $l'_\ell\in N_1$.
 Let $\mathscr{G}(m_1,l_\ell)$ be a Ln cover of $H'(m,l_\ell)$
 and
 \[\mathscr{G}(mp,y',c)=\{g(m_1,z)\in\mathscr{G}(m_1,l_\ell):
  g(m_1,z)\cap H(mp,y',c)\neq\emptyset\}.\]
  Then $H(mp,y',c)\subset\cup\mathscr{G}(mp,y',c)$
  and $\mathscr{G}(mp,y',c)\subset \mathscr{G}^*(m_1,y'l_\ell)$.
  Then we have $\mathscr{G}^*(m_1,xy')\subset\mathscr{G}(mp,y',c)$.
\par

\vspace{0.3cm}

\textbf{3C.}
 Let $\mathscr{I}_m(n_0,-t)
  =\{I_m(k,1_j):I(k,1_j)\in\mathscr{I}(mv^\alpha_1,-t) \hbox{ with }j<n_0\}$.
  Take $g(m_1,l'_\ell,z)\in\mathscr{G}(mp,y',c)$
  with $l'_\ell>a_j$.
  Let $e'=m_1+l'_\ell$, $b=m+1=m_1$,
  \[I(m_1,l'_\ell,z)= [q'_1,q'_{1e'}) \times...\times
  [q'_m,q'_{me'})\times[q'_b,q'_{be'})\quad\hbox{ and}\]
  \[J(m_1,l'_\ell,z)=[q'_{b+1},q'_{b+1e'}) \times...\times
  [q'_k,q'_{ke'})\times...\times[q'_{e'},q'_{e'e'})\times X_{e'}.\]
  Then $g(m_1,l'_\ell,z)=I(m_1,l'_\ell,z)\times J(m_1,l'_\ell,z)
  \in \mathscr{G}(m_1,l'_\ell)$ by Definition C in Section 3. Let
  \[I_m(m_1,l'_\ell,z)= [q'_1,q'_{1e'}) \times...\times
  [q'_m,q'_{me'}),\]
  \[I_{bk}(m_1,l'_\ell,z)=[q'_b,q'_{be'})\times ...\times[q'_k,q'_{ke'})\quad\hbox{ and}\]
  \[J_k(m_1,l'_\ell,z)=[q'_{k+1},q'_{k+1e'}) \times...
  \times[q'_{e'},q'_{e'e'})\times X_{e'}.\]
\par
\vspace{0.3cm}

\textbf{3D.}
On the other hand, take an $I(k,a_j,1)\in\mathscr{I}(k,a_\sigma, t)$.
  Then
  \[P(k,1_1)\in I(k,a_j,1)=
  [q_{1_1},q_{1_1a_j})\times ...\times[q_{1_m},q_{1_ma_j})\times[q_{1_b},q_{1_ba_j})
  \times ...\times[q_{1_k},q_{1_ka_j}).\]
  Let $I_{bk}(k,a_j,1)=[q_{1_b},q_{1_ba_j})\times ...\times[q_{1_k},q_{1_ka_j})$
  and $P_{bk}(k,1_1)=(q_{1_b},...,q_{1_k})$.
  Then
  \[P_{bk}(k,1_1)=(q_{1_b},...,q_{1_k})\in I_{bk}(k,a_j,1).\]
\par

Note $H(mp,y',c)\subset\cup\mathscr{G}(mp,y',c)$.
  Take a $g(m_1,l'_\ell,z)\in\mathscr{G}(mp,y',c)$
  with $P_{bk}(k,1_1)\in I_{bk}(m_1,l'_\ell,z)$.
  Then $I_{bk}(m_1,l'_\ell,z)=
  [q_{1_b},q_{1_be'})\times ...\times[q_{1_k},q_{1_ke'})$
  since $q_{1_i}\in Q^*_{a_j}\subset Q^*_{e'}$
  for $1_b\leq 1_i\leq 1_k$ by $e'>a_j$.
  Then \[I_{bk}(m_1,l'_\ell,z)\subset I_{bk}(k,a_j,1)\]
  by 1 of Proposition 3.2 since $e'>a_j$.
  Let
  \[\mathscr{G}^*(mp,y',c)=\{g(m_1,z)\in\mathscr{G}(mp,y',c):
  P_{bk}(k,1_1)\in I_{bk}(m_1,l'_\ell,z)\}.\]
  Then we have $\mathscr{G}^*(mp,y',c)\neq\emptyset$
  since $g(m_1,l'_\ell,z)\in\mathscr{G}^*(mp,y',c)$.
  Take a $g(m_1,z)$ from $\mathscr{G}^*(mp,y',c)$.
  It is easy to see $g(m_1,z)\subset g_c(p,y')$ and $H(m_1,z)\subset H'(m,l_\ell)$
  by the definition of $\mathscr{G}(mp,y',c)$.
  And then we have the following claim.
\par

3D1. $g(m_1,z)\cap [\cup\mathscr{G}(k,1_j)]=\emptyset$ if $j<n_0$.
\par

In fact, note $I(k,a_j,1)\cap I(k,1_j)=\emptyset$.
  Then we have $I_{bk}(m_1,l'_\ell,z)\cap I_{bk}(k,1_j)=\emptyset$ for each $j<n_0$
  since $(q_{1_1},...,q_{1_m})\in I_m(k,a_j,1)\cap I_m(k,1_j)$
  by 1 of Claim 1C,
  and $I_{bk}(m_1,l'_\ell,z)\subset I_{bk}(k,a_j,1)$.
  Then $I_k(m_1,l'_\ell,z)\cap I(k,1_j)=\emptyset$
  for every $j<n_0$.
  Then $g(m_1,z)\cap [\cup\mathscr{G}(k,1_j)]=\emptyset$ if $j<n_0$.
\par

3D2. $g(m_1,z)\cap[\cup\mathscr{G}(k,1_j)]=\emptyset$ if $j\geq n_0$.
\par

In fact, $j\geq n_0$ implies
  $q'_h\notin[q_{1_h},q_{1_ha_j})=p_h[I_m(k,1_j)]$ for some $h\leq m$ by the above \textbf{3.}
  Then $q_{1_ha_j}< q'_h$ since $q'_h\in [q_{1_h},q_{1_ha_i})=p_h[I_m(k,1_i)]$ if $i<n_0$.
  Then we have $[q_{1_h},q_{1_ha_j})\cap [q'_h,q'_{he})=\emptyset$.
  Then $g(m_1,z)\cap [\cup\mathscr{G}(k,1_j)]=\emptyset$ if $j\geq n_0$.
\par
\vspace{0.3cm}

\textbf{4.}
Take a $g(k,x)$ with $x\in H(mv^\alpha_1,-t)$.
  Then there exists a $g(k,x')\in \mathscr{G}(k,-t)$
  with $g(k,x)\cap g(k,x')\neq\emptyset$
  since $\mathscr{G}(k,-t)$ is a Ln cover of $H(mv^\alpha_1,-t)$.
  Then $g(k,x)\subset g(k,x')$
  by the definition of Ln covers.
  Then, for every $g(m_1,z)\in\mathscr{G}^*(mp,y',c)$, we have
  \[g(m_1,z)\cap g[k, H(mv^\alpha_1,-t)]=\emptyset.\]
  Note $\ell>k$ implies $g[\ell, H(mv^\alpha_1,-t)]\subset g[k, H(mv^\alpha_1,-t)]$.
  Then $g(m_1,z)\subset g_c(p,y')$
  and $[g(m_1,z)\cap E(\alpha k,p_k)]-g[\ell,U_s]\neq\emptyset$.
\end{proof}

1. Let $x\notin E(\alpha k,p_k)$. Then  $U(\alpha\ell k,p_kx)=g(\ell,x)-E(\alpha k,p_k)$
  is open by the definition A and B.
\par

2.  Let $\ell>m_5$, $x\in E(\alpha k,p_k)$ and \[U(\alpha\ell k,p_kx)=g(\ell,x)\cap E(\alpha k,p_k).\]
  Then $U(\alpha\ell k,p_kx)=g(\ell,x)\cap E(\alpha k,p_k)$ is open by
  the definition A and B.
\par

3.  Let $x\in Y$ and
  \[\mathscr{U}(x)=\{U(\alpha\ell k,p_kx): \alpha\in\mathscr{A},
  \ell>m_5,k\geq 2\hbox{ and }  p_k>m_5 \}.\]
\par

And then let $\mathscr{U}'\subset\mathscr{U}(x)$ be finite,
  $U(\alpha\ell k,p_kx)\in \mathscr{U}'$,
  $U[\alpha^-\ell k',p'_kx]=\cap\mathscr{U}'$ with $\alpha^-=\{\alpha_i: U(\alpha_i\ell_i k_i,p^i_kx)\in\mathscr{U}'\}$,
  \[\mathscr{U}_x=\{U[\alpha^-\ell k',p'_kx]: \mathscr{U}'\subset\mathscr{U}(x)
  \hbox{ is finite}\}\quad \hbox{ and } \quad \mathscr{U}=\cup\{\mathscr{U}_x: x\in Y\}.\]
  Then $\mathscr{U}$ is a family of open sets by the definition A and B.

\vspace{0.3cm}
A base $\mathscr{U}_x$ of neighborhoods of $x$
  is called an outer base of $x$ in Definition 1.3
  of \cite{t} also.

\begin{prop}
1. $\mathscr{U}_x$ is a base of neighborhoods of $x$ in some
  topological space $(Y,\tau)$ for every $x\in Y$.
\par

2. $\mathscr{H}_Y=\cup\{\mathscr{H}_x: x\in Y\}$ is a countable network of $(Y,\tau)$.
\par

3. $\mathscr{U}$ is a base of some topological space $(Y,\tau)$.
\end{prop}

\begin{proof}
Let $U(\alpha\ell k,p_kx),U(\beta\ell' n,p_nx)\in\mathscr{U}(x)$ and
  $t\in U(\alpha\ell k,p_kx)\cap U(\beta\ell'n,p_nx)$.
  Let $H(u^\alpha_1,t_\alpha)\subset U(\alpha\ell k,p_kx)$ and
  $H(v^\beta_1,t_\beta)\subset U(\beta\ell'n,p_nx)$.
  Pick a point $t$ such that $t\in H(u^\alpha_1,t_\alpha)\cap H(v^\beta_1,t_\beta)$.
\par

Case 1, $U(\beta\ell' n,p_nx)=g(\ell',x)-E(\beta n,p_n)$ and
  $U(\alpha\ell k,p_kx)=g(\ell,x)-E(\alpha k,p_k)$.
  Then $H(u^\alpha_1,t_\alpha)\cap E(\alpha k,p_k)=\emptyset$ and
  $H(v^\beta_1,t_\beta)\cap E(\beta n,p_n)=\emptyset$. Let
   \[l>\hbox{max}\{u^\alpha_1,v^\beta_1,\ell',\ell\}.\]
  Then $g(l,t)\cap U(\alpha\ell k,p_kx)=
  g(l,t)\cap[g(\ell,x)-E(\alpha k,p_k)]$.
  Then \[U(\alpha l k,p_kt)=g(l,t)\cap U(\alpha\ell k,p_kx)
  =g(l,t)-E(\alpha k,p_k)\in\mathscr{U}(t).\]
  Then, in the similar way, we have
  \[g(l,t)\cap U(\beta\ell' n,p_nx)=g(l,t)-E(\beta n,p_n)
  =U(\beta l n,p_nt)\in \mathscr{U}(t).\]
  Then, by the definitions of $H(u^\alpha_1,t_\alpha)$ and $H(v^\beta_1,t_\beta)$, we have
  \[H(l,t)\subset U[\alpha^- l k,p_kt]=U(\alpha l k,p_kt)\cap U(\beta l n,p_nt)\in \mathscr{U}_t.\]
\par

Case 2, $U(\beta\ell' n,p_nx)=g(\ell',x)-E(\beta n,p_n)$ and
  $U(\alpha\ell k,p_kx)=g(\ell,x)\cap E(\alpha k,p_k)$.
  Then $H(v^\beta_1,t_\beta)\cap E(\beta n,p_n)=\emptyset$
  and $H(a,t_\alpha)\subset g(\ell,x)\cap E(\alpha k,p_k)$.
  Let
   \[l>\hbox{max}\{a,v^\beta_1,\ell',\ell\}.\]
  Then $g(l,t)\cap U(\alpha\ell k,p_kx)=
  g(l,t)\cap[g(\ell,x)\cap E(\alpha k,p_k)]$.
  Then \[U(\alpha l k,p_kt)=g(l,t)\cap U(\alpha\ell k,p_kx)
  =g(l,t)\cap E(\alpha k,p_k)\in\mathscr{U}(t).\]
  Then, in the similar way, we have
  \[g(l,t)\cap U(\beta\ell' n,p_nx)=g(l,t)-E(\beta n,p_n)
  =U(\beta l n,p_nt)\in \mathscr{U}(t).\]
  Then, by the definitions of $H(u^\alpha_1,t_\alpha)$ and $H(v^\beta_1,t_\beta)$, we have
  \[H(l,t)\subset U[\alpha^- l k,p_kt]=U(\alpha l k,p_kt)\cap U(\beta l n,p_nt)\in \mathscr{U}_t.\]
\par

Case 3, $U(\beta\ell' n,p_nx)=g(\ell',x)\cap E(\beta n,p_n)$ and
  $U(\alpha\ell k,p_kx)=g(\ell,x)\cap E(\alpha k,p_k)$.
  Then $H(b,t_\beta)\subset g(\ell',x)\cap E(\beta n,p_n)$
  and $H(a,t_\alpha)\subset g(\ell,x)\cap E(\alpha k,p_k)$.
  Let
   \[l>\hbox{max}\{a,b,\ell',\ell\}.\]
  Then we have $g(l,t)\cap U(\alpha\ell k,p_kx)=
  g(l,t)\cap E(\alpha k,p_k)=U(\alpha lk,p_kt)\in\mathscr{U}(t)$ and
  $g(l,t)\cap U(\beta\ell' n,p_nx)=U(\beta l n,p_nt)\in \mathscr{U}(t).$
  Then, by the definitions of $H(u^\alpha_1,t_\alpha)$ and $H(v^\beta_1,t_\beta)$, we have
  \[H(l,t)\subset U[\alpha^- l k',p'_kt]=U(\alpha l k,p_kt)\cap U(\beta l n,p_nt)\in \mathscr{U}_t.\]

\vspace{0.3cm}

Let $U[\alpha^-\ell k',p'_kx],U[\beta^-\ell' n',p'_ny]\in\mathscr{U}$ and
  $t\in U[\alpha^-\ell k',p'_kx]\cap U[\beta^-\ell'n',p'_ny]$ in general.
Then, in the same way as the above proof,
  there exists a $g(l,t)$ and a $U[\gamma^- lk',p'_kt]\in\mathscr{U}_t$
  such that
  \[H(l,t)\subset U[\gamma^- l h',p'_ht]\subset
  U[\alpha^-\ell k',p'_kx]\cap U[\beta^-\ell'n',p'_ny].\]
  Then, by Proposition 1.2.3 in \cite{En}, $\mathscr{U}$ is a base for some
  topological space $(Y,\tau)$.
\end{proof}

  Let $B_0$ be a defined set in Definitions A-C. Then $B_0$ is c.o in $(Y,\tau)$
  because both $B_0$ and $Y-B_0$ are in $\mathscr{U}$.
  Denote the closure of $B_0$ in $(Y,\tau)$ by $Cl^*B_0$, and
  in Definitions of A-C by $\hbox{Cl}_\tau B_0$.
  Then \[Cl^*B_0=B_0=\hbox{Cl}_\tau B_0.\]
\par

Take  $E(\alpha k,p_k)=D(\alpha k,p_k)\cup H(k\sigma^\alpha_1,q_k\sigma)$
   from Definition A-B.
\par

\textbf{Note $\complement$.} $Cl^*D(\alpha k,p_k)=Cl_\tau D(\alpha k,p_k)$
  for every $\alpha\in\mathscr{A}$, $k\geq 2$ and $p_k\in N$.
\par

In fact, let $t\notin Cl_\tau D(\alpha k,p_k)=E(\alpha k,p_k)$.
  Then $U(\alpha l k,p_kt)=g(l,t)-E(\alpha k,p_k)$ satisfies
  $U(\alpha l k,p_kt)\cap E(\alpha k,p_k)=\emptyset$.
  Note $Y-E(\alpha k,p_k)\in\mathscr{U}$.
  Then $D(\alpha k,p_k)\subset E(\alpha k,p_k)$ implies
   $Cl^*D(\alpha k,p_k)\subset E(\alpha k,p_k)$.
  Then $t\notin Cl^*D(\alpha k,p_k)$. Then
  \[Cl^*D(\alpha k,p_k)\subset Cl_\tau D(\alpha k,p_k).\]
\par

Pick an $x\in H(k\sigma^\alpha_1,q_k\sigma)$.
  Then there exists an $H(v^\alpha_1,x)\in\mathscr{H}_d(k\sigma^\alpha_1,q_k\sigma)$
  with $H(v^\alpha_1,x)\cap D(\alpha k,p_k)=\emptyset$.
  Take a $g_c(\ell,x)$ for $\ell>m_5$.
  Then, by Fact 6.10-6.11 and 3 of Proposition 6.5,
  there exists a full $H(m_1,z)$ with $H(m_1,z)\subset g_c(\ell,x)$ and $g(m_1,z)\in \mathscr{G}^*(m_1,y'1_\ell)$.
  Then, for every $ E(\beta h,p_h)$, we have
  \[H(m_1,z)\cap [g_c(\ell,x)\cap E(\beta h,p_h)]\neq\emptyset\quad \hbox{and} \quad
  H(m_1,z)\cap [g_c(\ell,x)-E(\beta h,p_h)]\neq\emptyset.\]
  Then $x\in Cl^*D(\alpha k,p_k)$.
  Then $Cl_\tau D(\alpha k,p_k)\subset Cl^*D(\alpha k,p_k)$. $\square$
\vspace{0.3cm}

Then the closure operation $\hbox{Cl}_\tau$ in definitions A-C is the same as
  the closure operation $Cl^*$ in $(Y,\tau)$.
 \[ \textbf{Denote the topological space by $(Y,\tau)$.}\]
  So, in topological space $(Y,\tau)$, we can denote the closure
  of a set $A$ by $\hbox{Cl}_\tau A$ by Note $\complement$. Let
  \[\hbox{Int}_\tau A=Y-\hbox{Cl}_\tau(Y-A).\]

\begin{prop}
 1. $\cap_p Cl_\tau g[p,E(\alpha k,p_k)]=E(\alpha k,p_k)$.
\par

2. Let $t\in Cl_\tau g[p,E(\alpha k,p_k)]-E(\alpha k,p_k)$.
 Then there exists an $\ell$ such that:  if $l,q>\ell$,
 $H(l,t)\cap Cl_\tau g[q,E(\alpha k,q_k)]=\emptyset$.
\end{prop}

\begin{proof}
Let $p>p_k$. Then $g[p,E(\alpha k,p_k)]=D(\alpha k,p_k)\cup g[p,H(k\sigma^\alpha_1,q_k\sigma)]$.
  Then, by 1 of Definition B,
  \[Cl_\tau g[p,E(\alpha k,p_k)]=g[p,E(\alpha k,p_k)]\cup H(h\sigma^\alpha_1,q_h\sigma).\]
  Here $h=k+1$ and $q_h=\hbox{max}\{p_k,q_n\}$.
\par

  Let $t\in H(m_1,jz)=H'(m,1_j)\cap g(m_1,z)$
  for some $g(m_1,z)\in\mathscr{G}^*(m_1,1_j)$,
  \[\mathscr{H}(t,1_j-)=
  \{H'(m,1_i)\in\mathscr{H}'(m,y_0):i<j\},
  \quad H_-=\cup\mathscr{H}(m,1_j-),\]
  \[\mathscr{H}(m,1_j+)=
  \{H'(m,1_i)\in\mathscr{H}'(m,y_0):i>j\}\
  \hbox{ and  }\ H_+=\cup\mathscr{H}(m,1_j+).\]
\par

1.  $t\notin g[p,H_+]$ if $p>m_1=\ell_1$ by 6$'$ of Proposition 3.3.
\par

2.  $t\notin g[p,H_-\cap E(\alpha k,p_k)]$ if $p>\ell_2$
  for some $\ell_2$ since $\mathscr{H}(t_0,1_j-)$ is finite.
\par

3. Note $p>p_k$ and $t\in \Big[H(m_1,jz)\cap
  Cl_\tau g[p,E(\alpha k,p_k)]\Big] -E(\alpha k,p_k)$.
  Then $t\notin D(\alpha k,p_k)$ since
  $g[p,D(\alpha k,p_k)]=D(\alpha k,p_k)\subset E(\alpha k,p_k)$
  if $p>p_k$.
  Then there exists an $\ell_3$ such that
  $l>\ell_3$ implies $H(l,t)\cap D(\alpha k,p_k)=\emptyset$ by Fact 6.16.
\par

4. $t\in \Big[g[p,H(k\sigma^\alpha_1,q_k\sigma)]\cap
  H(m_1,jz)\Big]- E(\alpha k,p_k)$.
  Then, by the definition of $H(k\sigma^\alpha_1,q_k\sigma)$, there exists an $\mathscr{H}_d(v^\alpha_1,jz^l_v)\subset \mathscr{H}(k\sigma^\alpha_1,jz,q_k\sigma)$ such that
   \[t\in g[p,\cup\mathscr{H}_d(v^\alpha_1,jz^l_v)]-E(\alpha k,p_k).\]
   Then there exists an $H(v^\alpha_1,s')\in\mathscr{H}_d(v^\alpha_1,jz^l_v)$
   such that $t\in g[p,H(v^\alpha_1,s')]-E(\alpha k,p_k)$.
  Note $H(v^\alpha_1,s')$ is closed in $(Y,\rho)$.
  Then there exists $\ell_4>v^\alpha_1$ such that
  $H(h,t)\cap g[q,H(v^\alpha_1,s')]=\emptyset$
  if $h,q>\ell_4$ by 1 of Claim 6.8.
  Note $p>p_k>m_5$,
  \[t\in g[p,\cup\mathscr{H}_d(v^\alpha_1,jz^l_v)]\cap H(m_1,jz)
  \subset H(m_3,jz^l_v)\quad\hbox{ and}\]
  $H(m_3,jz^l_v)\cap H(m_3,jz^n_u)=\emptyset$ if $l\neq n$
  or $v\neq u$.
  Then
  \[H(h,t)\cap g[q,H(k\sigma^\alpha_1,q_k\sigma)]=\emptyset.\]
\par

5. $t\in H(h\sigma^\alpha_1,q_h\sigma)$. Note
  $H(h\sigma^\alpha_1,q_h\sigma)\cap g[p,H(k\sigma^\alpha_1,q_k\sigma)]=\emptyset$.
  Then, in the same way as the proof of Fact 6.17,
  there exists an $\ell_5$ such that
  $H(l,t)\cap H(h\sigma^\alpha_1,q^*_h\sigma)=\emptyset$ if $q^*_h,l>\ell_5$.
\par

Let $\ell,q>\hbox{max}\{\ell_i: i\leq 5\}$.
  Then, by 1 of Definition B, we have
  \[Cl_\tau g[q,E(\alpha k,p_k)]=g[q,E(\alpha k,p_k)]
  \cup H(h\sigma^\alpha_1,q_h\sigma).\]
  Here $q_h=\hbox{max}\{q,q_k\}$ and $h=k+1$.
  Then $H(\ell,t)\cap Cl_\tau g[q,E(\alpha k,p_k)]=\emptyset$.
\end{proof}

\begin{prop}
Let $U(\alpha\ell k,p_kx)=g(\ell,x)-E(\alpha k,p_k)$.
 Then:
\par

1.  $\cap_pCl_\tau g[p,U(\alpha\ell k,p_kx)] =U(\alpha\ell k,p_kx)$.
\par

2. Let $t\in Cl_\tau g[p,U(\alpha\ell k,p_kx)]-U(\alpha\ell k,p_kx)$.
 Then there exists an $\ell$ such that
 $H(l,t)\cap Cl_\tau g[q,U(\alpha\ell k,p_kx)]=\emptyset$ if $l,q>\ell$.
\end{prop}

\begin{proof}
Let $U_x=U(\alpha\ell k,p_kx)$, $p>p_k$ and $t\in Cl_\tau g[p,U_x]-U_x$.
 Then  $t\in g[p,U_x]-U_x$ by Definition C.

Let $t\in H_j= H'(m,l_j)$ for some $H'(m,l_j)\in\mathscr{H}'(m,y_0)$,
 \[\mathscr{H}(m,j+)=\{H(m,l_i)\in\mathscr{H}(m,y_0): i>j\},
  \quad H_+=\cup\mathscr{H}(m,j+),\]
 \[\mathscr{H}(m,j-)=\{H(m,l_j)\in\mathscr{H}(m,y_0): i<j\}
 \quad\hbox{ and }\quad H_-=\cup\mathscr{H}(m,j-).\]
  Then $\mathscr{H}(m,j-)$ is finite.
\par

\textbf{1.} Let $\ell_1=m_1$. Then $g[q, H_+]\cap H_j=\emptyset$ if $q>\ell_1$.
\par

\textbf{2.} Note $\rho(H_-,H'(m,l_j))=r>0$ by 2 of Proposition 3.1.
 Then there exists an $\ell_2$ such that
 $g[q, H_-\cap U_x]\cap H_j=\emptyset$ if $q>\ell_2$.
\par

\textbf{3.} Note $U_x=U(\alpha\ell k,p_kx)=g(\ell,x)-E(\alpha k,p_k)$.
   Then
   \[H'(m,l_j)\cap U_x=[g(\ell,x)\cap H'(m,l_j)]-[D(\alpha k,p_k)\cup H(k\sigma^\alpha_1,q_k\sigma)].\]
    Let $\mathscr{G}(m_1,k_j)$ be a Ln cover of $H'(m,1_j)$.
    From Q2 of Construction 6.1, take
  \[H^*(m,k_j)=H'(m,1_j)-\cup_{i<j}D(p_k,m^\alpha,k_i)\quad\hbox{ and}\]
  \[\mathscr{G}^*(m_1,k_j)=\{g(m_1,z)\in\mathscr{G}(m_1,k_j):
  g(m_1,z)\cap H^*(m_1,k_j)\neq\emptyset\}.\]

\textbf{Fact 7.4.1.} \textit{Let $g(m_1,z)\in \mathscr{G}^*(m_1,k_j)$
  and $s\in H(m_2,jz_n)\in\mathscr{H}(m_2,jz)$ with $n<k$.
  If $s\notin \cup_{i<j}D(p_k,m^\alpha,k_i)$, then
  $g(p,s)\cap[\cup_{i<j}D(p_k,m^\alpha,k_i)]=\emptyset$ for $p>p_k$.}
\par

\begin{proof}
Note $s\in H(m_2,jz_n)\subset H(m_1,jz)$ and $s\notin \cup_{i<j}D(p_k,m^\alpha,k_i)$.
 Take $g(m,s)$. Then $g(m,s)\cap[\cup_{i<j}H'(m,1_i)]=\emptyset$ by
 6 of Proposition 3.3 since $H(m,s)\subset H'(m,1_j)$.
 Take $H(p_k,s)$. Then $H(p_k,s)\subset H(m,s)$ since $p_k>m_5>m$.
\par

  Suppose
  $g(p_k,s)\cap[\cup_{i<j}D(p_k,m^\alpha,k_i)]\neq\emptyset$.
\par

Case 1. There exists an $n<k$ such that  $D(p_n,m^\alpha,n_i)
 \subset \cup_{i<j}D(p_k,m^\alpha,k_i)$ and
 $g(p_k,s)\cap D(p_n,m^\alpha,n_i)\neq\emptyset$.
  Then there exists a $g(p_n,s')$ with $H(p_n,s')\subset H'(m,1_i)$
  for some $i<j$ and $g(p_k,s)\cap g(p_n,s')\neq\emptyset$.
  Then $g(p_n,s)\cap g(p_n,s')\neq\emptyset$ since
  $g(p_k,s)\subset g(p_n,s)$ by $p_k>p_n$.
  Then \[g(p_k,s)\subset g(p_n,s)\subset g(p_n,s')\subset D(p_n,m^\alpha,n_i)
 \subset \cup_{i<j}D(p_k,m^\alpha,k_i).\]
 It is a contradiction to $s\notin \cup_{i<j}D(p_k,m^\alpha,k_i)$.
\par

Case 2.  Then there exists a $g(p_k,s')\subset \cup_{i<j}D(p_k,m^\alpha,k_i)$
  with $H(p_k,s')\subset \cup_{i<j}H'(m,1_i)$
  such that \[g(p_k,s)\cap g(p_k,s')\neq\emptyset.\]
  Then $g(p_k,s)\subset g(p_k,s')$ by 8 of Proposition 3.3.
  Then we have \[s\in g(p_k,s')\subset\cup_{i<j}D(p_k,m^\alpha,k_i).\]
  It is a contradiction to $s\notin \cup_{i<j}D(p_k,m^\alpha,k_i)$.
  So $g(p_k,s)\cap[\cup_{i<j}D(p_k,m^\alpha,k_i)]=\emptyset$.
  Then $g(p,s)\cap[\cup_{i<j}D(p_k,m^\alpha,k_i)]=\emptyset$ for $p>p_k$.
\end{proof}

\textbf{Fact 7.4.2.} \textit{Let $H(m_2,jz_n)\in\mathscr{H}(m_2,jz)$ with $n<k$.
  If $p>p_k\geq p_n$, then
  \[g[p,H(m_2,jz_n)\cap U_x]\subset U_x.\]}
\par

\begin{proof}
Note $\cup_vH(m_3,jz^n_v)=H(m_2,jz_n)$ by 1 of Claim 6.6.
 Take an $H(m_3,jz^n_v)$.
\par

Case 1, $v^\alpha_0<p_n$. Take $\mathscr{H}_d(v^\alpha_1,jz^n_v)$
  and $H(v^\alpha_1,t)\in\mathscr{H}_d(v^\alpha_1,jz^n_v)$.
  Then we have
  \[\emptyset\neq H(v^\alpha_0,t)-g[p_n,H(v^\alpha_1,t)]
  =H(v^\alpha_0,-v^\alpha_1t)\subset U_x\quad\hbox{ and}\]
  \[g[p_n,H(v^\alpha_0,-v^\alpha_1t)]\cap g[p_n,H(v^\alpha_1,t)]
  =\emptyset\quad\hbox{ by 8 of Proposition 3.3.}\]
  Note $\mathscr{G}(v^\alpha_0,jz^n_v)$ is a Ln cover of
  $H(m_3,jz^n_v)$.
  Let $H_d(v^\alpha_1,jz^n_v)=\cup\mathscr{H}_d(v^\alpha_1,jz^n_v)$
  and
  \[H(m_3,-jz^n_<)=\cup\{H(m_3,jz^n_v)-g[p_n,H_d(v^\alpha_1,jz^n_v)]:
  v^\alpha_0<p_n\}.\]
  Then $H(m_3,-jz^n_<)\subset U_x$ and
  $g[p_n,H(m_3,-jz^n_<)]\cap g[p_n,H_d(v^\alpha_1,jz^n_v)]=\emptyset$
  for every $v^\alpha_0<p_n$.
\par

Case 2, $v^\alpha_0=p_n$. Then $g[p_n,H(v^\alpha_1,t)]=g(v^\alpha_0,t)$
  for every $H(v^\alpha_1,t)\in\mathscr{H}_d(v^\alpha_1,jz^n_v)$.
  Then $g[p_n,\cup\mathscr{H}_d(v^\alpha_1,jz^n_v)]
  =[\cup\mathscr{G}(v^\alpha_0,jz^n_v)]\cap G(5,\alpha z^n_v)$.
  Then
  \[g[p_n,\cup\mathscr{H}_d(v^\alpha_1,jz^n_v)]\cap H(m_3,jz^n_v)
  =H(m_3,jz^n_v)\cap G(5,\alpha z^n_v).\]
  Let $H(m_3,-jz^n_v)=H(m_3,jz^n_v)-G(5,\alpha z^n_v)$.
  Note $p_n>m_5$.
  Then
  \[g[p_n,H(m_3,-jz^n_v)]\cap g[p_n,G(5,\alpha z^n_v)]=\emptyset.\]

\par

Case 3, $v^\alpha_0>p_n>m_5$. Then we have
   \[H(m_3,jz^n_v)]\cap g[p_n,H_d(v^\alpha_1,jz^n_v)]=H(m_3,jz^n_v)]\cap
   G(5,\alpha z^n_v).\]
  Let $H(m_3,-jz^n_v)=H(m_3,jz^n_v)-G(5,\alpha z^n_v)$.
  Note $p_n>m_5$.
  Then we have $g[p_n, H(m_3,-jz^n_v)]\cap G(5,\alpha z^n_v)=\emptyset$.
  Then
  \[g[p_n,H(m_3,-jz^n_v)]\cap g[p_n,G(5,\alpha z^n_v)]=\emptyset.\]
  Let
  \[H(m_3,-jz^n_\geq)=\cup\{H(m_3,jz^n_v)-g[p_n,H_d(v^\alpha_1,jz^n_v)]:
  v^\alpha_0\geq p_n\}.\]
  Then $H(m_3,-jz^n_\geq)\subset U_x$ and
  $g[p_n,H(m_3,-jz^n_\geq)]\cap g[p_n,H_d(v^\alpha_1,jz^n_v)]=\emptyset$
  for every $v^\alpha_0\geq p_n$.
\par

\vspace{0.3cm}

Let
  $H(m_3,-jz^n_\sigma)=H(m_3,-jz^n_<)\cup H(m_3,-jz^n_\geq).$
  Then, for every $v\in N$,
  \[g[p_n,\cup\mathscr{H}_d(v^\alpha_1,jz^n_v)]\cap g[p_n,H(m_3,-jz^n_\sigma)]
  =\emptyset.\]
  Then, for every $s\in U_x\cap H(m_2,jz_n)$ with $n<k$,
  we have $g[p,\{s\}]\cap H(m_2,jz_n)\subset U_x$ if $p>p_n$.
\end{proof}

\textit{The proof of Proposition 7.4 is continued.}
   Let $\mathscr{G}(m_1,k_j)$ be a Ln cover of $H'(m,1_j)$.
   Take $H^*(m,k_j)=H'(m,1_j)-\cup_{i< j}D(p_k,m^\alpha,k_i)$ and
  \[\mathscr{G}^*(m_1,k_j)=\{g(m_1,z)\in\mathscr{G}(m_1,k_j):
  g(m_1,z)\cap H^*(m,k_j)\neq\emptyset\}\]
  from Q2 in Section 6.
   Let  $s\in H^*(m,k_j)$. Then $s\notin \cup_{i< j}D(p_k,m^\alpha,k_i)$
   and $g[p,\{s\}]\cap[\cup_{i< j}D(p_k,m^\alpha,k_i)]=\emptyset$
   by Fact 7.4.1.
   Then  $H(p,s)\cap[\cup_{i\neq j}H'(m,1_i)]=\emptyset$
   by the above \textbf{1} and \textbf{2}.
   Then $H(p,s)\subset H(m_3,jz^n_v)$ for some $g(m_1,z)\in\mathscr{G}^*(m_1,k_j)$.
\par

Case 1, $n<k$.
  Let $s\in H(m_3,jz^n_v)\cap U_x$.
  Then $H(p,s)\subset H(m_3,jz^n_v)\cap U_x$ by Fact 7.4.2.
  Then $t\notin H(p,s)$ if $p>p_k$.
\par

Case 2, $n>k$.
  Then $H(m_3,jz^n_v)\cap g(\ell,x)\subset U_x$ since $U_x=g(\ell,x)-E(\alpha k,p_k)$.
  Then $H(p,s)\subset H(m_3,jz^n_v)\cap U_x$ if $s\in H(m_3,jz^n_v)\cap g(\ell,x)$,
  $p>\ell$ and $p>p_k>m_5$.
  Then $t\notin H(p,s)$.
\par

Case 3, $n=k$. Note $t\in g[p,U_x]-U_x$. Then there exists
  an $s\in H(m_3,jz^k_v)\cap U_x$ such that
  $t\in \Big[g[p,\{s\}]\cap H(m_3,jz^k_v)\Big]-U_x$. Note,
  by the definition of $E(\alpha k,p_k)$,
  \[\Big[g[p,\{s\}]\cap H(m_3,jz^k_v)\Big]-H(k\sigma^\alpha_1,q_k\sigma)\subset U_x.\]
  Then $t\in H(k\sigma^\alpha_1,q_k\sigma)$.
\par

To see it suppose $t\notin H(k\sigma^\alpha_1,q_k\sigma)$.
  Note $t\in g[p,\{s\}]\cap H(m_3,jz^k_v)$.
  Then we have $t\in \Big[g[p,\{s\}]\cap H(m_3,jz^k_v)\Big]-H(k\sigma^\alpha_1,q_k\sigma)\subset U_x$.
  It is a contradiction to $t\in \Big[g[p,\{s\}]\cap H(m_3,jz^k_v)\Big]-U_x$.
\par

Then $t\in H(k\sigma^\alpha_1,q_k\sigma)$ implies that there exists an $H(m_3,jz^k_v)$
  and an $H(v^\alpha_1,s')\in\mathscr{H}_{d}(v^\alpha_1,jz^k_v)$
  such that $t\in H(v^\alpha_1,s')$.
  Then $g(v^\alpha_1,s')\cap[\cup\mathscr{H}_{d}(v^\alpha_1,jz^k_v)]=H(v^\alpha_1,s')$.
  Then $t\in H'(m,l_j)\cap g(v^\alpha_1,s')=H(mv^\alpha_1,s')$
  and \[t\notin H(mv^\alpha_1,-s')=H(mv^\alpha_1,s')-H(v^\alpha_1,s')\subset U_x.\]
\par

Let $p\geq v^\alpha_1$. Note $g[p,U_x]\cap H(mv^\alpha_1,s')=
   H(mv^\alpha_1,-s')\subset U_x$. Then we have
  $g[p,U_x]\cap g(v^\alpha_1,s')\cap H'(m,l_j)=H(mv^\alpha_1,-s')$.
  Then
\begin{align}
 H(mv^\alpha_1,-s')&\subset g[p,H(mv^\alpha_1,-s')]
                    \cap g(v^\alpha_1,s')\cap H'(m,l_j)\nonumber\\
                  &\subset g[p, U_x]\cap g(v^\alpha_1,s')\cap H'(m,l_j)
                    =H(mv^\alpha_1,-s').
  \quad\nonumber
\end{align}
  Then $g[p,H(mv^\alpha_1,-s')]\cap H(v^\alpha_1,s')=\emptyset$.
   Let $p>\ell_3=v^\alpha_1$. Then we have
   \[ H(v^\alpha_1,s')\cap g[p,U_x]=\emptyset.\]
   Note $t\in H(v^\alpha_1,s')$ implies $H(v^\alpha_1,t)=H(v^\alpha_1,s')$.
   Then $H(v^\alpha_1,t)\cap g[p,U_x]=\emptyset$ if $p>\ell_3$.
\par

Let $q,l >\hbox{max}\{\ell_i: i\leq 3\}$.
  Then $H(l,t)\cap Cl_\tau g[q,U_x]=\emptyset$.
\end{proof}

Note $g[p,U(\alpha\ell k,p_kx)]$ is a closed set
  if $U(\alpha\ell k,p_kx)=g(\ell,x)-E(\alpha k,p_k)$ by Definition C.
\begin{prop}
Let $U_x=U(\alpha\ell k,p_kx)=g(\ell,x)-E(\alpha k,p_k)$.
 Then:
\par

1. $Cl_\tau g[l,g[p,U_x]]=g[p,U_x]$ if $l\geq p$.
\par

2. $\cap_lCl_\tau g[l,g[p,U_x]] =g[p,U_x]$.
\end{prop}

\begin{proof}
Note $g[p,U_x]$ is a c.o set and $g[l,g[p,U_x]] =g[p,U_x]$ if $l\geq p$.
  Then $t\notin g[p,U_x]$ implies $H(l,t)\cap g[p,U_x]=\emptyset$ for
  some $l$. Then we have $Cl_\tau g[l,g[p,U_x]]=g[p,U_x]$
  by the definition C.
\end{proof}

Recall  Proposition 3.3.
 Function $\mathscr{G}$ satisfies the following
 conditions:
\par
1 \  $\cap_\ell g(\ell,i,y)=\{ y\}$.
\par

4 \ $y\in g(\ell,i,x)$ implies $g(\ell,j,y)\subset g(\ell,i,x)$
  for some $j$.
\par

5 \ $g(\ell+1,j,x)\subset g(\ell,i,x)$.
\par

6 \ $j>k$ implies  $H(\ell,k)\cap(\cup\mathscr{G}(\ell,j))=\emptyset$.
\par

6$'$ \  $H(\ell,k,h)\cap g(\ell,j,l)\neq\emptyset$
  and $H(\ell,k,h)\neq H(\ell,j,l)$ imply $k>j$.

\par
7  \  Each $\mathscr{G}(\ell,i)$ is c.o.D family.
\par
8  \  If $g(\ell,i,l),g(\ell,j,e)\in\mathscr{G}_\ell$ with $j>i$,
    then $g(\ell,i,l)\cap g(\ell,j,e)=\emptyset$ or
    $g(\ell,j,e)\subset g(\ell,i,l)$.

\vspace{0.3cm}

\textbf{Theorem 1. }a.
$(Y,\tau)$ is a stratifiable space and
  $\mathscr{G}_Y$ is a $g$-function of $(Y,\tau)$.

\begin{proof}
To prove Theorem 1 we show the function $\mathscr{G}_Y$ such that:
\par
2, $x\in g(\ell,x_\ell)\in\mathscr{G}_Y$ implies $x_\ell\rightarrow x$ and
\par
3, if $H$ is closed and $x\notin H$, then $x \notin
   Cl_{\tau}(\cup \{ g(\ell,x')\in\mathscr{G}_Y: x'\in H \})$ for some $\ell$.
\par

\textit{proof of 2.} Let $x\in Y$.
  Let $g(\ell,x_\ell)=g(\ell,i_\ell,l_\ell)$, $x_\ell\in H(\ell,i_\ell,l_\ell)$,
   $x\in H(\ell,i'_\ell,l'_\ell)=\{q_{l_1}\}\times
   ...\times\{q_{l_\ell}\}\times J(\ell,i'_\ell,l'_\ell)$ and
   $S=\{x_\ell: \ell\in N\}$.
  Suppose that there exists an $U[\alpha^- lk,p_kx]$ with
  \[S-U[\alpha^- lk,p_k,x]=S_1=\{x_{n_i}: i\in N\}.\]
  Let $x_{\ell_i}\in H(\ell_i,i_{\ell_i},l_{\ell_i})\subset H(\ell,j_i,k_i)
  \subset g(\ell,j_i,k_i)$
  for $i\in N$, and let \[N_1=\{j_i:
  x_{\ell_i}\in H(\ell_i,i_{\ell_i},l_{\ell_i})\subset H(\ell,j_i,k_i)\hbox{ for }
  i\in N\}.\]
\par
Suppose that there exists a strictly increasing infinite subsequence of $N_1$,
  to say $j_1<j_2<...<j_i<...$. Then $x_{\ell_i}\in g(\ell,j_i,k_i)$ implies
  $g(\ell_i,x_{\ell_i})\subset g(\ell,j_i,k_i)$
  by 4 of Proposition 3.3.
  Note $x\in H(\ell,i'_\ell,l'_\ell)$. Then
  $j_i\rightarrow +\infty$
  implies $j_i>i'_\ell$ when $i>n_0$ for some $n_0$.
  Then, by 6 of Proposition 3.3,
  $ g(\ell,j_i,k_i)\cap H(\ell,i'_\ell,l'_\ell)=\emptyset$,
  a contradiction to $x\in g(\ell_i,x_{\ell_i})\subset g(\ell,j_i,k_i)$.
\par

So we may assume  $j=j_1=j_2=...=j_i=...$ and $g(\ell,j_i,k_i)=g(\ell,j,k_i)$.
 Let $\mathscr{G}(\ell, j)=\{g(\ell,j,k_i): i\in N\}$.
 Then $\mathscr{G}(\ell, j)$ is a c.o.D family in $(Y,\rho)$
 by 7 of Proposition 3.3.
 Note $x\in g(\ell_i,x_{\ell_i})\subset g(\ell,j,k_i)$ for each $i$.
 Then it is easy to see $k_1=k_2=...=k_i=...$.
 So $H(\ell,j_i,k_i)=H(\ell,j,k)$ for each $i$.
\par

Suppose $H(\ell,i'_\ell,l'_\ell)\neq H(\ell,j,k)
 =\{q_{j_1}\}\times...\times\{q_{j_\ell}\}\times J(\ell,j,k)$.
 Note the definition of $H(\ell,i'_\ell,l'_\ell)$.
 Then there exists an $e\leq \ell$ with
 $q_{j_e}\neq q_{l_e}$.
\par

 In fact, suppose $q_{j_e}=q_{l_e}$ for each $e\leq \ell$.
 Then $H(\ell,j,k)\neq  H(\ell,i'_\ell,l'_\ell)$
 implies $J(\ell,j,k)\cap J(\ell,i'_\ell,l'_\ell)=\emptyset$.
 Then $g(\ell,j,k)\cap g(\ell,i'_\ell,l'_\ell)=\emptyset$,
 a contradiction to $x\in g(\ell_i,x_{\ell_i})\subset g(\ell,j,k)$.
\par

 So $H(\ell,i'_\ell,l'_\ell)\neq H(\ell,j,k)$ implies
 $q_{j_\ell}\neq q_{l_\ell}$ for some $e\leq \ell$.
 Note $x_{\ell_i}\in H(\ell,j,k)$, and
 $\cap_i g(\ell_i,x_{\ell_i})=\{x\}$ since $x\in g(\ell_i,x_{\ell_i})$.
 Then $x_{\ell_i}$ converges to $x$ in $(Y,\rho)$,
 and $x\in H(\ell,j,k)$ since $H(\ell,j,k)$ is closed in $(Y,\rho)$.
 Note $x\in H(\ell,i'_\ell,l'_\ell)$.
 Then $H(\ell,i'_\ell,l'_\ell)=H(\ell,j,k)$ by 1 of Proposition 3.1,
 a contradiction to supposition $H(\ell,i'_\ell,l'_\ell)\neq H(\ell,j,k)$.
\par

 This implies $H(\ell,j,k)=H(\ell,i'_\ell,l'_\ell)$.
 So $S_1\subset H(\ell,j,k)=H(\ell,i'_\ell,l'_\ell)$, a contradiction to
 $S_1\cap U[\alpha^- lk,p_kx]=\emptyset$ and
 $H(\ell,i'_\ell,l'_\ell)\subset U[\alpha^- lk,p_kx]$.
 So $S-U[\alpha^-,l,p,h,x]$ is finite for each $U[\alpha^- lk,p_kx]$.
 So $x_\ell\rightarrow x$.
\par

\textit{proof of 3.} Let $H$ be closed in $(Y,\tau)$.
  Then there exists an $O\in\tau$ and
  a family
  $\mathscr{O}=\{U_{\lambda}: \lambda\in\Lambda\}\subset\mathscr{U}$
  such that $O=\cup\mathscr{O}$ and
  \[H=Y-O=Y-\cup\mathscr{O}=
  Y-\cup\{U_\lambda: \lambda\in\Lambda\}.\]
  Note $U_\lambda\in \mathscr{U}$.
  Then $U_\lambda=\cap\{U_{i\lambda}\in\mathscr{U}(x_\lambda):i\leq n(\lambda)\}$.
  Let $U_{i\lambda}=Y-H_{i\lambda}$.
  Then $U_\lambda=Y-\cup_{i\leq n(\lambda)}H_{i\lambda}$. Then
  \[H=Y-\cup\{U_\lambda: \lambda\in\Lambda\}=
  Y-\cup\{Y-\cup_{i\leq n(\lambda)}H_{i\lambda}: \lambda\in\Lambda\}
  =\cap\{\cup_{i\leq n(\lambda)}H_{i\lambda}: \lambda\in\Lambda\}.\]
  Let $s\notin H$. Then there exists an $\cup_{i\leq n(\lambda)}H_{i\lambda}$
  such that $s\notin \cup_{i\leq n(\lambda)}H_{i\lambda}\supset H$.
\par

Note $s\notin \cup_{i\leq n(\lambda)}H_{i\lambda}$. Then
  $s\notin H_{i\lambda}$ for each $i\leq n(\lambda)$.
  Then, by Proposition 7.3-7.5, there exists an $\ell_i$
  such that $H(l_i,s)\cap Cl_\tau g[q_i,H_{i\lambda}]=\emptyset$ if $l_i,q_i>\ell_i$.
  Let $\ell>\hbox{max}\{q_i,l_i:i\leq n(\lambda)\}$.
  Then $H(l,s)\cap [\cup_{i\leq n(\lambda)}Cl_\tau g[q,H_{i\lambda}]]=\emptyset$
  if $l,q>\ell$.
  Note \[\cup_{i\leq n(\lambda)}Cl_\tau g[q,H_{i\lambda}]
  =Cl_\tau g[q,\cup_{i\leq n(\lambda)}H_{i\lambda}].\]
  Then $H(l,s)\cap Cl_\tau g[q,\cup_{i\leq n(\lambda)}H_{i\lambda}]=\emptyset$
  if $l,q>\ell$.
  Then $H(l,s)\cap Cl_\tau g[q,H]=\emptyset$
  if $l,q>\ell$.
\end{proof}

\section{ Properties  of neighborhoods  of
  stratifiable space $(Y,\tau)$ }

\begin{prop}
Let $y\in H_0$, $y'\in U[\alpha^- b k',p'_ky]=
 \cap_{i\leq n}U(\alpha_i b_i k_i,p_{k_i}y)\in\mathscr{U}_y$,
 $U[\alpha^- b k',p'_ky]=g(b,y)-E[\alpha^- k',p'_k]$ and
 $U(\delta bk,p_ky')=g(b,y')-E(\delta k,p_k)$
 with $\delta\neq \alpha_i$ for $i\leq n$.
 Then there exists an $H(l,t)\subset E(\delta k,p_k)$
 and $H(l,t)\subset U[\alpha^- b k',p'_ky]$.
\end{prop}

\begin{proof} Let $\alpha^-=\{\alpha_i:i\leq n\}$.
 Note $E[\alpha^- k',p'_k]\cap H_0=\emptyset$ for $\alpha\in\mathscr{A}$.
\par

\textbf{A.} Suppose $y\neq y'$. Then there exists an $l$
  with $g(l,y')\cap g(l,y)=\emptyset$.
  Then we may assume  $y'=y\in H_0$.
\par

\textbf{B.}
  Take $g(l,y)$.
  Then there exists $c$ such that
  $H(m,\ell,h)\cap g_c(l,y)\neq\emptyset$ and $\ell>c$
  imply $g(m,\ell,h)\subset g_c(l,y)$ by 1 of Corollary 5.3.
  Let $1_j$ be the least number such that
  $1_j>c$ and $H'(m,1_j)\cap g_c(l,y)\neq\emptyset$.
\par

Let $\mathscr{G}(m_1,1_j)$ be an Ln cover of $H'(m,1_j)$ and
 \[\mathscr{G}^*(m_1,y1_j)=\{g(m_1,x)\in\mathscr{G}(m_1,1_j):
  g(m_1,x)\subset g_c(l,y)\}.\]
  Take $G_d(n\sigma^\alpha_1,jz,p)
  =\cup\mathscr{G}_d(n\sigma^\alpha_1,jz,p\sigma)$ from Fact 6.11.
  Then, by Fact 6.11,  we have the following fact:
\par

\textbf{Fact B**.}
1. $\emptyset\neq\mathscr{G}^*(m_1,y1_j)\subset\mathscr{G}^*(m_1,1_j)$.
\par

2. $\mathscr{H}_d(n\sigma^\alpha_1,1_j,p\sigma)|H(m_1,z)
  =\mathscr{H}_d(n\sigma^\alpha_1,jz,p\sigma)$
  for every $g(m_1,z)\in\mathscr{G}^*(m_1,y1_j)$.
\par

3. $H(m_1,z)-G_d(n\sigma^\alpha_1,jz,p)
  \subset  g_c(l,y)$ and
  $G_d(n\sigma^\alpha_1,jz,p)\subset g_c(l,y)$
  for every $g(m_1,z)\in\mathscr{G}^*(m_1,y1_j)$.  $\Box$
\par

Let  $k>n^*=\hbox{max}\{k_i: i\leq n\}$. If $H(l,t)\subset H(k\sigma^\delta_1,q_k\sigma)
  \cap H(m_2,jz_k)\subset H(m_1,z)$ for $g(m_1,z)\in\mathscr{G}^*(m_1,y1_j)$,
  then $H(l,t)\cap E[\alpha^- k',p'_k]=\emptyset$
  since $E[\alpha^- k',p'_k]\cap H(m_1,z)\subset\cup_{i\leq n}H(m_2,jz_i)$
  by 2 of Fact $B^{**}$ for $\alpha^-$.
  So, without loss of generality, we may assume $k\leq\hbox{min}\{k_i: i\leq n\}$.

\vspace{0.3cm}

\textbf{C.} Take $\alpha^-=\{\alpha_i:i\leq n\}$.
\par

1.  Let $e$ be the least number
  such that $i>e$ implies that $\{\alpha_h(i): h\leq n\}$ is different each other.
  We may assume
  \[\alpha_1(i)<...<\alpha_b(i)=\beta(i)<\delta(i)<
  \gamma(i)=\alpha_{b+1}(i)<...<\alpha_n(i).\]
\par

2. Take $\mathscr{H}_d(m_5,\theta t^i)$,
  $\mathscr{H}_{d1}(m_5,\theta t^i)$ and
  $\mathscr{H}_{d2}(m_5,\theta t^i)$ from $\Re2$.
   Let
  \[\mathscr{H}_{-z}(m_5,\theta t^i)=
  \{H(m_5,t^i_{jl})\in\mathscr{H}(m_5,\theta t^i_j):
   l\leq\theta(i)\hbox{ for } j\leq \theta(i)\}\quad\hbox{ and}\]
   \[\mathscr{H}_{-y}(m_5,\theta t^i)=
  \{H(m_5,t^i_{jl})\in\mathscr{H}(m_5,\theta t^i_j):
   l>\theta(i)\hbox{ for } j>\theta(i)\}.\]
   Note $\gamma(i)=\alpha_{b+1}(i)<...<\alpha_n(i)$.
   Then
   $\mathscr{H}_{-z}(m_5,\gamma t^i)\subset...
   \subset\mathscr{H}_{-z}(m_5,\alpha_n t^i)$
   and \[\mathscr{H}_{-z}(m_5,\gamma t^i)=
   \cap\{\mathscr{H}_{-z}(m_5,\alpha_j t^i): b+1\leq j\leq n\}.\]
   Note $\alpha_1(i)<...<\alpha_b(i)=\beta(i)$.
   Then
   $\mathscr{H}_{-y}(m_5,\alpha_1 t^i)\supset...
   \supset\mathscr{H}_{-y}(m_5,\beta t^i)$
   and \[\mathscr{H}_{-y}(m_5,\beta t^i)=
   \cap\{\mathscr{H}_{-y}(m_5,\alpha_j t^i): 1\leq j\leq b\}.\]
\par

   Take $\mathscr{H}_d(m_5,\delta t^i)$ from $\Re2$. Let
  \[\mathscr{H}(m_5\delta,-\beta\gamma,t^i)=
  \mathscr{H}_d(m_5,\delta t^i)\cap[\mathscr{H}_{-y}(m_5,\beta t^i)
  \cap\mathscr{H}_{-z}(m_5,\gamma t^i)].\]
  Then
  \begin{align}
 \mathscr{H}(m_5\delta,-\beta\gamma,t^i)
 =\{H(m_5,t^i_{jl}):&\
 \beta(i)<l\leq\delta(i)\hbox{ if }
       \delta(i)<j\leq\gamma(i),\hbox{ or } \nonumber\\
  &\delta(i)<l\leq\gamma(i)\hbox{ if } \beta(i)<j\leq\delta(i)\ \}.
  \quad\nonumber
\end{align}
\par

\vspace{0.3cm}

3.  We translate $\mathscr{H}(m_5\delta,-\beta\gamma,t^i)$ into
  $\mathscr{H}(m_5\delta,-\beta\gamma,z^k_v)$.
\par

To do it take $\mathscr{H}(m_5,x^v)$ and $\mathscr{G}(m_5,x^v)$ for $v=i$
  from $\Re2$. Let $x^v=z^k_v$ and take $H(m_3,jz^k_v)$ from Claim 6.6
  for $v>e$. Here $e$ is in the above 1 of C. Let
\begin{align}
 \mathscr{H}(m_5\delta,-\beta\gamma,z^k_v)
 =\{H(m_5,t^v_{jl})\in&\mathscr{H}(m_5,x^v):\
 \beta(v)<l\leq\delta(v)\hbox{ if }
       \delta(v)<j\leq\gamma(v), \nonumber\\
  &\hbox{ or }\delta(v)<l\leq\gamma(v)\hbox{ if } \beta(v)<j\leq\delta(v)\ \}.
  \quad\nonumber
\end{align}
  Note
  $H(m_5,t^v_{jl})=H(m,j'z)\cap g(m_5,t^v_{jl})$
  for $H(m_5,t^v_{jl})\in\mathscr{H}(m_5\delta,-\beta\gamma,z^k_v)$.
  Let \[\mathscr{G}(m_5\delta,-\beta\gamma,z^k_v)=\{g(m_5,t^v_{jl})\in\mathscr{G}(m_5,x^v):
  H(m_5,t^v_{jl})\in\mathscr{H}(m_5\delta,-\beta\gamma,z^k_v)\}.\]
\par

\vspace{0.3cm}

4. Take  $H(m_3,jz^k_v)$ and $\mathscr{G}^*_d(v^\alpha_1,jz^k_v)$ from B1 of
  Construction 6.1.  Let
  \[\mathscr{H}(m_5\delta,-\beta\gamma, q_kz)=
  \cup\{\mathscr{H}(m_5\delta,-\beta\gamma,z^k_v): z^k_v\in H(m_1,jz),
  v>e \hbox{ and }v>q_k\}\quad \hbox{ and }\]
  \[H(m_5\delta,-\beta\gamma, q_kz)
  =\cup\mathscr{H}(m_5\delta,-\beta\gamma, q_kz).\]
  Then, for $ j\leq n$,
  $H(m_5\delta,-\beta\gamma, q_kz)\cap E(\alpha_j k_j,p_{k_j})=\emptyset.$
  Then, by the definition of $\mathscr{H}(m_5\delta,-\beta\gamma,z^k_v),$
  we have
   \[H(m_5\delta,-\beta\gamma, q_kz)\cap[\cup\mathscr{H}_d(v^\delta_1,jz^k_v)]\subset H(k\sigma^\delta_1,q_k\sigma)
   \subset E(\delta k,p_k).\]

\begin{cl}
Let
   \[\mathscr{G}(m_5\delta,-\beta\gamma,q_k)=
   \cup\{\mathscr{G}(m_5\delta,-\beta\gamma, q_kz): g(m_1,z)\in\mathscr{G}^*(m_1,1_j)
   \hbox{ and } j\in N\}\hbox{ and}\]
   \[\mathscr{H}(m_5\delta,-\beta\gamma,q_k) =\cup\{\mathscr{H}(m_5\delta,-\beta\gamma, q_kz):
   g(m_1,z)\in\mathscr{G}^*(m_1,1_j) \hbox{ and } j\in N\}.\]
 Then:  1. $\mathscr{G}(m_5\delta,-\beta\gamma,q_kz)$ is a Ln family,
\par

2. For every $H(m_5,t^v_{jl})\in \mathscr{H}(m_5\delta,-\beta\gamma,q_k)$,
  there uniquely
  exists a $g(m_5,t^v_{jl})\in\mathscr{G}(m_5\delta,-\beta\gamma,q_k)$
  with $H(m_5,t^v_{jl})=g(m_5,t^v_{jl})\cap H(m_1,j'z)$,
\par

3. $H(m_5\delta,-\beta\gamma, q_k)\cap E(\alpha_j k_j,p_{k_j})=\emptyset$
  for each $\alpha_j\in\alpha^-$.
  Here $H(m_5\delta,-\beta\gamma,q_k)=\cup\mathscr{H}(m_5\delta,-\beta\gamma,q_k)$.
\par

4.  $H(m_5\delta,-\beta\gamma, q_k)\cap[\cup\mathscr{H}_d(v^\delta_1,jz^k_v)]
  \subset H(k\sigma^\delta_1,q_k\sigma)\subset E(\delta k,p_k)$.
\end{cl}

\textit{The proof is continued.}
  Let $H(v^\delta_1,t)\in \mathscr{H}_d(v^\delta_1,jz^k_v)|H(m_5\delta,-\beta\gamma, q_k)$.
  Then $H(v^\delta_1,t)\subset E(\delta k,p_k)$,
  $H(v^\delta_1,t)\subset g_c(l,y)$ and $H(v^\delta_1,t)\cap E(\alpha_j k_j,p_{k_j})=\emptyset$.
  Then $H(v^\delta_1,t)\subset g_c(l,y)-E(\alpha_j k_j,p_{k_j})$
  for $j\leq n$.
\end{proof}

\vspace{0.3cm}

Let $\mathscr{A}$ be a \textit{mad} family on $N$.
  Take an $\alpha\in \mathscr{A}$.
  Let $k\geq 1$. Then there exists a family
  $\mathscr{E}(\alpha,k)=\{E(\alpha k,p): p>m_5\}$ by Definition B.
  Let $\mathscr{E}(\alpha)=\cup_k\mathscr{E}(\alpha,k)$.
\par

Let $\mathscr{V}_y=\{V_\lambda: \lambda\in \Lambda\}$ be
  an arbitrary  neighborhoods base of $y\in H_0$ in $(Y,\tau)$.
  Then $V_\lambda$ is open in $(Y,\tau)$ for every $V_\lambda\in\mathscr{V}_y$
  by the definition of  neighborhoods base
  before Proposition 7.2.
  Let $b>m_5$, $y\in H_0$, $p_1>m_5$, $c$ on $(g(b,y),m)$ and
  \[\mathscr{U}(y,2,c)=\{U(\alpha b2,p_2y)=g_c(b,y)-E(\alpha2,p_2)\in\mathscr{U}(y):
  \alpha\in \mathscr{A}\}.\]
  Take $\mathscr{U}_y$ from Proposition 7.2.
  Then, for arbitrary $U(\alpha b2,p_2y)\in\mathscr{U}(y,2,c)$, there exists a
  $V_\alpha\in\mathscr{V}_y$ such that $V_\alpha\subset U(\alpha b2,p_2y)$ by the definition
  of $\mathscr{V}_y$.
  Then there exists an $U[\alpha^- b k',p'_ky]\in\mathscr{U}_y$
  such that $U[\alpha^- b k',p'_ky]\subset V_\alpha$ by the definition of
  $\mathscr{U}_y$.
  Let $U[\alpha^- b k',p'_ky]=
  \cap_{i\leq n}U(\alpha_i b_i k_i,p_{k_i}y)\in\mathscr{U}_y$, $\alpha^-=\{\alpha_i: i\leq n\}$ and
   $E[\alpha^- k',p'_k]=\cup_{i\leq n}E(\alpha_i k_i,p_{k_i})$.
\par

\par
 Take $\mathscr{H}_d(k\sigma^\alpha_1,q_k\sigma)$ from B2, and
  take $H(m_5\alpha,-\beta\gamma, q_kz)$ from Claim 8.2 for
   $g(m_1,z)\in\mathscr{G}^*(m_1,1_j)$.
  Let
  \[\mathscr{H}(m_5\sigma^\alpha_1,-\beta\gamma,q_kz)
  =\mathscr{H}_d(k\sigma^\alpha_1,q_k\sigma)|H(m_5\alpha,-\beta\gamma, q_kz)
  \quad\hbox{ and}\]
  \[H(m_5\sigma^\alpha_1,-\beta\gamma,q_kz)=
  \cup\mathscr{H}(m_5\sigma^\alpha_1,-\beta\gamma,q_kz).\]
\par

\textit{Call $V_\alpha$} \textbf{fine} \textit{if  there exists an $n\leq k$ and a full
  $g(m_1,z)\in \mathscr{G}^*(m_1,1_j)$
  such that
  $H(m_5\alpha,-\beta\gamma,q_nz)-
  H(m_5\sigma^\alpha_1,-\beta\gamma,q_nz)\subset U[\alpha^- b n',p'_ns]\subset
  V_\alpha \subset U(\alpha b 2,p_2y)$
  and $H(m_5\sigma^\alpha_1,-\beta\gamma,q_nz)\cap Cl_\tau V_\alpha=\emptyset$
  for $V_\alpha\in\mathscr{V}_y$.}

\begin{prop}
Let $y\in H_0$, $U[\alpha^- b k',p'_ky]=
 \cap_{i\leq n}U(\alpha_i b_i k_i,p_{k_i}y)\in\mathscr{U}_y$, $V_\alpha\in\mathscr{V}_y$ and
 $U(\alpha b 2,p_2y)\in\mathscr{U}(y,2,c)$
 such that \[U[\alpha^- b k',p'_ky]\subset V_\alpha\subset U(\alpha b2,p_2y)\subset g_c(b,y).\]
 Then $V_\alpha$ is fine.
\end{prop}

\begin{proof}
We prove the proposition by induction on $k$ for $U[\alpha^- b k',p'_ky]$.
  To do it note $U[\alpha^- b k',p'_ky]=g_c(b,y)-E[\alpha^- k',p'_k]\subset
  V_\alpha\subset U(\alpha b2,p_2y)$
  since $y\in H_0$ and $H_0\cap E(\beta h,p_h)=\emptyset$ for every $\beta\in\mathscr{A}$
  and every $h\geq 2$.
  Let $\alpha^-=\{\alpha_i: i\leq n\}$
  and $U[\alpha^- b k',p'_ky]=\cap_{i\leq n}U(\alpha_i b_i k_i,p_{k_i}y)$.
  Then $\alpha\in\alpha^-$ by Proposition 8.1.
  Then  $\alpha=\alpha_a$ and $E(\alpha k,p_{k_a})
  \subset E[\alpha^- k',p'_k]=\cup_{i\leq n}E(\alpha_i k_i,p_{k_i})$.
  Then we may assume $E(\alpha k,p_k)\subset E[\alpha^- k',p'_k]$.
\par

In fact,  if $p_{k_a}\leq p_k$, then
  $E(\alpha k,p_k)\subset E(\alpha k,p_{k_a})\subset E[\alpha^- k',p'_k]$.
\par

Let $p_k<p_{k_a}$. Then $E(\alpha k,p_{k_a})\subset E(\alpha k,p_k)$ and
  $U(\alpha b k,p_ky)\subset U(\alpha bk,p_{k_a})$.
  Then $U[\alpha^- bk',p'_ky]\cap U(\alpha bk,p_ky)
  \subset U[\alpha^- bk',p'_ky]\subset V_\alpha\subset U(\alpha b2,p_2y).$
  we denote $E[\alpha^- k',p'_k]\cup E(\alpha k,p_k)$ by $E[\alpha^- k',p'_k]$,
  and $U[\alpha^- bk',p'_ky]\cap U(\alpha bk,p_ky)$ by $U[\alpha^- bk',p'_ky]$
  still. Then $E(\alpha k,p_k)\subset E[\alpha^- k',p'_k]$.
  Then we have the following fact.
\par

\vspace{0.3cm}

\textbf{Fact 1.} There exists an $U[\alpha^- b k',p'_ky]\in\mathscr{U}_y$
  such that $E(\alpha k,p_k)\subset E[\alpha^- k',p'_k]$, $\alpha\in\alpha^-$
  and $U[\alpha^- b k',p'_ky]=g_c(b,y)-E[\alpha^- k',p'_k]\subset V_\alpha\subset
  U(\alpha b 2,p_2y)\in\mathscr{U}(y,2,c)$.
\par

\vspace{0.3cm}

A. $k=2$  and $U[\alpha^- b2',p'_2y]=g_c(b,y)-E[\alpha^- 2',p'_2]
  \subset V_\alpha\subset U(\alpha b2,p_2y)$.
\par

  Then $E(\alpha2,p_2)\subset E[\alpha^-2',p'_2]$ by Fact 1.
  Note $V_\alpha\subset U(\alpha b2,p_2y)$.
   Then we have $E(\alpha2,p_2)\cap U(\alpha b2,p_2y)=\emptyset$.
  Then $E(\alpha2,p_2)\cap Cl_\tau V_\alpha=\emptyset$
  since $E(\alpha2,p_2)$ is a c.o set by 1 of Definition B.
  Then, by 4 of Claim 8.2, $H(m_5\sigma^\alpha_1,-\beta\gamma,q_2z)\cap Cl_\tau V_\alpha=\emptyset$
  since $H(m_5\sigma^\alpha_1,-\beta\gamma,q_2z)\subset E(\alpha2,p_2)$.
  Note $H(m_1,z)-E(\alpha2,p_2)\subset U(\alpha b2,p_2y)$
  by the definition of $U(\alpha b2,p_2y)$
  for every $H'(m,1_j)\in\mathscr{H}'(m,y_0)$ and
  every $g(m_1,z)\in\mathscr{G}^*(m_1,1_j)$ with $g(m_1,z)\subset g_c(b,y)$.
  Then, by 3 and 4 of Claim 8.2, we have
  \[H=H(m_5\alpha,-\beta\gamma, q_2z)-H(m_5\sigma^\alpha_1,-\beta\gamma,q_2z)
  =H(m_5\alpha,-\beta\gamma, q_2z)-E(\alpha2,p_2) \quad\hbox{ and}\]
  \[H=H(m_5\alpha,-\beta\gamma, q_2z)-E[\alpha^-2',p'_2]
  \subset U[\alpha^- b2',p'_2y]\subset V_\alpha.\]
\par

Then $V_\alpha$ is fine if $U[\alpha^- b2',p'_2y]\subset V_\alpha\subset U(\alpha b2,p_2y)$
  with $\alpha\in\alpha^-$.
\par

\vspace{0.3cm}
B. $k=3$  and $U[\alpha^- b3',p'_3y]=g_c(b,y)-E[\alpha^-3',p'_3]
 \subset V_\alpha\subset U(\alpha b2,p_2y)$.
\par

Then $\alpha\in\alpha^-$ and
  $E(\alpha3,p_3)\subset E[\alpha^-3',p'_3]$ by Fact 1.
\par

Case 1, there exists an $H'(m,1_j)\in\mathscr{H}'(m,y_0)$ and a full
  $g(m_1,z)\in\mathscr{G}^*(m_1,1_j)$  such that
  $H(m_5\sigma^\alpha_1,-\beta\gamma,q_3z)\cap Cl_\tau V_\alpha=\emptyset$.
  Then we can prove that $V_\alpha$ is fine.
\par

In fact, by 3 and 4 of Claim 8.2, we have
  \[H=H(m_5\alpha,-\beta\gamma, q_3z)-H(m_5\sigma^\alpha_1,-\beta\gamma,q_3z)
  =H(m_5\alpha,-\beta\gamma, q_3z)-E(\alpha3,p_3)\hbox{ and}\]
  \[H=H(m_5\alpha,-\beta\gamma, q_3z)-E[\alpha^- 3',p'_3]
  \subset U[\alpha^- b3',p'_3y]\subset V_\alpha.\]
\par

Case 2, for every $H'(m,1_j)\in\mathscr{H}'(m,y_0)$ and every full
  $g(m_1,z)\in\mathscr{G}^*(m_1,1_j)$,
  we have
  \[H(m_5\sigma^\alpha_1,-\beta\gamma,q_3z)\cap Cl_\tau V_\alpha\neq\emptyset.\]
\par

  Let $\delta\in \alpha^-$ with $\delta\neq\alpha$. Then
  $H(m_5\sigma^\alpha_1,-\beta\gamma,q_3z)\cap E(\delta3,p_3)=\emptyset$ by 3 of Claim 8.2
  for every $H'(m,1_j)\in\mathscr{H}'(m,y_0)$ and every $g(m_1,z)\in\mathscr{G}^*(m_1,1_j)$.
  Then $E(\alpha3,p_3)$ is the unique open set with $H(m_5\sigma^\alpha_1,-\beta\gamma,q_3z)\subset E(\alpha3,p_3)$
  in $\alpha^-$.
  Take $D(\alpha3,p_3)=g[p_2,D(1\sigma^\alpha_1,G_1)]\cup
  g[p_3,H(2\sigma^\alpha_1,q_2\sigma)]$.
  Then, by the definition of $E(\alpha2,p_2)$,
  $g[p_2,D(1\sigma^\alpha_1,G_1)]\subset E(\alpha2,p_2)$.
  Note $g[p_3,H(2\sigma^\alpha_1,q_2\sigma)]-E(\alpha2,p_2)$ is open and
  dense in $E(\alpha3,p_3)-E(\alpha2,p_2)$ by the definition of $E(\alpha3,p_3)$.
  Then
  \[H(m_5\sigma^\alpha_1,-\beta\gamma,q_3z)\cap Cl_\tau \big[V_\alpha\cap
  g[p_3,H(2\sigma^\alpha_1,q_2\sigma)]\big]\neq\emptyset.\]
  Let $V_\alpha(p_3)=V_\alpha\cap
  g[p_3,H(2\sigma^\alpha_1,q_2\sigma)]$
  and $x\in V_\alpha(p_3)$.
  Then there exists an open neighborhood
  $U[\beta^-\ell k',p'_kx]$
  such that $U[\beta^-\ell k',p'_kx]\subset V_\alpha(p_3)$
  because both $V_\alpha$ and $g[p_3,H(2\sigma^\alpha_1,q_2\sigma)]$
  are open sets in $(Y,\tau)$.
  Then we have the following fact.
\par
\vspace{0.3cm}

\textbf{Fact B.1.}
  $U[\beta^-\ell k',p'_kx]\subset V_\alpha(p_3)\subset V_\alpha$ for some $U[\beta^-\ell k',p'_kx]\in \mathscr{U}_x$.
\vspace{0.3cm}

Note $p_3>m_5$ and $x\in V_\alpha(p_3)=V_\alpha\cap
  g[p_3,H(2\sigma^\alpha_1,q_2\sigma)]\subset g[p_3,H(2\sigma^\alpha_1,q_2\sigma)]$.
  Then $x\in E(\alpha2,p_2)$ or
  $x\in g[p_3,H(2\sigma^\alpha_1,q_2\sigma)]-E(\alpha2,p_2)$.
\par

Note $x\in V_\alpha(p_3)\subset V_\alpha\subset U(\alpha b2,p_2y)=g_c(b,y)-E(\alpha2,p_2) $.
  Then  $x\in g[p_3,H(2\sigma^\alpha_1,q_2\sigma)]-E(\alpha2,p_2)$.
  Then $x\in g(\ell,x)-E(\alpha2,p_2)$ for some $\ell\in N$ with $g(\ell,x)\subset g_c(b,y)$.
  Note 1 of Proposition 7.2. We may let
  \[E_\cap[\theta h,p_h]=g(\ell,x)\cap\bigcap_{i\leq n_1} E(\theta_i h_i,p^i_h),
  \quad E[\beta^-2',p'_2]=\cup_{i\leq n_2} E(\beta_i2_i,p^i_2)\quad \hbox{ and}\]
  \[U[\beta^-\ell k',p'_kx]=
  U[\beta^-\ell2',p'_2x]=E_\cap[\theta h,p_h]
  \cap g(\ell,x)\cap\big[g_c(b,y)-E[\beta^-2',p'_2]\big].\]
Note $U[\beta^-\ell2',p'_2x]\subset U(\alpha b2,p_2y)=g_c(b,y)-E(\alpha2,p_2)$.
  Then we have
  \[U[\beta^-\ell2',p'_2x]=U[\beta^-\ell2',p'_2x]\cap U(\alpha b2,p_2y)\quad\hbox{ and}\]
  \[U[\beta^-\ell2',p'_2x]=E_\cap[\theta h,p_h]\cap g(\ell,x)\cap\big[g_c(b,y)-E[\beta^-2',p'_2]\big]\cap[g_c(b,y)-E(\alpha2,p_2)].\]
  Then $U[\beta^-\ell2',p'_2x]=E_\cap[\theta h,p_h]\cap g(\ell,x)\cap\Big[g_c(b,y)-\big[E[\beta^-2',p'_2]\cup E(\alpha2,p_2)\big]\Big]$.
  Let  $\beta^-=\{\beta_i: i\leq n_2\}$ with $\alpha\in \beta^-$.
  Denote $E[\beta^-2',p'_2]\cup E(\alpha2,p_2)$ by $E[\beta^-2',p'_2]$
  still. Then we have proved the following fact.
\par
\vspace{0.3cm}

\textbf{Fact B.2.} $\alpha\in \beta^-$, $U[\beta^-\ell2',p'_2x]
  =E_\cap[\theta h,p_h]\cap g(\ell,x)\cap\big[g_c(b,y)-E[\beta^-2',p'_2]\big]$ and
  $E(\alpha2,p_2)\subset E[\beta^-2',p'_2]$.
\vspace{0.3cm}

In order to prove that $V_\alpha$ is fine in case 2 if $U[\beta^-\ell2',p'_2x]\subset V_\alpha
  \subset U(\alpha b2,p_2y)$ and $\alpha\in\beta^-$, we give the following
  decompositions.
\par
\vspace{0.3cm}

\textbf{Decomposition B.3.}  Decompose $E_\cap[\theta h,p_h]$.
\par
\vspace{0.3cm}

  To do it note $x\in U[\beta^-\ell2',p'_2x]\subset V_\alpha(p_3)=V_\alpha\cap
  g[p_3,H(2\sigma^\alpha_1,q_2\sigma)]$.
\par

I.\ \  Note the definition of $H(k\sigma^\alpha_1,q_k\sigma)$.
  Then we have $x\notin H(3\sigma^\alpha_1,q_3\sigma)$ since
  $g[p_3,H(2\sigma^\alpha_1,q_2\sigma)]\cap H(3\sigma^\alpha_1,q_3\sigma)
  =\emptyset$.
  Then $x\notin H(n\sigma^\alpha_1,q_n\sigma)$ since
  $H(n\sigma^\alpha_1,q_n\sigma)\cap D(n\sigma^\alpha_1,q_n\sigma)=\emptyset$
  and $g[p_3,H(2\sigma^\alpha_1,q_2\sigma)]\subset D(n\sigma^\alpha_1,q_n\sigma)$.
  Note $x\in \cap_{i\leq n_1} E(\theta_ih_i,p^i_h)$.
  Then $x\in H(2\sigma^{\theta_i}_1,q_2\sigma)$
  for $n'_1<i\leq  n_1$, and $x\in D(\theta_i h_i,p^i_h)$ for $i\leq n'_1\leq n_1$.
  Note $x\in D(\theta_i h_i,p^i_h)$ implies that there exists a
  $g(\ell_1,x)\subset D(\theta_i h_i,p^i_h)\subset E(\theta_i h_i,p^i_h)$
  if $\ell_1=\hbox{max}\{v^{\theta_i}_1:i\leq n'_1\quad \hbox{and}\quad x\in H(m_3,j'z^2_v)\}$.
  Here $H(m_3,j'z^2_v)$ is defined in B2 of Construction 6.1.
  Then we may assume
  \[x\in E'_\cap[\theta h,p_h]=
  g(\ell_1,x)\cap\bigcap_{n'_1<i\leq n_1} E(\theta_i2_i,p^i_2).\]
  Then $x\in \bigcap_{n'_1<i\leq n_1} H(2\sigma^{\theta_i}_1,q^i_2\sigma)$,
  $Cl_\tau D(\theta_i2,p^i_2)=D(\theta_i2,p^i_2)\cup H(2\sigma^{\theta_i}_1,q^i_2\sigma)$
  for $n'_1<i\leq n_1$, and $x\in U[\beta^- b2',p'_2y]=g_c(b,y)-E[\beta^-2',p'_2]$.
\par

II.\ \ Assume $\theta=\theta_{n'}$  with $\theta_1(v)\leq ...\leq \theta_{n'}(v)$
  and $n'=n_1-n'_1$. Note
  $H(v^{\theta_i}_1,x)\in\mathscr{H}_d(2\sigma^{\theta_i}_1,q^i_2\sigma)$
  since $x\in H(2\sigma^{\theta_i}_1,q^i_2\sigma)$.
  Then $H(v^\theta_1,x)\subset\bigcap_{n'_1<i\leq n_1} H(2\sigma^{\theta_i}_1,q^i_2\sigma)$.
  Note $x\in U[\beta^- b2',p'_2y]=g_c(b,y)-E[\beta^-2',p'_2]$. Then
  $H(l,x)\subset U[\beta^- b2',p'_2y]$ for some $l\in N$ by 2 of Proposition 7.2.
  Then we can assume  $H(l,x)\subset H(v^\theta_1,x)$.
  Let $D[g_x,\theta,p'_2]=g(l,x)\cap \bigcap_{n'_1<i\leq n_1} D(\theta_i2,p^i_2)$.
  Then $H(l,x)\subset Cl_\tau D[g_x,\theta,p'_2]$
  since $D(\theta_i2,p^i_2)$ is open and dense in $(E(\theta_i2,p^i_2),\tau)$ for
  $n'_1<i\leq n_1$.
\par
  Then there exists a $g(k_t,t)\subset D[g_x,\theta,p'_2]$ such that
  $g(k_t,t)-E[\beta^-2',p'_2]\neq\emptyset$.
\par

In fact, suppose $g(k_t,t)\subset E[\beta^-2',p'_2]$ for every
  $g(k_t,t)\subset D[g_x,\theta,p'_2]$.
  Then $D[g_x,\theta,p'_2]\subset E[\beta^-2',p'_2]$
  and $D[g_x,\theta,p'_2]$ is open and dense in $(g(l,x)\cap E'_\cap[\theta h,p_h],\tau)$.
  Then \[H(l,x)\subset Cl_\tau D[g_x,\theta,p'_2]\subset Cl_\tau E[\beta^-2',p'_2]=E[\beta^-2',p'_2].\]
  It is a contradiction to $H(l,x)\subset U[\beta^- b2',p'_2y]
  =g_c(b,y)-E[\beta^-2',p'_2]$.
\par

Then $g(k_t,t)-E[\beta^-2',p'_2]\neq\emptyset$ for some $g(k_t,t)\subset D[g_x,\theta,p'_2]$.
  Then there exists a $g(\ell_s,s)\subset D[g_x,\theta,p'_2]$ such that
  $H(\ell_s,s)\subset g(\ell_s,s)-E[\beta^-2',p'_2]$. Then
  \[g(\ell_s,s)\subset g(\ell_1,x)\cap\bigcap_{n'_1<i\leq n_1} D(\theta_i2_i,p^i_2)
    \subset g(\ell_1,x)\cap\bigcap_{n'_1<i\leq n_1} E(\theta_i2_i,p^i_2)
    \subset E_\cap[\theta h,p_h].\]
    Then  $\emptyset\neq\cup\mathscr{G}^*(m,s1'_\ell)\subset g_{c_s}(\ell_s,s)$
    by 1 of Fact 6.11.
    Let $H(m_1,z)\in\mathscr{G}^*(m_1,s1'_\ell)$.
    Then we have $H(m_1,z)\subset g_{c_s}(\ell_s,s)\subset E_\cap[\theta h,p_h]$,
    $g_{c_s}(\ell_s,s)\subset g(\ell_s,s)\subset g(\ell,x)$
    and $H(\ell_s,s)\subset g(\ell_s,s)-E[\beta^-2',p'_2]$.
    This completes a decomposition of $E_\cap[\theta h,p_h]$.\quad $\Box$
\par

\vspace{0.3cm}

\textbf{Decomposition B.4.}  Decompose $U[\beta^- b2',p'_2y]=g_c(b,y)-E[\beta^-2',p'_2]$.
\par
\vspace{0.3cm}

Note $g(\ell_s,s)\subset E_\cap[\theta h,p_h]$ and $U[\beta^-\ell2',p'_2x]
  =E_\cap[\theta h,p_h]\cap\Big[g(\ell,x)-E[\beta^-2',p'_2]\Big]$.
  Let  $U[\beta^-\ell_s 2',p'_2s]=g(\ell_s,s)-E[\beta^-2',p'_2]$.
  Then $H(\ell_s,s)\subset U[\beta^-\ell_s 2',p'_2s]$ and
  \[g(\ell_s,s)\cap U[\beta^- b2',p'_2y]= g(\ell_s,s)\cap U[\beta^-\ell2',p'_2x]
  =g(\ell_s,s)-E[\beta^-2',p'_2].\]
  Then $U[\beta^-\ell_s 2',p'_2s]\subset U[\beta^-\ell2',p'_2x]\subset V_\alpha(p_3)$.
  \quad $\Box$
\par

\textit{The proof of Case 2 is continued.}
  Note $V_\alpha\subset U(\alpha b2,p_2y)=g_c(b,y)-E(\alpha2,p_2)$.
  Then $V_\alpha\cap E(\alpha2,p_2)=\emptyset$. Then $E(\alpha2,p_2)\cap Cl_\tau V_\alpha=\emptyset$ since $E(\alpha2,p_2)$ is a c.o set. Let
  $W_\beta=U[\beta^- b2',p'_2y]$.  Then
  $E(\alpha2,p_2)\subset E[\beta^-2',p'_2]$ and
  \[U[\beta^- b2',p'_2y]\subset W_\beta\subset U(\alpha b2,p_2y).\]
  Take $H(m_1,z)\in\mathscr{H}^*(m_1,s1'_\ell)$
  since $\emptyset\neq\cup\mathscr{G}^*(m,s1'_\ell)\subset g_c(\ell_s,s)$.
  Then, by 3 of Claim 8.2,
  we have
  $\mathscr{H}(m_5\alpha,-\beta\gamma,q_2z)\subset H(m_1,z)$,
  \[H=H(m_5\alpha,-\beta\gamma, q_2z)-H(m_5\sigma^\alpha_1,-\beta\gamma,q_2z)
  =H(m_5\alpha,-\beta\gamma, q_2z)-E[\beta^-2',p'_2].\]
  Note $H(m_5\sigma^\alpha_1,-\beta\gamma,q_2z)\subset E(\alpha2,p_2)$
  by 4 of Claim 8.2.
  Then
  \[H(m_5\sigma^\alpha_1,-\beta\gamma,q_2z)\cap Cl_\tau V_\alpha=\emptyset.\]
  Then $H(m_1,z)\subset E_\cap[\theta h,p_h]$ and
  $H(m_5\alpha,-\beta\gamma, q_2z)\subset H(m_1,z)\subset g_c(\ell_s,s)$.
  Then $H(m_5\alpha,-\beta\gamma, q_2z)-E[\beta^-2',p'_2]
  \subset g(\ell_s,s)-E[\beta^-2',p'_2]=U[\beta^-\ell_s2',p'_2s]$.
  Then \[H\subset U[\beta^-\ell_s2',p'_2s]=E_\cap[\theta h,p_h]\cap\big[g(\ell_s,s)-E[\beta^-2',p'_2]\big]
 \subset U[\beta^-\ell2',p'_2x]\subset V_\alpha.\]
 This completes a proof of Case 2.
 \par

Summarizing the above Case 1 and Case 2,
  we have the following Fact.

\vspace{0.3cm}
\textbf{Fact B.5.}
Let $U[\alpha^- b3',p'_3y]=g_c(b,y)-E[\alpha^-3',p'_3]
 \subset V_\alpha\subset U(\alpha b2,p_2y)$
  with $E(\alpha3,p_3)\subset E[\alpha^-3',p'_3]$  and $\alpha\in\alpha^-$.
  Then the following 1 or 2 holds.
\par

1.  There exists an $H'(m,1_j)\in\mathscr{H}'(m,y_0)$ and a full
  $g(m_1,z)\in\mathscr{G}^*(m_1,1_j)$  such that
  $H(m_5\sigma^\alpha_1,-\beta\gamma,q_3z)\subset H(m_1,z)$,
  $H(m_5\sigma^\alpha_1,-\beta\gamma,q_3z)\cap Cl_\tau V_\alpha=\emptyset,$
  \[H=H(m_5\alpha,-\beta\gamma, q_3z)-H(m_5\sigma^\alpha_1,-\beta\gamma,q_3z)=
  H(m_5\alpha,-\beta\gamma, q_3z)-E(\alpha3,p_3)\quad \hbox{and}\]
  \[H=H(m_5\alpha,-\beta\gamma, q_3z)-E[\beta^-3'_i,p'_3]
  \subset U[\alpha^- b3',p'_3y]\subset V_\alpha.\]
\par

2. If  there exists $U[\beta^-\ell2',p'_2x]
  =E_\cap[\theta h,p_h]\cap g(\ell,x)\cap\big[g_c(b,y)-E[\beta^-2',p'_2]\big]$
  with $U[\beta^-\ell2',p'_2x]\subset V_\alpha\subset U(\alpha b2,p_2y)$ and
  $E(\alpha 2,p_2)=E(\alpha_a 2_a,p^2_a)\subset E[\beta^-2',p'_2]$,
  then there exists $U[\beta^-\ell_s2',p'_2s]=g(\ell_s,s)-E[\beta^-2',p'_2]
  \subset U[\beta^-\ell2',p'_2x]$ and exists $H(m_1,z)\in\mathscr{H}^*(m_1,s1'_\ell)$
  such that
   \[H=H(m_5\alpha,-\beta\gamma, q_2z)-H(m_5\sigma^\alpha_1,-\beta\gamma,q_2z)
  =H(m_5\alpha,-\beta\gamma, q_2z)- E[\beta^-2',p'_2],\]
  \[H(m_5\sigma^\alpha_1,-\beta\gamma,q_2z)\cap Cl_\tau V_\alpha=\emptyset
  \ \hbox{ and }\  H\subset U[\beta^-\ell2',p'_2x]\subset V_\alpha.\]
\vspace{0.3cm}

\textit{The proof of B is continued.}   If $k=3$ and
 $U[\alpha^- b3',p'_3y]\subset V_\alpha\subset U(\alpha b2,p_2y)$,
 then $V_\alpha$ is fine by Fact B.5.
\vspace{0.3cm}

C. Assume, for $k=h=l+1$, we have proved the following fact.
\par

\textbf{Fact B.h.}
Let $U[\alpha^- bh',p'_hy]=g_c(b,y)-E[\alpha^-h',p'_h]\subset V_\alpha\subset U(\alpha b2,p_2y)$
  with $\alpha\in\alpha^-$ and $E(\alpha h,p_h)\subset E[\alpha^-h',p'_h]$.
  Then the following 1 or 2 holds.
\par

1.  There exists an $H'(m,1_j)\in\mathscr{H}'(m,y_0)$ and a full
  $g(m_1,z)\in\mathscr{G}^*(m_1,1_j)$  such that
  $H(m_5\sigma^\alpha_1,-\beta\gamma,q_hz)\subset H(m_1,z)$,
  $H(m_5\sigma^\alpha_1,-\beta\gamma,q_hz)\cap Cl_\tau V_\alpha=\emptyset,$
  \[H=H(m_5\alpha,-\beta\gamma, q_hz)-H(m_5\sigma^\alpha_1,-\beta\gamma,q_hz)
  =H(m_5\alpha,-\beta\gamma, q_hz)-E(\alpha h,p_h)\quad \hbox{and}\]
  \[H=H(m_5\alpha,-\beta\gamma, q_hz)-E[\alpha^- h',p'_h]
  \subset U[\alpha^- bh',p'_hy]\subset V_\alpha.\]
\par

2. If there exists an $l<h$ and exists an $U[\beta^-\ell l',p'_lx]$ such that
  \[U[\beta^-\ell l',p'_lx]=E_\cap[\theta h,p_h]\cap g(\ell,x)
  \cap\big[g_c(b,y)-E[\beta^- l',p'_l]\big]\subset V_\alpha\subset U(\alpha b2,p_2y)\]
  and $E(\alpha l,p_l)\subset E[\beta^- l',p'_l]$,
  then there exists an $U_s\subset U[\beta^-\ell l',p'_lx]$
  such that $H(\ell_s,s)\subset U_s=g(\ell_s,s)-E[\beta^- l',p'_l]$, and exists an
   $H(m_1,z)\in\mathscr{H}^*(m_1,s1'_\ell)$ such that
   \[H=H(m_5\alpha,-\beta\gamma, q_lz)-H(m_5\sigma^\alpha_1,-\beta\gamma,q_lz)
  =H(m_5\alpha,-\beta\gamma, q_lz)-E[\beta^- l',p'_l],\]
  \[H(m_5\sigma^\alpha_1,-\beta\gamma,q_lz)\cap Cl_\tau V_\alpha=\emptyset\
   \hbox{ and }\ H\subset U[\beta^-\ell l',p'_lx]\subset V_\alpha.\]
\vspace{0.3cm}

\textbf{C*.} Let $k=h+1$ and
 $U[\alpha^- bk',p'_ky]=g_c(b,y)-E[\alpha^- k',p'_k]
 \subset V_\alpha\subset U(\alpha b2,p_2y).$
\par

Then $\alpha\in\alpha^-$ and $E(\alpha k,p_k)\subset E[\alpha^- k',p'_k]$
  by Fact 1.
\par

Case 1, there exists an $H'(m,1_j)\in\mathscr{H}'(m,y_0)$ and a full
  $g(m_1,z)\in\mathscr{G}^*(m_1,1_j)$
  such that
  \[H(m_5\sigma^\alpha_1,-\beta\gamma,q_kz)\cap Cl_\tau V_\alpha=\emptyset.\]
  Then, by 3 and 4 of Claim 8.2, we have
  \[H=H(m_5\alpha,-\beta\gamma, q_kz)-H(m_5\sigma^\alpha_1,-\beta\gamma,q_kz)
  =H(m_5\alpha,-\beta\gamma, q_kz)-E(\alpha k,p_k)\quad \hbox{and}\]
  \[H=H(m_5\alpha,-\beta\gamma, q_kz)-E[\alpha^- k',p'_k]
  \subset U[\alpha^- bk',p'_ky]\subset V_\alpha.\]
  Then $V_\alpha$ is fine. Otherwise we have the following case.
\par

Case 2, for every $H'(m,1_j)\in\mathscr{H}'(m,y_0)$ and
  every full $g(m_1,z)\in\mathscr{G}^*(m_1,1_j)$,
  $H(m_5\sigma^\alpha_1,-\beta\gamma,q_kz)\cap Cl_\tau V_\alpha\neq\emptyset.$
\par

Let
 $\delta\in\alpha^-$ with $\delta\neq\alpha$. Then
  $H(m_5\sigma^\alpha_1,-\beta\gamma,q_kz)\cap E(\delta k,p_k)=\emptyset$ by 3 of Claim 8.2.
  Then we can prove that $g[p_k,H(h\sigma^\alpha_1,q_h\sigma)]-E(\alpha h,p_h)$
  is open and dense in $\Big( D(\alpha k,p_k)-E(\alpha h,p_h),\tau\Big)$ for the unique $\alpha\in\alpha^-$.
\par

  In fact, note $H(k\sigma^\alpha_1,q_k\sigma)\cap E(\alpha h,p_h)=\emptyset$.
  Then $H(k\sigma^\alpha_1,q_k\sigma)\subset Y-E(\alpha h,p_h)$.
  Note the definitions of $D(\alpha k,p_k)$ and $E(\alpha h,p_h)$.
  Then \[H(k\sigma^\alpha_1,q_k\sigma)\subset Cl_\tau [D(\alpha k,p_k)-E(\alpha h,p_h)]
  \quad\hbox{ and}\]
  $D(\alpha k,p_k)-E(\alpha h,p_h)=g[p_k,H(h\sigma^\alpha_1,q_h\sigma)]-E(\alpha h,p_h).$
  Then we have
  \[H(k\sigma^\alpha_1,q_k\sigma)\subset Cl_\tau[D(\alpha k,p_k)-E(\alpha h,p_h)]
  =Cl_\tau \big[g[p_k,H(h\sigma^\alpha_1,q_h\sigma)]-E(\alpha h,p_h)\big].\]
  Then $H(m_5\sigma^\alpha_1,-\beta\gamma,q_kz)\cap Cl_\tau \Big[V_\alpha\cap
  \big[g[p_k,H(h\sigma^\alpha_1,q_h\sigma)]-E(\alpha h,p_h)\big]\Big]\neq\emptyset$.
  Let
  \[V_\alpha(p_k)=V_\alpha\cap
  \big[g[p_k,H(h\sigma^\alpha_1,q_h\sigma)]-E(\alpha h,p_h)\big]\quad \hbox{
  and }\quad x\in V_\alpha(p_k).\]
  Then, in the same way as the proof of Fact B.1, we have the following fact.
\par
\vspace{0.3cm}
\textbf{Fact C.1.} There exists an open neighborhood
  $U[\beta^-\ell h',p'_hx]\in\mathscr{U}_x$
  such that
  $U[\beta^-\ell h',p'_hx]\subset V_\alpha(p_k)\subset V_\alpha\subset U(\alpha b2,p_2y)$.

\vspace{0.3cm}

Then, by 1 of Proposition 7.2, we have
  \[E_\cap[\theta h',p'_h]=g(\ell,x)\cap\bigcap_{i\leq n_1} E(\theta_i h_i,p^i_h),
  \quad E[\beta^- e',e'_h]=\cup_{i\leq n_2} E(\beta_ie'_i,e^i_h)\quad \hbox{ and}\]
  \[U[\beta^-\ell h',p'_hx]=E_\cap[\theta h',p'_h]\cap
  g(\ell,x)\cap\big[g_c(b,y)-E[\beta^- e',e'_h]\big]\subset g_c(b,y).\]
  Note $U[\beta^-\ell h',p'_hx]\cap E(\alpha h,p_h)=\emptyset$
  since $U[\beta^-\ell h',p'_hx]\subset V_\alpha(p_k)$.
  Then we have $U[\beta^-\ell h',p'_hx]\subset g_c(b,y)-E(\alpha h,p_h)=U(\alpha bh,p_hy)$.
  Then
  \[U[\beta^-\ell h',p'_hx]=U[\beta^-\ell h',p'_hx]\cap U(\alpha bh,p_hy)
  \quad\hbox{ and}\]
  \[U[\beta^-\ell h',p'_hx]=E_\cap[\theta h',p'_h]\cap
  g(\ell,x)\cap\Big[g_c(b,y)-\big[E[\beta^- e',e'_h]\cup E(\alpha h,p_h)\big]\Big].\]
  Let $E[\beta^- h',e'_h]=E[\beta^- e',e'_h]\cup E(\alpha h,p_h)$
  and $\beta^-=\{\beta_i:i\leq n_2\}$.
\par

 Then we have the following fact in the same way as the proof of Fact B.2.

\par
\vspace{0.3cm}

\textbf{Fact C.2.} $\alpha\in \beta^-$,
  $U[\beta^-\ell h',p'_hx]=E_\cap[\theta h',p'_h]\cap
  g(\ell,x)\cap\big[g_c(b,y)-E[\beta^- h',e'_h]\big]$ and
  $E(\alpha h,p_h)\subset E[\beta^- h',e'_h]$.
\par
\vspace{0.3cm}

Then, in the same way as the proof of Decomposition B.3, there exists
  $g(\ell_s,s)$ and  $\mathscr{H}^*(m_1,s1'_\ell)$
  such that $g(\ell_s,s)\subset g(\ell,x)$ and
  $\emptyset\neq\cup\mathscr{G}^*(m,s1'_\ell)\subset g_c(\ell_s,s)$
  for every $g(m_1,z)\in\mathscr{G}^*(m_1,s1'_\ell)$.
  Then we have the following fact.
\par
\vspace{0.3cm}

\textbf{Fact C.3.}  There exist $g_{c_s}(\ell_s,s)$ and $\mathscr{H}^*(m_1,s1'_\ell)$
  such that $g(\ell_s,s)\subset E_\cap[\theta h',p'_h]$ and
  $H(\ell_s,s)\subset U[\beta^-\ell_sh',e'_hs]=g(\ell_s,s)-E[\beta^- h',e'_h]
  \subset U[\beta^-\ell h',p'_hx]\subset V_\alpha(p_k)$.

\par
\vspace{0.3cm}

\textbf{C**.} Take $\mathscr{G}^*(m_1,h_\sigma)=\cup_j\mathscr{G}^*(m_1,h_j)$
  from Q2 of Construction 6.1. Let
  \[\mathscr{G}^*(m_1,s,h_\sigma)=\{g(m_1,z)\in\mathscr{G}^*(m_1,h_\sigma):
  g(m_1,z)\subset g_{c_s}(\ell_s,s)\}.\]
  Then $\emptyset\neq\mathscr{G}^*(m_1,s1'_\ell)\subset \mathscr{G}^*(m_1,s,h_\sigma)$.
\par

Case 1, there exists  a full $g(m_1,z)\in\mathscr{G}^*(m_1,s,h_\sigma)$
  such that
  \[H(m_5\sigma^\alpha_1,-\beta\gamma,q_hz)\cap Cl_\tau V_\alpha=\emptyset.\]
  Then,  by 3 and 4 of Claim 8.2, we have
  \[H=H(m_5\alpha,-\beta\gamma, q_hz)-H(m_5\sigma^\alpha_1,-\beta\gamma,q_hz)
  =H(m_5\alpha,-\beta\gamma, q_hz)-E(\alpha h,p_h)\quad \hbox{and}\]
  \[H=H(m_5\alpha,-\beta\gamma, q_hz)-E[\alpha^- h',p'_h]
  \subset U[\beta^-\ell_se,p_es]\subset V_\alpha.\]

  Then $V_\alpha$ is fine.
\par

If Case 1 is not true, then we have the following case.

Case 2, for every $g(m_1,z)\in\mathscr{G}^*(m_1,s,h_\sigma)$,
  $H(m_5\sigma^\alpha_1,-\beta\gamma,q_hz)\cap Cl_\tau V_\alpha\neq\emptyset.$
\par

  Let $\delta\in \alpha^-$ with $\delta\neq\alpha$. Then
  $H(m_5\sigma^\alpha_1,-\beta\gamma,q_hz)\cap E(\delta h,p_h)=\emptyset$ by 3 of Claim 8.2.
  Then $E(\alpha h,p_h)$ is the unique open set with $H(m_5\sigma^\alpha_1,-\beta\gamma,q_hz)\subset E(\alpha h,p_h)$
  and $\alpha\in\alpha^-$. Note $h=l+1$ by inductive assumption C. Take
  \[D(\alpha h,p_h)=g[p_2,D(1\sigma^\alpha_1,G_1)]\cup
  g[p_3,H(2\sigma^\alpha_1,q_2\sigma)]\cup...\cup g[p_h,H(l\sigma^\alpha_1,q_l\sigma)].\]
  Then we have $g[p_h,H(l\sigma^\alpha_1,q_l\sigma)]\subset E(\alpha h,p_h)$
  by the definition of $E(\alpha h,p_h)$.
  Note $g[p_h,H(l\sigma^\alpha_1,q_l\sigma)]-E(\alpha l,p_l)$ is open and
  dense in $E(\alpha h,p_h)-E(\alpha l,p_l)$.
  Then
  \[H(m_5\sigma^\alpha_1,-\beta\gamma,q_hz)\cap Cl_\tau \Big[V_\alpha\cap\big[
  g[p_h,H(l\sigma^\alpha_1,q_l\sigma)]-E(\alpha l,p_l)\big]\Big]\neq\emptyset.\]
  Let $V_\alpha(p_h)=V_\alpha\cap\big[
  g[p_h,H(l\sigma^\alpha_1,q_l\sigma)]-E(\alpha l,p_l)\big]$
  and $x\in V_\alpha(p_h)$.
  Then there exists an open neighborhood
  $U[\gamma^-\ell l',p'_lx]\in\mathscr{U}_x$
  such that $U[\gamma^-\ell l',p'_lx]\subset V_\alpha(p_h)$.
  Then
  $U[\gamma^-\ell l',p'_lx]\subset V_\alpha(p_h)\subset V_\alpha$.
  Then we have the following fact.
\par
\vspace{0.3cm}

\textbf{Fact C.4.} There exists an open neighborhood
  $U[\gamma^-\ell l',p'_lx]\in\mathscr{U}_x$
  such that
  $U[\gamma^-\ell l',p'_lx]\subset V_\alpha(p_h)\subset V_\alpha
  \subset U(\alpha b2,p_2y)$.
\vspace{0.3cm}

Then, in the same way as the proof of Fact B.2 and C.2, we have the following fact.
\par
\vspace{0.3cm}

\textbf{Fact C.5.} $\alpha\in \gamma^-$, $U[\gamma^-\ell l',p'_lx]=
  E_\cap[\theta l',p'_l]\cap g(\ell,x)\cap\big[g_c(b,y)-E[\gamma^- l',p'_l]\big]$ and
  $E(\alpha l,p_l)\subset E[\gamma^- l',p'_l]$.
\par

Then Fact C.5 satisfies the assumption of 2 of Fact B.h.
Then $V_\alpha$ is fine by the inductive assumption Fact B.h.
\par

Then, by inductive proof A, B and C, we complete a proof of the proposition.
\end{proof}

\begin{prop}
 Let $U[\alpha^- b k',p'_ky]\subset V_\alpha\subset U(\alpha b 2,p_2y)$ and
 $g(m_1,z)\in\mathscr{G}^*(m_1,1_\ell)$ such that\
 $H(m_5\alpha,-\beta\gamma,q_kz)\ -
  H(m_5\sigma^\alpha_1,-\beta\gamma,q_kz)\ \subset\ U[\alpha^- b k',p'_ky]\subset V_\alpha$
  and
 $H(m_5\sigma^\alpha_1,-\beta\gamma,q_kz)\cap Cl_\tau V_\alpha=\emptyset.$
Then, for every
   $H(v^\alpha_1,t_1)\in\mathscr{H}(m_5\sigma^\alpha_1,-\beta\gamma,q_kz)$ and
   every
   $H(v^\alpha_1,t_2)\in
   \mathscr{H}_{-1}(k\sigma^\alpha_1,jz,q_k\sigma)|H(v^\alpha_0,t_1)$
   with
   \[H(v^\alpha_1,t_2)\subset H(m_5\alpha,-\beta\gamma,q_kz)\ -
   H(m_5\sigma^\alpha_1,-\beta\gamma,q_kz),\]
   $H(v^\alpha_1,t_1)$ and $H(v^\alpha_1,t_2)$ satisfy
   \[H(v^\alpha_1,t_1)\cap Cl_\tau V_\alpha=\emptyset,\quad
   H(v^\alpha_1,t_2)\subset V_\alpha\quad\hbox{ and}\]
   \[H(v^\alpha_1,t_1)\cup H(v^\alpha_1,t_2)\subset H(v^\alpha_0,t_1)
   \in\mathscr{H}(v^\alpha_0,jz^k_v).\]
\end{prop}

\begin{proof}
Let $U[\alpha^- b k',p'_ky]\subset V_\alpha\subset U(\alpha b 2,p_2y)$.
  Then, by Proposition 8.3, $V_\alpha$ is fine.
  Then there exists an $H(m_5\sigma^\alpha_1,-\beta\gamma,q_kz)$ with
  $H(m_5\sigma^\alpha_1,-\beta\gamma,q_kz)\cap Cl_\tau V_\alpha=\emptyset$
   and
   \[H(m_5\alpha,-\beta\gamma,q_kz)\ -
  H(m_5\sigma^\alpha_1,-\beta\gamma,q_kz)\ \subset\ U[\alpha^- b k',p'_ky]\subset V_\alpha.\]
  Take an $H(v^\alpha_1,t_1)\in\mathscr{H}(m_5\sigma^\alpha_1,-\beta\gamma,q_kz)$,
  and an $H(v^\alpha_1,t_2)\in\mathscr{H}_{-1}(v^\alpha_1,jz^k_v)|$ with
  \[H(v^\alpha_1,t_2)\subset H(m_5\alpha,-\beta\gamma,q_kz)\ -
  H(m_5\sigma^\alpha_1,-\beta\gamma,q_kz)\]
  and $H(v^\alpha_1,t_1)\cup H(v^\alpha_1,t_2)\subset H(v^\alpha_0,t_1)$
  by the definition of $\mathscr{H}_d(v^\alpha_1,jz^k_v)$.
  Then $H(v^\alpha_0,t_1)\in\mathscr{H}(v^\alpha_0,jz^k_v)$.
  Then $H(v^\alpha_1,t_1)\cap Cl_\tau V_\alpha=\emptyset$
  and $H(v^\alpha_1,t_2)\subset V_\alpha$.
\end{proof}

\section{stratifiable space $(X,\tau)$ in not $M_1$-spaces}

Recall $\mathscr{A}$  a \textit{ mad } family on $N$.
  $\alpha, \beta, \delta$ and $\gamma$ are used to
  denote members in $\mathscr{A}$.

\begin{prop}
Let $\mathscr{A}=\cup_i\mathscr{A}_i$.
  Then there exists an $\mathscr{A}_i$
  such that $\mathscr{A}_i$ is unbounded.
\end{prop}

\begin{proof}
Suppose that each $\mathscr{A}_i$ is bounded.
 Then $\mathscr{A}$ is bounded, a contradiction.
\end{proof}

  Take a $\alpha\in \mathscr{A}$.
  Let $k> 1$. Then there exists a family
  $\mathscr{E}(\alpha,k)=\{E(\alpha k,p_k): p_k>m_5\}$ by 1 of Definition B.
  Let $\mathscr{E}(\alpha)=\cup_k\mathscr{E}(\alpha,k)$.
\par

Take $\mathscr{V}_y=\{V_\lambda: \lambda\in \Lambda\}$ and
  $\mathscr{U}(y,2,c)$ from Proposition 8.3.
  Let $b>m_5$, $y\in H_0$ and $p_2>m_5$. Note
  \[\mathscr{U}(y,2,c)=\{U(\alpha b2,p_2y)=g_c(b,y)-E(\alpha2,p_2)\in\mathscr{U}(y):
  \alpha\in \mathscr{A}\}.\]
  Take $\mathscr{U}_y$.
  Then, for every $U(\alpha b2,p_2y)\in\mathscr{U}(y,2,c)$, there exists an
  $U[\alpha^- b_\alpha k,p_ky]\in\mathscr{U}_y$
 and a $V_\alpha\in\mathscr{V}_y$ such that
 \[U[\alpha^- b_\alpha k',p'_ky]=\cap_{i\leq n}U(\alpha_i b_i k_i,p_{k_i}y)
  \subset V_\alpha\subset U(\alpha b2,p_2y)\subset g_c(b_\alpha,y)\]
 with $\alpha\in\alpha^-=\{\alpha_i: i\leq n\}$ by Proposition 8.1.
 Then, by Proposition 8.3, $V_\alpha$ is fine.
 Then there exists a full $g(m_1,z)$
 such that
 \[H(m_5\alpha,-\beta\gamma,q_kz)\ -
  H(m_5\sigma^\alpha_1,-\beta\gamma,q_kz)\ \subset\
  U[\alpha^- b_\alpha k',p'_ky]\subset V_\alpha\subset g_c(b_\alpha,y)\quad\hbox{ and}\]
   \[H(m_5\sigma^\alpha_1,-\beta\gamma,q_kz)\cap Cl_\tau V_\alpha=\emptyset.\]
 Denote $H(m_1,z)$ by $H(m_1,z\alpha)$.
 Let $\mathscr{H}(y,\mathscr{A})=\{H(m_1,z\alpha):
  \alpha\in \mathscr{A}\}$ and
  \[\mathscr{U}[y,\mathscr{A}]=\{U[\alpha^- b_\alpha k',p'_ky]:
  U[\alpha^- b_\alpha k',p'_ky]\subset V_\alpha\subset U(\alpha b2,p_2y)
  \hbox{ and }\alpha\in \mathscr{A}\}.\]
  Note $H(m_1,z\alpha)\in \mathscr{H}_Y$ is countable.
  Then there exists  an $H(m_1,z\alpha)=H(m_1,z)$
  and a unbounded subfamily $\mathscr{A}_1\subset\mathscr{A}$
  such that
  \[\mathscr{U}[y,\mathscr{A}_1]=\{U[\alpha^- b_\alpha k',p'_ky]:
  H(m_1,z\alpha)=H(m_1,z)
  \hbox{ and }\alpha\in \mathscr{A}_1\}.\]
  Take an
  $U[\alpha^- b_\alpha k',p'_ky]=\cap_{i\leq n}U(\alpha_i b_i k_i,p_{k_i}y)\in\mathscr{U}[y,\mathscr{A}_1]$.
  Then $\alpha\in\alpha^-$.
  Then there exists an unbounded subfamily $\mathscr{A}_2\subset\mathscr{A}_1$
  such that $|\alpha^-|=|\beta^-|=n'$, $k_\alpha=k_\beta=k$ and
  $p_{k_\alpha}=p_{k_\beta}=p_k$ if $\alpha,\beta\in \mathscr{A}_2$.
  Let \[\mathscr{U}[y,\mathscr{A}_2]=\{U[\alpha^- b_\alpha k',p'_ky]:
  |\alpha^-|=n', k_\alpha=k, p_{k_\alpha}=p_k
  \hbox{ and }\alpha\in\mathscr{A}_2\}.\]
  Take the relative  $\mathscr{H}(y,\mathscr{A}_2)=\{H(m_1,z\alpha)=H(m_1,z):
  \alpha\in \mathscr{A}_2\}$.
  Then $\mathscr{H}(y,\mathscr{A}_2)=\{H(m_1,z)\}$.

Call $\mathscr{U}[y,\mathscr{A}_2]$ \textbf{an idea family.}
\par

Summing up to the above result  we have the following proposition.

\begin{prop}
Let $V_\alpha$ be fine for every $\alpha\in\mathscr{A}$.
  Then there exists an idea family $\mathscr{U}[y,\mathscr{A}_2]$
  with the same full set $H(m_1,z)$.
\end{prop}

 Let $[n_1,n_2,...,n_h]
  =\{\alpha\in\mathscr{A}:
  \alpha(i)=n_i \hbox{ for } i=1,2,...,h\}$.
  Then, for arbitrary $h\in N$, \[\mathscr{A}=\cup\{[n_1,n_2,...,n_h]:
  n_1,n_2,...,n_h\in N \hbox{ with }
  n_1<n_2<...<n_h\}.\]
  Let $\mathscr{A}'\subset\mathscr{A}$.
  Fix an $h\in N$ and let
  $\mathcal{A}'_h=\{[n_1,n_2,...,n_h]':
  [n_1,n_2,...,n_h]'\subset\mathscr{A}'\}$.
  Then $\mathscr{A}'=\cup\mathcal{A}'_h$.

\begin{prop}
If there exists $b\in N$ such that, for
 every $h>b$ and every
  $[n_1,n_2,...,n_h]'\in\mathcal{A}'_h$,
  there exists a $k\in N$ with $\alpha(h+1)<k$
  for every $\alpha\in[n_1,n_2,...,n_h]'$,
  then $\mathscr{A}'$ is bounded.
\end{prop}

\begin{proof}
1. Note $\mathcal{A}'_b$ is countable.
 Take a $[1_i,2_i,...,b_i]'$ from $\mathcal{A}'_b$.
 Then there exists a $k_1$ such that $\alpha(b+1)=2'_j<k_1$
  for each $\alpha\in[1_i,...,b_i]'$. Let $\beta_i(b+1)=k_1$.
  Let $\mathcal{A}'_{b+1}[1_i,...,b_i]=\{[1_i,...,b_i,2_j]'
  \in\mathcal{A}'_{b+1}:
  [1_i,...,b_i,2_j]'\subset[1_i,...,b_i]'\}$.
  Then $\mathcal{A}'_{b+1}[1_i,...,b_i]$ is finite.
   Take a $[1_i,...,b_i,2_j]'$ from $\mathcal{A}'_{b+1}[1_i,...,b_i]$.
  Then there exists a $k_{2j}$ such that $\alpha(b+2)<k_{2j}$
  for each $\alpha\in[1_i,...,b_i,2_j]'$.
  Let $\beta_i(b+2)=k_2=\hbox{max}\{k_{2j}: j\leq k_1\}$.
\par

2. Assume that $\mathcal{A}'_k[1_i,...,k_l]$ is finite
 and $\beta_i(b+k)=k_k$.
 Denote $k+1$ by $n$ and let
 $\mathcal{A}'_n[1_i,...,k_l]=\{[1_i,...,k_l,n_j]'\in\mathcal{A}'_k:
 [1_i,...,k_l,n_j]'\subset[1_i,...,k_l]'
 \in\mathcal{A}'_k[1_i,...,k_l]\}$.
 Take a $[1_i,...,k_l,n_j]'$ from $\mathcal{A}'_n[1_i,...,k_l]$.
 Then there exists a $k_{nj}$ such that $\alpha(n+1)<k_{nj}$
  for each $\alpha\in[1_i,...,k_l,n_j]'$.
  Note that $\mathcal{A}'_n[1_i,...,k_l]$ is finite.
  Let $\beta_i(b+n)=k_{k+1}=\hbox{max}\{k_{nj}: j\leq k_k\}$.
\par

Then, by induction, we have a $\beta_i$ for
  every $[1_i,2_i,...,b_i]\in\mathcal{A}'_b$
  such that $\alpha(n)\leq\beta_i(n)$
  if $n>1$ for every $\alpha\in[1_i,2_i,...,b_i]$.
\par

Let $\beta(n)=\beta_1(n)+...+\beta_n(n)$.
    Then $\beta$ is a bounded on $\mathscr{A}'$.
\end{proof}

\begin{cor}
Let $\mathscr{A}'_2$ be unbounded. Then, for each $b\in N$,
 there exists an $h>b$ and a $[n_1,n_2,...,n_h]\in\mathcal{A}'_2$
 such that for each $k\in N$,  $\alpha(h+1)\geq k$
 for some $\alpha\in[n_1,n_2,...,n_h]$.
\end{cor}
\begin{proof}
This is an inverse no proposition of Proposition 9.3.
\end{proof}

\begin{prop}
$\mathscr{V}_y=\{V_\lambda: \lambda\in \Lambda\}$
  is not closure preserving.
\end{prop}

\begin{proof}
Note $v^\alpha_1=\alpha(m+v)+1$ and $H(m_3,jz^k_v)\subset H(m_1,jz)$
  for $\alpha\in\mathscr{A}_2$ with $H(m_1,z)
  =H(m_1, z\alpha)\subset H(m_1,jz)$.
  Then there exists a $q_k$
 such that, for every $\alpha\in\mathscr{A}_2$, we have
 \[H(m_5\alpha,-\beta\gamma,q_kz)\ -
  H(m_5\sigma^\alpha_1,-\beta\gamma,q_kz)\ \subset\
  U[\alpha^- b k',p'_ky]\subset V_\alpha\subset g_c(b,y)\quad\hbox{ and}\]
   \[H(m_5\sigma^\alpha_1,-\beta\gamma,q_kz)\cap Cl_\tau V_\alpha=\emptyset.\]
\par

Let $[n_1,n_2,...,n_v]'\subset\mathscr{A}_2$
  satisfy Corollary 9.4.
  Then, for arbitrary $l\in N$, there exists an $\alpha_h\in [n_1,n_2,...,n_v]'$
  with $\alpha_h(v+1)\geq l$.
  Let $\alpha^*=\{\alpha_h:h\in N\}$.
  Then we may assume $\alpha_1(m+v)<\alpha_2(m+v)<...<\alpha_h(m+v)<...$.
\par

\textbf{A.} Take $\alpha_1\in\alpha^*$ for $p=1$. Denote $\alpha_1$ by $\alpha$.
  Take relative $U[\alpha^- b_\alpha k',p'_ky]$, $V_\alpha$ and $U(\alpha b2,p_2y)$.
  Then $U[\alpha^- b_\alpha k',p'_ky]=g(b,y)-\cup_{i\leq n}E(\alpha_i k_i,p_{k_i})
  \subset V_\alpha\subset U(\alpha b2,p_2y)$.
  Let $\alpha^-_1=\{\beta_i: i\leq n'\}$.
  Then $\alpha_1\in\alpha^-_1$.
  Then, by Proposition 8.4, for every  $H(v^\alpha_1,t_1)\in\mathscr{H}(m_5\sigma^\alpha_1,-\beta\gamma,q_kz)$
  and every $H(v^\alpha_1,t_2)\in
  \mathscr{H}_{-1}(k\sigma^\alpha_1,jz,q_k\sigma)|H(v^\alpha_0,t_1)$ with
  $H(v^\alpha_1,t_2)\subset
  H(m_5\alpha,-\beta\gamma,q_kz)-H(m_5\sigma^\alpha_1,-\beta\gamma,q'_kz),$
  $H(v^\alpha_1,t_1)$ and $H(v^\alpha_1,t_2)$
  satisfy $H(v^\alpha_1,t_1)\cap Cl_\tau V_\alpha=\emptyset$,
  $H(v^\alpha_1,t_2)\subset V_\alpha$
  and \[H(v^\alpha_1,t_1)\cup H(v^\alpha_1,t_2)\subset H(v^\alpha_0,t_1)
  \in\mathscr{H}(v^\alpha_0,jz^k_v).\]
  Let $\mathscr{H}(v, \alpha^-_1)=\mathscr{H}(m_5\sigma^\alpha_1,-\beta\gamma,q_kz)|
  H(m_5,jz^k_v)$
  and $H(v, \alpha^-_1)=\cup\mathscr{H}(v, \alpha^-_1)$.
\par

\textbf{B.} Let $p=n$ and $\delta=\alpha_n$. Take $\mathscr{H}(m_5\delta,-\beta\gamma,z^k_v)$
  in 3 of proof of Proposition 8.1. Let
  $H(m_5\delta,-\beta\gamma,z^k_v)=\cup\mathscr{H}(m_5\delta,-\beta\gamma,z^k_v)$.
  Assume we have had
  \[\mathscr{H}(v, \alpha^-_n)=\mathscr{H}(m_5\sigma^\delta_1,-\beta\gamma,q_kz)|
  H(m_5\delta,-\beta\gamma,z^k_v)\]
  and $H(v,\alpha^-_n)=\cup\mathscr{H}(v,\alpha^-_n)$.
  Then we have $\cap_{i\leq n}H(v,\alpha^-_i)\neq\emptyset$
  by the definition of $\mathscr{H}(m_5\delta,-\beta\gamma,z^k_v)$.
  Let
  \[\mathscr{H}(v, \alpha^-_n,\wedge)=\bigwedge_{i\leq n}\mathscr{H}(v, \alpha^-_i)
  =\{\cap_{i\leq n}H(v^{\alpha_i}_1,t^i_1):
  H(v^{\alpha_i}_1,t^i_1)\in\mathscr{H}(v, \alpha^-_i)
  \hbox{ for }i\leq n\}.\]
  Then $\mathscr{H}(v, \alpha^-_n,\wedge)\subset \mathscr{H}(v, \alpha^-_n)$ and
  $\cup\mathscr{H}(v, \alpha^-_n,\wedge)=\cap_{i\leq n}H(v,\alpha^-_i)$.
\par

Take an $H(v^\delta_1,t^n_1)\in\mathscr{H}(v, \alpha^-_n,\wedge)$.
  Then, by Corollary 4.9, there exists an $\ell$, exists a Ln cover
  $\mathscr{G}(\ell,jz^k_v)$ of $H(m_3,jz^k_v)$ and exists a
  $\mathscr{G}(\ell,t^n_1)\subset\mathscr{G}(\ell,jz^k_v)$
  such that $\cup\mathscr{H}(\ell,t^n_1)
  \subset H(v^\delta_1,t^n_1)\subset\cup\mathscr{G}(\ell,t^n_1)$.
  Then there exists an $\alpha_h\in\alpha^*$ with
  $\alpha_h(m+v)>\ell$. Denote $\alpha_h$ by $\theta$.
  Note that \[\mathcal{G}(v^\theta_0,jz^k_v)
  =\{\mathscr{G}(v^\theta_0,jz^k_v),
  \mathscr{G}(v^\theta_1,jz^k_v)\}\]
  is a covers sequence of $H(m_3,jz^n_v)$.
  Then there exists an subfamily
  $\mathscr{G}(v^\theta_1,t^n_1)\subset\mathscr{G}(v^\theta_1,jz^k_v)$
  such that  $\cup\mathscr{H}(v^\theta_1,t^n_1)
  \subset H(v^\delta_1,t^n_1)\subset\cup\mathscr{G}(v^\theta_1,t^n_1).$
\par

On the other hand,  note $\theta=\alpha_h\in\alpha^*\subset[n_1,n_2,...,n_v]'\subset\mathscr{A}_2$.
  Then there exist $U[\theta^- bk',p'_ky]$, $V_\theta$ and $U(\theta b2,p_2y)$ such that
 $U[\theta^- bk',p'_ky]\subset V_\theta\subset U(\theta b2,p_2y)$ and
 $H(m_1,z)\in\mathscr{H}(m_1,1_\ell)$ satisfying
 \[H(m_5\theta,-\beta\gamma, q'_kz)-H(m_5\sigma^\theta_1,-\beta\gamma,q'_kz)
 \subset U[\theta^- bk',p'_ky]\subset V_\theta\]
  and
 $H(m_5\sigma^\theta_1,-\beta\gamma,q'_kz)\cap Cl_\tau V_\theta=\emptyset$.
 Let $h=n+1$.
  Then, by Proposition 8.4, for every
  $H(v^\theta_1,t^h_1)\in\mathscr{H}(m_5\sigma^\theta_1,-\beta\gamma,q'_kz)$
  and every $H(v^\theta_1,t^h_2)\in\mathscr{H}_{-1}(k\sigma^\theta_1,jz,q_k\sigma)$
  with $H(v^\theta_1,t^h_2)\subset H(v^\theta_0,t^h_1)$,
  we have \[H(v^\theta_1,t^h_2)\subset H(m_5\theta,-\beta\gamma, q'_kz)-H(m_5\sigma^\theta_1,-\beta\gamma,q'_kz),\]
  \[H(v^\theta_1,t^h_1)\cap Cl_\tau V_\theta=\emptyset,\quad
  H(v^\theta_1,t^h_2)\subset V_\theta \quad \hbox{and}\quad
  H(v^\theta_1,t^h_1)\cup H(v^\theta_1,t^h_2)\subset H(v^\theta_0,t^h_1).\]
  Then, by induction for every $\alpha_h$, we denote $\alpha_h$ by $\theta$. Then
  $H(v^\theta_1,t^h_1)\cap Cl_\tau V_\theta=\emptyset$,
  $H(v^\theta_1,t^h_2)\subset V_\theta$
  and $H(v^\theta_1,t^h_1)\cup H(v^\theta_1,t^h_2)\subset H(v^\theta_0,t^h_1).$
   Let
  \[\mathscr{H}(v, \alpha^-_h)=\mathscr{H}(m_5\sigma^\theta_1,-\beta\gamma,q_kz)|
  H(m_5\theta,-\beta\gamma,z^k_v)\hbox{ and }\]
  \[\mathscr{H}(v, \alpha^-_h,\wedge)=\bigwedge_{i\leq h}\mathscr{H}(v, \alpha^-_i).\]
\par

\textbf{C.} Note $v^\alpha_1=\alpha(m+v)+1$,
  $\alpha_1(m+v)<\alpha_2(m+v)<...<\alpha_h(m+v)<...$ and
  $\cup\mathscr{H}(v^\theta_1,t^n_1)
  \subset H(v^\delta_1,t^n_1)\subset\cup\mathscr{G}(v^\theta_1,t^n_1)$
  for $\delta=\alpha_n$ and $\theta=\alpha_{n+1}$.
  Then, by 3 of Proposition 3.5, we have $\cap_hH(v^{\alpha_h}_0,t^h_1)=\cap_hH(v^{\alpha_h}_1,t^h_1)=\{t\}\neq\emptyset$.
  Note $t\in H(v^{\alpha_h}_1,t^h_1)$ and
  $H(v^{\alpha_h}_1,t^h_1)\cap Cl_\tau V_{\alpha_h}=\emptyset$.
  Then $t\notin Cl_\tau V_{\alpha_h}$ for every $h\in N$.
\par

On the other hand, we have $t\in H(v^{\alpha_h}_0,t^h_1)\subset g(v^{\alpha_h}_0,t^h_1)$.
  Then, by Note D2,
  \[t^h_2\in H(v^{\alpha_h}_1,t^h_2)\subset H(v^{\alpha_h}_0,t^h_1)
  \subset g(v^{\alpha_h}_0,t^h_1)=g(v^{\alpha_h}_0,t^h_2)\quad
  \hbox{ and }\quad H(v^{\alpha_h}_1,t^h_2)\subset V_{\alpha_h}.\]
  Then $t^h_2\rightarrow t$ if $h\rightarrow +\infty$ by 2 in the proof of Theorem 1.a.
  Then we have $t\in Cl_\tau[\cup_hV_{\alpha_h}]$.
  Then $\mathscr{V}_y=\{V_\lambda: \lambda\in \Lambda\}$
  is not closure preserving.
\end{proof}

\vspace{0.3cm}
\textbf{Theorem 1.} b. $(Y,\tau)$ is not an $M_1$-space.

\begin{proof}
Suppose that $\mathscr{B}=\cup_i\mathscr{B}_i$ is a $\sigma$-closure preserving base
 of $(Y,\tau)$.
  Let $\mathscr{B}(i,y)=\{B\cap g(i,y): y\in B\in\mathscr{B}_i\}$.
  Then $\mathscr{B}(y)=\cup_i\mathscr{B}(i,y)$ is a closure preserving outer base
  of $y$ in $(Y,\tau)$.
  It is a contradiction to Proposition 9.5.
\end{proof}

\section{\bf Problems.}

Theorem 1  suggest the following problems.
\par

\textbf{Problem 1.}  Simplify space $(X,\tau)$
  or look for a simpler counterexample.
\par

Note stratifiable $\mu$-spaces
 have various dimension theoretical properties and nice preservation properties
 under topological operations.
 But, for hereditarily $M_1$-spaces, we have only a simple definition.
 So we venture to hope the following problem having a positive answer after the counterexample.
\par

\textbf{Problem 2.}  Are hereditarily $M_1$-spaces stratifiable
  $\mu$-spaces?
\par

Liu \cite{Liu} has given a characterization of  $M_1$-spaces
  by $g$-functions. We don't know how to construct a
  $\sigma$-closure preserving base for $M_1$-spaces.
\par

\textbf{Problem 3.} Construct a $\sigma$-closure preserving base for $M_1$-spaces.
\par

  Recall a well-known theorem in \cite{gr}:
  A space $X$ is stratifiable if and only if $X$ is
  semi-stratifiable and monotonically normal.
  So we have the following problem 4.
 \par
\textbf{Problem 4.} Character differences between stratifiable
  $\mu$-spaces and stratifiable spaces.
\par

The following problem 4 is more difficult than Problem 4.

\textbf{Problem 5.} Character differences between
   $M_1$-spaces and stratifiable spaces.
\par

Junnila \cite{j1} proved that every (first countable)
 topological space is the continuous image of a stratifiable (metrizable)
 $\sigma$-discrete $T_1$-space under an open mapping.
 Reference Lin \cite{L2}. The following problem
 should be difficult.

\textbf{Problem 6.} Character open images of $M_1$-spaces.
\par

\section{\bf Acknowledgement.}

Professor Gary Gruenhage went through the primitive paper, and told us
that some sections are too complicated and too hard to understand.
So we are going to rewriting the paper from Section 5.

\end{document}